\documentclass[11pt]{amsart}
\usepackage{amsthm,amsfonts,amssymb,amsmath,oldgerm}
\numberwithin{equation}{section}
\usepackage{fullpage}
\usepackage{mathtools}
\usepackage{color}
\usepackage{amsfonts}
\usepackage{calrsfs}
\usepackage{graphicx}

\usepackage{tikz}
\usepackage{tikz-cd}
\usepackage{amsmath}
\usepackage{amssymb}

\DeclareMathAlphabet{\pazocal}{OMS}{zplm}{m}{n}

\def\eps{\varepsilon }

\newcommand\R{\mathbb R}
\newcommand\C{\mathbb C}

\def\eps{\varepsilon}

\newcommand\br{\begin{remark}}
	\newcommand\er{\end{remark}}
\newcommand\bp{\begin{pmatrix}}
	\newcommand\ep{\end{pmatrix}}
\newcommand{\be}{\begin{equation}}
	\newcommand{\ee}{\end{equation}}
\newcommand\ba{\begin{equation}\begin{aligned}}
		\newcommand\ea{\end{aligned}\end{equation}}

\newcommand{\bap}{\begin{app}}
	\newcommand{\eap}{\end{app}}
\newcommand{\begs}{\begin{exams}}
	\newcommand{\eegs}{\end{exams}}
\newcommand{\beg}{\begin{example}}
	\newcommand{\eeg}{\end{exaplem}}
\newcommand{\bpr}{\begin{proposition}}
	\newcommand{\epr}{\end{proposition}}
\newcommand{\bt}{\begin{theorem}}
	\newcommand{\et}{\end{theorem}}
\newcommand{\bc}{\begin{corollary}}
	\newcommand{\ec}{\end{corollary}}
\newcommand{\bl}{\begin{lemma}}
	\newcommand{\el}{\end{lemma}}
\newcommand{\bd}{\begin{definition}}
	\newcommand{\ed}{\end{definition}}
\newcommand{\brs}{\begin{remarks}}
	\newcommand{\ers}{\end{remarks}}

\newcommand{\SAB}{S_{A,B}}
\newcommand{\CAB}{C_{A,B}}

\newcommand{\A }{\mathcal{A}}

\newcommand{\RR}{{\mathbb R}}

\newcommand{\NN}{{\mathbb N}}

\newcommand{\CC}{{\mathbb C}}

\newcommand{\tu}{{\widetilde{u}}}

\newcommand{\rmd}{{\mathrm{d}}}

\newcommand{\Ooor}{{\Omega_*\setminus\big(\{\mu_k:k=1,\cdots,m\}\cup\{\nu_\ell:\ell=1,\dots,p\}\big)}}

\newcommand{\Ker}{{\mathrm{Ker}}}

\newcommand{\rmb}{{\mathrm{b}}}

\newcommand{\Range}{{\rm Range }}

\newtheorem{theorem}{Theorem}[section]
\newtheorem{proposition}[theorem]{Proposition}
\newtheorem{corollary}[theorem]{Corollary}
\newtheorem{lemma}[theorem]{Lemma}

\theoremstyle{remark}
\newtheorem{remark}[theorem]{Remark}
\theoremstyle{definition}
\newtheorem{definition}[theorem]{Definition}

\newtheorem{example}[theorem]{Example}

\newcommand\cB{{\mathcal B}}

\newcommand\cC{{\mathcal C}}

\newcommand\cL{{\mathcal L}}

\newcommand\cT{{\mathcal T}}

\newcommand\bV{{\mathbb V}}
\newcommand\bW{{\mathbb W}}

\newcommand\bX{{\mathbb X}}
\newcommand\bY{{\mathbb Y}}

\newcommand{\tnu}{{\widetilde{\nu}}}

\newcommand{\dom}{\text{\rm{dom}}}

\newcommand{\beq}{\begin{equation}}
	\newcommand{\eeq}{\end{equation}}

\allowdisplaybreaks

\title{The Damped Waves Equation and generalized Cosine and Sine families on Banach spaces}

\author{ Yuri Latushkin}
\address{University of Missouri, Columbia, MO 65211}
\email{latushkiny@missouri.edu}
\thanks{\hspace*{-0.17in}\textbf{AMS MSC 2020 Mathematics Subject Classification:} 37L15, 47D06, 34G10\\\textbf{Keywords:} Damped waves Equations, Generalized Cosine and Sine families, integral representations. \\Y.L. was supported by the NSF grants DMS-2106157 and would like to thank the Courant Institute of Mathematical Sciences and especially Prof.\ Lai-Sang Young for the opportunity to visit CIMS.\\A. P. research  was partially supported under the Simons Foundation Grant nr. 524928.}
\author{Alin Pogan}
\address{Miami University, Oxford, OH 45056}
\email{pogana@miamioh.edu}
\begin{document}

\begin{abstract}
We study the abstract damped wave equation on a Banach space, allowing the damping coefficient to be unbounded. By recasting the equation as a first-order system and identifying conditions under which the associated block operator generates a $C_0$-group, we construct generalized cosine and sine families that represent mild and classical solutions, extending the classical undamped theory. We establish existence, uniqueness, regularity, invariant subspaces, growth rate, and trigonometric type identities for these families. Our setup applies to a broad class of damped wave, Klein--Gordon, and higher-order PDE examples, including cases where damping restores well-posedness that fails in the undamped equation. 
\end{abstract}

	\maketitle

	\vspace{0.3cm}
	\begin{minipage}[h]{0.48\textwidth}
		\begin{center}
			University of Missouri \\
			Department of Mathematics\\
			810 East Rollins Street\\ Columbia, MO 65211, USA
		\end{center}
	\end{minipage}
	\begin{minipage}[h]{0.48\textwidth}
		\begin{center}
			Miami University\\
			Department of Mathematics\\
			100 Bishop Circle\\
			Oxford, OH 45056, USA
		\end{center}
	\end{minipage}
	
	\vspace{0.3cm}

	
\section{Introduction }\label{s1}
In this paper we study the second order damped waves equation 
\begin{equation}\label{AB}
	\begin{cases}
		u''(t) + 2Bu'(t) = Au(t),\;t\in\RR,\\
		u(0) = x,\; u'(0) = y
	\end{cases}
\end{equation}
on a Banach space $\bX$. We are interested in finding large classes of linear unbounded operators $A$ and $B$ for which the initial value problem \eqref{AB} has mild solutions and study their properties. An interesting example we have in mind is 
\begin{equation}\label{AB-leitmotif}
	\begin{cases}
		\partial_t^2u + 2\partial_t\partial_\xi u+2b(x)\partial_tu = -\partial_\xi^{4} u-\partial_\xi^2 u + a(x)\,\partial_\xi u + V(\xi)u,\\
		u(\xi,0) = f(\xi),\quad \partial_tu(\xi,0) =g(\xi).
	\end{cases}
\end{equation}
Here $a(\cdot)$, $b(\cdot)$, $V(\cdot)$ are smooth functions. The functions $f,g$ are from a Sobolev space.

In the undamped case ($B=0$), equation \eqref{AB} is well-posed provided the linear operator $A$ generates a strongly continuous cosine family denoted $\{C_{A,0}(t)\}_{t\in\RR}$. If $\{S_{A,0}(t)\}_{t\in\RR}$ is the associated sine family defined by $S_{A,0}(t)x=\displaystyle\int_0^tC_{A,0}(s)x\rmd s$, $t\in\RR$, $x\in\bX$, then the initial value problem has a unique mild solution 
given by $u(t)=C_{A,0}(t)x+S_{A,0}(t)y$, $t\in\RR$, see, e.g., \cite[Corollary 3.14.8]{ABHN}. In this paper we construct two strongly continuous operators families, called generalized cosine and sine families, which give the mild solutions of \eqref{AB} in the general undamped case. Once we define the generalized cosine and sine operators families we study their properties, including finding precise conditions for differentiability, important invariant subspaces, generalized trigonometric identities and exponential growth.

There are two natural ways to define mild solutions of \eqref{AB}. The first, is to solve \eqref{AB} formally using the Laplace Transform on $[0,\infty)$ and $(-\infty,0]$, see Definition~\ref{Def-mild} below. The second is to formally integrate \eqref{AB} twice to obtain an integral equation, see Definition~\ref{Def-Int-Mild} below. In this paper we show that the two definitions are equivalent. We will use the option that uses the Laplace Transform, which allows us to introduce the generalized cosine and sine families earlier in the exposition. One inspiration for this approach is the general theory of strongly continuous cosine families $\{C_{A,0}(t)\}_{t\in\RR}$ whose generator is defined via the Laplace Transform. 

An important role in our analysis is played by the quadratic pencil $Q(\lambda):=\lambda^2 I+2\lambda B-A$, $\lambda\in\CC$. Also, we recall the notation
$\mathcal{C}_{-\nu}(\mathbb{R},\bX)
	= \big\{ f \in \mathcal{C}(\mathbb{R},\bX) : \sup_{t \in \mathbb{R}} e^{-\nu|t|}\|f(t)\| < \infty \big\}$. 
\begin{definition}\label{Def-mild}
	If $x\in\dom(B)$ and $y\in\bX$, then a function $u \in \mathcal{C}_{-\nu}(\mathbb{R},\bX)$ is called \emph{Laplace Transform mild solution} of \eqref{AB}
	provided $(\mathcal{L}u)(\lambda),\, (\mathcal{L}u(-\cdot))(\lambda) \in \dom(Q(\lambda))$ and there exists $\tnu>\nu$ such that
	\begin{equation}\label{Def-Lap-Mild}
		Q(\lambda)(\mathcal{L}u)(\lambda) = \lambda x + 2Bx + y,\; Q(-\lambda)(\mathcal{L}u(-\cdot))(\lambda) = \lambda x - 2Bx - y\;\mbox{for any}\;\lambda \in \mathbb{C}^+_{\tnu}. 
	\end{equation}
\end{definition}
\begin{definition}\label{Def-Int-Mild}
	If $x\in\dom(B)$ and $y\in\bX$, then a function $u \in \mathcal{C}(\mathbb{R}, \bX)$ is called an \emph{integral mild solution} provided	
	\begin{equation}\label{domain-cond}
		\int_0^t u(s)\,\rmd s \in \dom(B),\quad
		\int_0^t (t-s)u(s)\,\rmd s \in \dom(A)\;\mbox{for any}\; t\in\RR,
	\end{equation}
	\begin{equation}\label{cont-cond}
		t \mapsto B\int_0^t u(s)\,\rmd s, \quad
		t \mapsto A\int_0^t (t-s)u(s)\,\rmd s \;\in\; \mathcal{C}(\mathbb{R}, \bX),
	\end{equation}
	\begin{equation}\label{Int-Mild}
		u(t) + 2B\int_0^t u(s)\,\rmd s
		= A\int_0^t (t-s)u(s)\,\rmd s + x + t(y + 2Bx)\;\mbox{for any}\; t\in\RR.
	\end{equation}
\end{definition}
In Theorem~\ref{t4.8}, cf. also Remark~\ref{r4.9}, we show that these two definitions are equivalent. For this reason, we will refer in this introduction to both type of solutions simply as mild solutions.

Our goal is to prove the existence and uniqueness of classical and especially  mild solutions of \eqref{AB} under some general conditions on operators $A$ and $B$ that include many PDE examples. In particular we are interested in the case when the linear operator $B$ is \emph{unbounded}. Many linear models that can be written in the form \eqref{AB} are obtained by linearizing a nonlinear PDE along a smooth, localized pattern, see Remark~\ref{r8.10}. As a consequence the coefficients of all the differential operators are smooth functions. The literature discussing special cases or very specific models of \eqref{AB} with $B$ unbounded is quite vast, especially in the case when the space $\bX$ is a Hilbert space and one assumes $A$ be self-adjoint and positive definite, see e.g., \cite{Ar,BBD,CR,CT1,CT2,FZG,FST,IT1,IT2,ITY,XL1,XL2,XL3}. Another goal of our paper is to establish a general framework that applies to the case of damped waves equations in Banach spaces, where several Hilbert space tools are not readily available.

\noindent\textbf{Hypothesis (HAB).} Assume $A_0:\dom(A_0)\subseteq\bX\to\bX$, $B_0:\dom(B_0)\subseteq\bX\to\bX$ are
generators of $C_0$-groups on $\bX$, $W_0\in\mathcal{B}(\dom(A_0),\bX)$ $B_1,\widetilde{B}_1\in\mathcal{B}(\bX)$ are such that
\begin{itemize}
	\item[(i)] $A_0 B_0 = B_0 A_0$;
	\item[(ii)] $\dom(A_0) \subseteq \dom(B_0)$,
	$B_0(\dom(A_0^2)) \subseteq \dom(A_0)$;
	\item[(iii)] $A_0^\pm:= \pm A_0 - B_0$ generate $C_0$-groups on $\bX$;
	\item[(iv)] $B_1(\dom(A_0))\subseteq\dom(A_0)$ and $(B_1)\big|_{\dom(A_0)} \in \mathcal{B}(\dom(A_0))$;
	\item[(v)] $B_0 B_1 - B_1 B_0 = \widetilde{B}_1\,x$
	for any $x \in \dom(A_0^2)$;
	\item[(vi)] $A = A_0^2 - B_0^2 + W_0$, $B = B_0 + B_1$, $\dom(A)=\dom(A_0^2)$ and $\dom(B)=\dom(B_0)$.
\end{itemize}
Hypothesis (HAB)(ii) implies $\dom(A_0^2) \subseteq \dom(B_0^2)$. Hence, it is natural to consider $A$ as an unbounded linear operator on $\bX$ with domain $\dom(A) = \dom(A_0^2)$. The class of operators covered by Hypothesis (HAB) is very general. In our setup the linear operators $A$ and $B$ do not commute, only their ``leading order terms" $A_0$ and $B_0$ commute. Hypothesis (HAB) is especially amiable to multidimensional PDE models. 
Special cases of (HAB) are given in Section~\ref{s8}, see Hypothesis (H1) or (H2) or (H3). Each of them imply (HAB) cf. Lemma~\ref{l8.1}, and are easier to check.

Our hypotheses allow us to prove the existence and uniqueness of mild solutions of the damped wave \eqref{AB} even in a special case when $A$ does not generate a $C_0$-semigroup, see Example~\ref{e8.4} with $\gamma\in(-1,0)$. In such a case the undamped equation obtained by dropping the damping term $Bu'$ in \eqref{AB} is not well-posed, since $A$ cannot generate a $C_0$-cosine family. This shows that there are cases when the existence and uniqueness of mild solutions of \eqref{AB} is due to the presence of damping. 

A classical approach to solve the initial value problem \eqref{AB} is to rewrite it as a first order initial value problem. We set $v=u'+Bu$. Then \eqref{AB} is equivalent to
\begin{equation}\label{ABG}
	\begin{pmatrix}u\\v\end{pmatrix}' = G\begin{pmatrix}u\\v\end{pmatrix},\;
	\begin{pmatrix}u(0)\\v(0)\end{pmatrix} = \begin{pmatrix}x\\Bx+y\end{pmatrix},
\end{equation}
where
\begin{equation}\label{eq:G}
	G = \begin{bmatrix} -B & I \\ A+B^2 & -B \end{bmatrix}.
\end{equation}  
In the undamped case, when $B=0$ and $A$ generates a cosine family with phase space $\bV\times\bX$, the linear operator $G$ defined above generates a $C_0$-group on $\bV\times\bX$, see, e.g.,\cite[Theorem 3.14.11]{ABHN}. In the general damped case we prove that the linear operator $G$ generates a $C_0$-group on $\bY\times\bX$ for an appropriate subspace $\bY$ of $\bX$, assuming Hypothesis  (HAB), see Section~\ref{s2}. The key idea is to decompose $G=G_0+V_0$, where $G_0$ contains all the ``leading order" unbounded terms of $A$ and $B$ in $G$ along with $I$ in the upper-right corner, see \eqref{eq:3} below, and $V_0$ is a bounded linear operator. This type of decomposition is especially important in the case when $B$ is unbounded. A crucial result is Lemma~\ref{l2.1} below, which we use to show that $G_0$ defined in \eqref{eq:3} generates a $C_0$-group that can be given explicitly. The latter fact is the reason of our somewhat unusual choice of substitution $v=u'+Bu$, see Remark~\ref{r2.2-new}.

Next, we formulate another set of assumptions on $A$ and $B$ used in this paper.

\noindent\textbf{Hypothesis (ABY).} Assume $A : \dom(A) \subseteq\bX \to\bX$, $B : \dom(B) \subseteq\bX \to\bX$ are closed linear operators, $\bY$ is a Banach space such that $\bY \hookrightarrow \bX$, $\rho(B)\ne\emptyset$ and
\begin{itemize}
	\item[(ABY1)] $\dom(A) \subseteq \dom(B^2) \cap\bY$;
	\item[(ABY2)] $B(\dom(A)) \subseteq\bY \subseteq \dom(B)$;
	\item[(ABY3)] The linear operator $G:\dom(A)\times\bY\subseteq\bY\times\bX\to\bY\times\bX$ defined in \eqref{eq:G} generates a $C_0$-group on $\bY\times\bX$ denoted $\{\cT(t)\}_{t\in\RR}$.
\end{itemize}
Under Hypothesis (ABY) we prove existence and uniqueness of classical and mild solutions of \eqref{AB}, under appropriate assumptions on the initial conditions. In turn this leads to a comprehensive treatment of the properties of the cosine and sine families that describe the mild solutions of \eqref{AB}, see Section~\ref{s4}. In other words, Hypothesis (ABY) can be seen as a reformulation of well-posedness of \eqref{AB} via passing to the first order system \eqref{ABG}. On the other hand Hypothesis (HAB) can be viewed as a sufficient condition for Hypothesis (ABY) to hold designed to be easily checked in concrete examples, as described in Section~\ref{s8}.

For unbounded operators $A$, $B$ on a Banach space $\bX$, the well-known book \cite{Fa}
separates the notion of well-posedness of the (second order in time) damped wave equation (understood as existence and uniqueness of classical solutions for a dense set of initial data) from the existence of the phase space $\bY\times\bX$ (where the respective first order in time differential operator generates a strongly continuous group).  
We find this separation inconvenient. Instead, our Hypothesis \textnormal{(ABY)} directly postulates the existence of the phase space and at the same time provides reformulation of the well-posedeness of 
the damped wave equation understood in terms of either mild or classical solutions via well-posedeness of the respective first order in time equation on $\bY\times\bX$. Indeed, the core results of the current paper, Theorem \ref{t3.1} and \ref{t3.10}, show that Hypothesis \textnormal{(ABY)} imply the existence and uniqueness of the classical and mild solutions respectively. Going further, we use Hypothesis \textnormal{(ABY)} to define the generalized cosine and sine families and describe in Section \ref{s4} their ``trigonometric'' properties and in Section \ref{s5} various properties of the non-homogeneous damped wave equations also implied by Hypothesis \textnormal{(ABY)}. Here, Hypotheses \textnormal{(ABY1)--(ABY2)} are of preliminary nature helping to set the stage for the main Hypothesis \textnormal{(ABY3)} of well-posedeness of the first order in time system. Naturally, a question arises how to verify Hypothesis \textnormal{(ABY3)}. One of the main achievement of the current paper is formulating (easily verifiable for concrete multi- and one-dimensional
PDEs) assumptions on the operators $A$ and $B$ listed as Hypothesis \textnormal{(HAB)} (and special cases of \textnormal{(HAB)}, see Hypotheses \textnormal{(H1)}, \textnormal{(H2)} and \textnormal{(H3)} in Section \ref{s8}) under which Hypothesis \textnormal{(ABY)} does hold. That \textnormal{(HAB)} implies \textnormal{(ABY)} we show in yet another core result, 
Theorem \ref{t2.3}. Finally, in the forthcoming work \cite{SeconPaper} we show that the phase space $\bY\times\bX$ is uniquely determined by  Hypothesis \textnormal{(ABY)}; 
this therefore completes the theory of the abstract \emph{damped} wave equation to the level of the undamped case, cf. \cite[Theorem 3.14.11]{ABHN}.

From (ABY2) one can readily derive
\begin{equation}\label{space embedding}
	\dom(A)\hookrightarrow\bY\hookrightarrow\dom(B)\hookrightarrow\bX,	
\end{equation}	
see Remark~\ref{r3.1} below, whence in general, all the inclusions are strict, see Example~\ref{e8.5}.

In the undamped case, when $A$ is the generator of a cosine family, the linear operator $\left[\begin{smallmatrix} 0 & I \\ A & 0 \end{smallmatrix}\right]$ generates not only a $C_0$-semigroup but a $C_0$-group on $\bV\times\bX$, the phase space of $A$, due to the fact that the cosine family is an even function and the sine family is an odd function. In the undamped case we expect the  
generalized cosine and sine families to be defined on the full line, are neither even nor odd. This is the reason why we need  to insist in Hypothesis (ABY) for the operator $G$ to generate a $C_0$-group, not merely a $C_0$-semigroup.

Assuming \textnormal{(ABY1)--(ABY3)} we focus on finding classical solutions and then Laplace Transform mild solutions of \eqref{AB}. To achieve this goal we first characterize the resolvent operator of $G$ in terms of the invertibility properties of the pencil $Q(\lambda)=\lambda^2 I+2\lambda B-A$. Another crucial step is to understand the consequences of hypothesis (ABY3) for the block matrix components of $\{\cT(t)\}_{t\in\RR}$. At this stage it is more convenient to work with the Laplace Transform mild solutions because this allows us to define the generalized cosine and sine families earlier in the paper as solutions of the special cases of \eqref{AB} with either $y=0$ or $x=0$. The families are denoted by $\{\CAB(t)\}_{t\in\RR}$ and $\{\SAB(t)\}_{t\in\RR}$, respectively, and are defined in \eqref{def-CAB} and \eqref{def-SAB}. 

One of our main goals is to prove the existence and uniqueness of the integral mild solutions of \eqref{AB}, which is achieved once we have the existence of the generalized cosine and sine families of solution operators.
Another goal is to prove that under Hypothesis (HAB) the generalized cosine and sine families retain all the important properties that one expects from the cosine and sine families in the undamped case.  In particular, the subspaces $\bY$ and $\dom(A)$ are invariant under $\CAB(t)$ for any $t\in\RR$, while $\SAB(t)$ maps $\bX$ into $\bY$ and $\bY$ into \dom(A). In Section~\ref{s4} we prove additional details summarized in Figure~\ref{Spaces} below. Also, we are interested to find exact conditions when the trajectories of the generalized cosine and sine families are of class $\mathcal{C}^k$, $k=1,2$. As expected by comparing to the undamped case, the generalized sine function has better regularity properties than the generalized cosine function. The trajectory $\SAB(\cdot)x$ of the generalized sine is always of class $\mathcal{C}^1$ and is of class $\mathcal{C}^2$ if $x\in\bY$. The trajectory $\CAB(\cdot)x$ of the generalized cosine is of class $\mathcal{C}^1$ if $x\in\bY$ and is of class $\mathcal{C}^2$ if $x\in\dom(A)$. A crucial ingredient in these proofs is our assumption that $\dom(A)\subseteq\dom(B^2)$ and that the linear operator $B$ maps $\dom(A)$ into $\bY$ and $\bY$ into $\bX$.

\begin{figure}[htbp]
\begin{center}
	\begin{tikzpicture}[
		every node/.style={font=\normalsize},
		arr/.style={-stealth, thin},
		lbl/.style={font=\small, inner sep=2pt}
		]
		
		\node (X1) at (0, 8)  {$\bX$};
		\node (X2) at (8, 8)  {$\bX$};
		\node (dB1) at (0, 5.5) {$\operatorname{dom}(B)$};
		\node (dB2) at (8, 5.5) {$\operatorname{dom}(B)$};
		\node (Y1) at (0, 4)  {$\bY$};
		\node (Y2) at (8, 4)  {$\bY$};
		\node (dA1) at (0, 0) {$\operatorname{dom}(A)$};
		\node (dA2) at (8, 0) {$\operatorname{dom}(A)$};
		
		\draw[arr] (X1) -- node[lbl, above]{$S'_{A,B}(t)$} (X2);
		
		\draw[arr] (X1) -- node[lbl, above, sloped, pos=0.2]{$S_{A,B}(t)$}   (Y2);
		\draw[arr] (dB1) -- node[lbl, above, sloped, pos=0.6]{$C_{A,B}(t)$} (X2) ;
		\draw[arr] (Y1) -- node[lbl, above, sloped, pos=0.2]{$B$} 
		node[lbl, below, sloped, pos=0.65]{$C'_{A,B}(t)$} (X2) ;
		
		
		
		\draw[arr] (Y1) -- node[lbl, above, sloped, pos=0.4]{$C_{A,B}(t)$} 
		node[lbl, below, sloped, pos=0.65]{$S'_{A,B}(t)$} (Y2);
		\draw[arr] (Y1) -- node[lbl, above, sloped, pos=0.25]{$S_{A,B}(t)$}   (dA2);
		\draw[arr] (dA1) -- node[lbl, above, sloped, pos=0.65]{$C'_{A,B}(t)$} 
		node[lbl, above, sloped, pos=0.25]{$B$} (Y2)  ;
		
		
		
		\draw[arr] (dA1) -- node[lbl, above, sloped, pos=0.4]{$C_{A,B}(t)$} 
		node[lbl, below, sloped, pos=0.65]{$S'_{A,B}(t)$} (dA2);
		
	\end{tikzpicture}
\end{center}
\caption{A detailed representation of the action of generalized cosine and sine families and their derivatives on the sequences of spaces from \eqref{space embedding}.}
\label{Spaces}
\end{figure}
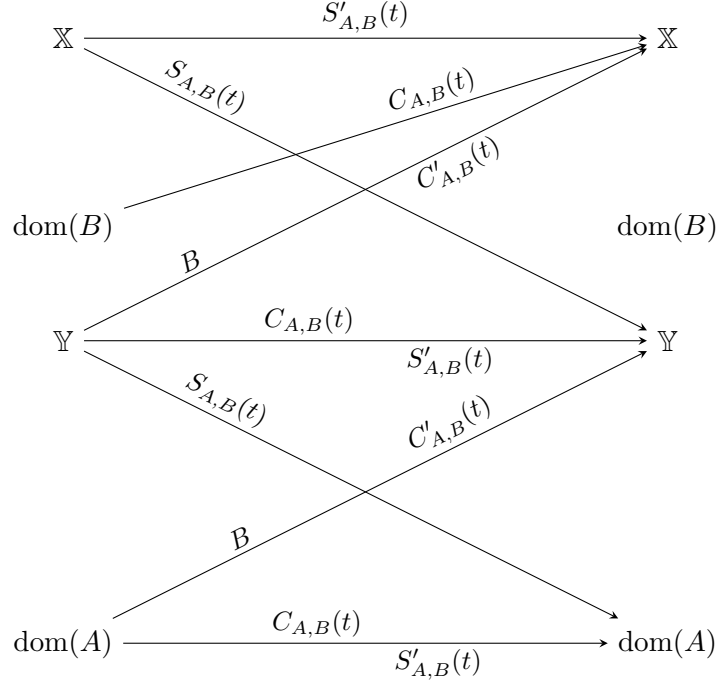

The generalized cosine and sine functions and their first order derivatives (in the strong sense) are exponentially bounded in various norms, c.f. Lemma~\ref{l4.5}. In addition, these operator valued functions maintain their exponential boundedness when multiplied by $A$ or $B$. While some of these properties follow immediately from the exponential growth property of the $C_0$-group $\{\cT(t)\}_{t\in\RR}$, others require more work. 

Only after we have established a significant number of  properties of any generalized cosine and sine families we can show existence and uniqueness of the integral mild solutions of \eqref{AB}. In the process of proving the latter result we also show that the integral mild solutions and Laplace Transform mild solutions of \eqref{AB} do coincide. 

A major difference between the damped and the undamped case is that the values of the generalized cosine function are bounded linear operators acting from $\dom(B)$ to $\bX$, not necessarily from $\bX$ to itself. This fact appears in a natural way, as one can readily see from \eqref{Def-Lap-Mild} or \eqref{Int-Mild}. This is a major difference between the current work and the seminal work of Fattorini (\cite{Fa}), whose notion of well-posedness of \eqref{AB} requires that $\CAB(t)\in\mathcal{B}(\bX)$ for any $t\in\RR$. Despite this shortcoming, the book \cite{Fa} served as a major inspiration for us.

The existence of the generalized cosine and sine family allows us to study the inhomogeneous equation  
\begin{equation}\label{ABf}
	\begin{cases}
		u''(t) + 2Bu'(t) - Au(t) = f(t),\; t\in\RR,\\
		u(0) = x,\,
		u'(0) = y.
	\end{cases}
\end{equation}
Here $f \in L^1_{\mathrm{loc}}(\mathbb{R}, \bX)$, $x, y \in\bX$. Similar to the homogeneous problem \eqref{AB}, we can define mild solutions of \eqref{ABf} by taking Laplace Transform, or by integrating the equation twice. In Section~\ref{s5} we show that regardless of which definition we use, \eqref{ABf} has a unique mild solution given by formula $u=\CAB(\cdot)x+\SAB(\cdot)y+\SAB*f$, assuming $x\in\dom(B)$, $y\in\bX$ and some minimal conditions on $f$, see Theorem~\ref{t5.5}. Also, we discuss classical solutions of \eqref{ABf} and prove their existence and uniqueness provided $x\in\dom(A)$ and $y\in\bY$ (which are exactly the conditions needed to guarantee the existence of a unique solution of the associated homogeneous equation \eqref{AB}) and  $f\in\mathcal{C}^1(\RR,\bX)$ or $f\in\mathcal{C}(\RR,\bY)$. The case $f\in\mathcal{C}^1(\RR,\bX)$ was also considered by Fattorini in \cite{Fa}, under the additional assumption that implies that $\CAB(t)\in\mathcal{B}(\bX)$ for any $t\in\RR$. 
Once we have existence and uniqueness results of the solutions (classical and mild) of \eqref{ABf}, we can naturally obtain ``trigonometric" identities for $\CAB(\cdot)$ and $\SAB(\cdot)$, see Lemma~\ref{l5.1} and Lemma~\ref{l5.6} below.

\noindent\textbf{A glossary of notation.}
$\bX$ is a Banach space with norm $\|\cdot\|$, $L^p(\RR^m,\bX)$, $p\geq 1$, denotes the Lebesgue space of functions on $\RR^m$ with values in $\bX$ with the Lebesgue measure $\rmd x$, 
$W^{s,p}(\RR^m,\bX)$, $s> 0$, $p\geq 1$, is the Sobolev space of $\bX$-valued functions, $\mathcal{C}(\Omega,\bX)$ is the space of continuous $\bX$-valued functions, $\mathcal{C}^m(\Omega,\bX)$ is the space of functions on $\Omega$ having continuous derivatives of order up to $m\in\NN\cup\{0\}$. $\cC^m_\rmb(\R^k,\C^k)$ denotes the space of all functions  in $\cC^m_\rmb(\R^k,\C^k)$ having all partial derivatives of order up to $m$ bounded. If $\bX$ is a Hilbert space then we denote by $H^s(\RR^m,\bX)=W^{s,2}(\RR^m,\bX)$ for $s>0$. The space of locally $p$-integrable functions with values in a Banach space $\bX$ is denoted by $L^p_{\mathrm{loc}}(\RR^m,\bX)$. 
The identity operator on the Banach space $\bX$ is denoted by $I_\bX$. The set of bounded linear operators in the Banach space $\bX$ is denoted by $\cB(\bX)$. The space of continuous linear operator between two topological vector spaces $\bV_1$ and $\bV_2$ is denoted by $\cC\cL(\bV_1,\bV_2)$. For an operator $B$ on a Banach space $\bX$, we use  $\dom(B)$, $\Ker B$, $\Range B$, $\mathrm{Graph}(B)$, $\sigma(B)$, $\rho(B)$, and $B_{|\bW}$ to denote the domain, kernel, range, graph, spectrum, resolvent set, and the restriction of $B$ to a subspace $\bW$ of $\bX$. The adjoint operator is denoted by $B^*$. We denote by $\sigma_{\mathrm{disc}}(B)$ the set of isolated eigenvalues of finite algebraic multiplicity of the linear operator $B$, and by $\sigma_{\mathrm{ess}}(B)$ its complement in the spectrum of $B$. We use notation $R(\mu,B)=(\mu I_\bX-B)^{-1}$ for the resolvent of $B$ at $\mu\in\rho(B)$. For any Banach space $\bX$ we use notation $\langle x,x^*\rangle:=x^*(x)$ for any $(x,x^*)\in\bX\times\bX^*$.  The direct sum of two subspaces $\bW_1$ and $\bW_2$ is denoted by $\bW_1\oplus \bW_2$. The operator of multiplication by a function $g$ is denoted by $M_g$.  The open disc in $\CC$ centered at $a$ of radius $\eps>0$ is denoted by $D(a,\eps)$. We denote by $\CC_\omega^+$ the set of all $\lambda\in\CC$ such that $\mathrm{Re}\lambda>\omega$. 
If $K: \mathbb{R} \to \mathcal{B}(\bX_1, \bX_2)$ is such that $K(\cdot)\,x_1 \in \mathcal{C}^1(\mathbb{R}, \bX_2)$ for any $x_1 \in \bX_1$ we denote by $K': \mathbb{R} \to \mathcal{B}(\bX_1, \bX_2)$ is the strongly continuous operator-valued function defined by $K'(t)\,x = (K(t)\,x)'$. 	
	
\section{Group Generators and First Order Systems}\label{s2}
In this section we prove that under Hypotheses (HAB) the operator $G$ defined in \eqref{eq:G} generates a $C_0$-group on a subspace of $\bX\times\bX$. We begin by proving the following general lemma
\begin{lemma}\label{l2.1}
	Assume $A_0:\dom(A_0)\subseteq\bX\to\bX$ and $B_0:\dom(B_0)\subseteq\bX\to\bX$ are generators
	of $C_0$-groups $\{e^{tA_0}\}_{t\in\R}$ and $\{e^{tB_0}\}_{t\in\R}$ on $\bX$, respectively, such that
	\begin{enumerate}
		\item[(i)] $A_0 B_0 = B_0 A_0$,
		\item[(ii)] $\dom(A_0)\subseteq\dom(B_0)$, $B_0(\dom(A_0^2))\subseteq\dom(A_0)$,
		\item[(iii)] $A_0^\pm := \pm A_0 - B_0$ generate $C_0$-groups $\{e^{tA_0^\pm}\}_{t\in\R}$ on $\bX$, respectively.
	\end{enumerate}
	Then, $G_0:\dom(A_0^2)\times\dom(A_0)\subseteq\dom(A_0)\times \bX\to\dom(A_0)\times \bX$ defined by
	\begin{equation}\label{eq:3}
		G_0 = \begin{bmatrix} -B_0 & I \\ A_0^2 & -B_0 \end{bmatrix}
	\end{equation}
	is the generator of the $C_0$-group $\{\cT_0(t)\}_{t\in\R}$ on $\dom(A_0)\times\bX$ given by
\begin{equation}\label{eq:4}
\cT_0(t) = \begin{bmatrix}
\tfrac{1}{2}(e^{tA_0^+}+e^{tA_0^-})\big|_{\dom(A_0)} &
\displaystyle\int_0^t e^{(t-\tau)A_0^+}e^{\tau A_0^-}\,\rmd\tau \\
\tfrac{1}{2}(e^{tA_0^+}-e^{tA_0^-})A_0\big|_{\dom(A_0)} &
\tfrac{1}{2}(e^{tA_0^+}+e^{tA_0^-})
\end{bmatrix},\; t\in\RR.
\end{equation}
\end{lemma} 
\begin{proof} First, we prove that $G_0$ and $\cT_0(t)$ are well-defined. Indeed, assumption (ii) implies that $G_0$ is well-defined and we may set $\dom(A_0^\pm)=\dom(A_0)$. Since $e^{tA_0^\pm}(\dom(A_0^\pm))\subseteq\dom(A_0^\pm)$, we have
\begin{equation}\label{l2.1.1}
		e^{tA_0^\pm}(\dom(A_0))\subseteq\dom(A_0)\quad\text{for any }t\in\R.
	\end{equation}
From assumption (i) and the definition of $A_0^\pm$ we infer that $A_0,B_0,A_0^+,A_0^-$ commute. Similarly, $ R(\lambda_1,A_0)$, $R(\lambda_2,B_0)$, $R(\lambda_3,A_0^+)$ and $R(\lambda_4,A_0^-)$ commute 
for any $\lambda_1\in\rho(A_0)$, $\lambda_2\in\rho(B_0)$, $\lambda_3\in\rho(A_0^+)$ and $\lambda_4\in\rho(A_0^-)$. It follows that 
	\begin{equation}\label{l2.1.2}
		e^{t_1A_0},e^{t_2B_0},e^{t_3A_0^+},e^{t_4A_0^-}\;\mbox{commute for any}\;t_1,t_2,t_3,t_4\in\R,
	\end{equation}	
	\begin{equation}\label{l2.1.3}
		e^{tA_0}e^{sB_0}=e^{sB_0}e^{tA_0},\;\;e^{tA_0^\pm}=e^{tA_0}e^{-tB_0}=e^{-tB_0}e^{tA_0}\;\mbox{for any}\; t,s\in\R
	\end{equation}	
	From \eqref{l2.1.1}, \eqref{l2.1.2} and \eqref{l2.1.3} we obtain 
	\begin{equation}\label{l2.1.4}
		\int_0^t e^{(t-\tau)A_0^+}e^{\tau A_0^-}x\,\rmd\tau
		= e^{tA_0^+}\int_0^t e^{-A_0\tau}e^{\tau B_0}e^{-A_0\tau}e^{-\tau B_0}x\,\rmd\tau=e^{tA_0^+}\int_0^t e^{-2\tau A_0}x\,\rmd\tau\in\dom(A_0)
	\end{equation}
	for any $t\in\R$, $x\in\bX$. Hence $\cT_0(t)$ is well-defined.
	
	Next, we prove that $\{\cT_0(t)\}$ is a $C_0$-group on $\dom(A_0)\times \bX$. Since $A_0$ and $A_0^\pm$ commute,
	\begin{equation}\label{l2.1.5}
		A_0 e^{tA_0^\pm}x = e^{tA_0^\pm}A_0 x\;\mbox{for any}\;\,t\in\R,\;x\in\dom(A_0).
	\end{equation}
	Moreover, from \eqref{l2.1.3}, \eqref{l2.1.4} and \eqref{l2.1.5} we can see that 
	\begin{equation}\label{l2.1.6}
		A_0\int_0^t e^{(t-\tau)A_0^+}e^{\tau A_0^-}x\,\rmd\tau
		= e^{tA_0^+}A_0\int_0^t e^{-2\tau A_0}x\,\rmd\tau
		= \tfrac{1}{2}(e^{tA_0^+}-e^{tA_0^-})x
	\end{equation}
	for any $x\in \bX$, $t\in\R$. Clearly, by letting $t=0$ in \eqref{eq:4}, we have $\cT_0(0) = I_{\dom(A_0)\times \bX}$. The group property of $\cT_0$ follows after a long but straightforward computation based on \eqref{l2.1.2}--\eqref{l2.1.6}. 
	
	Next, we prove the strong continuity of $\cT_0$. Fix $y\in\dom(A_0)$, $x\in \bX$. Then, from \eqref{eq:4} and \eqref{l2.1.6} we obtain 
	\begin{align}\label{2.1.14}
		&\left\|\cT_0(t)\tbinom{y}{x}-\tbinom{y}{x}\right\|_{\dom(A_0)\times\bX}=\left\|\tfrac{1}{2}(e^{tA_0^+}+e^{tA_0^-})y-y
		+e^{tA_0^+}\int_0^t e^{-2\tau A_0}x\,d\tau\right\|+\left\|\tfrac{1}{2}(e^{tA_0^+}-e^{tA_0^-})A_0 y\right\|
		\nonumber\\
		&\qquad\qquad\qquad\qquad+\left\|\tfrac{1}{2}(e^{tA_0^+}+e^{tA_0^-})A_0y-A_0y
		+\tfrac{1}{2}(e^{tA_0^+}-e^{tA_0^-})x\right\|   
		+\left\|\tfrac{1}{2}(e^{tA_0^+}+e^{tA_0^-})x-x\right\|
	\end{align}
	for any $t\in\RR$. Since $\{e^{tA_0^\pm}\}_{t\in\RR}$ and $\{e^{tA_0}\}_{t\in\RR}$ are $C_0$-groups we know that
	\begin{equation}\label{2.1.15}
		\lim_{t\to0}e^{tA_0^\pm}z_1=z_1,\quad\lim_{t\to0}\int_0^t e^{-2\tau A_0}z_2\,d\tau=0
		\;\mbox{for any}\;z_1,z_2\in\bX,
	\end{equation}
	and so $\{\cT_0(t)\}_{t\in\RR}$ is a $C_0$-group.
	
	To finish the proof of the lemma we prove that $G_0$, defined in \eqref{eq:3}, is the generator of $\{\cT_0(t)\}_{t\in\R}$. First, we show that there exists $\omega_0\in\RR$ such that $\CC_{\omega_0}^+\subseteq\rho(G_0)$ and compute $R(\lambda,G_0)$ for $\lambda\in\CC_{\omega_0}^+$. Since $\{e^{tA_0^\pm}\}$ are $C_0$-groups, there exist $\omega_0^\pm\in\R$ and $M_0^\pm\geq 0$ such that
	\begin{equation}\label{2.1.16}
		\|e^{tA_0^\pm}\|\le M_0^\pm e^{\omega_0^\pm|t|}\;\;\mbox{for any}\;\;t\in\R.
	\end{equation}
	Define $\omega_0=\max\{\omega_0^+,\omega_0^-\}$ and $Q_0(\lambda):\dom(A_0^2)\subseteq\bX\to \bX$ by
	\begin{equation}\label{2.1.17}
		Q_0(\lambda)=(\lambda I+B_0)^2-A_0^2,\;\mbox{for}\;\lambda\in\CC.
	\end{equation}
	From assumption (ii),  $\dom(A_0^2)\subseteq\dom(B_0^2)$. By assumption (i), 
	\begin{equation}\label{2.1.18}
		Q_0(\lambda)=(\lambda I+B_0-A_0)(\lambda I+B_0+A_0)=(\lambda I-A_0^+)(\lambda I-A_0^-)\;\mbox{for any}\;\lambda\in\CC.
	\end{equation}
	It follows that $Q_0(\lambda)$ is invertible with bounded inverse for $\lambda\in\mathbb{C}_{\omega_0}^+$ and 
	\begin{equation}\label{2.1.19}
		Q_0^{-1}(\lambda)=R(\lambda,A_0^-)R(\lambda,A_0^+)=R(\lambda,A_0^+)R(\lambda,A_0^-)\;\mbox{for any}\;\lambda\in\CC.
	\end{equation}
	Next, we solve 
	\begin{equation}\label{2.1.20}
		\lambda\binom{u}{v}-G_0\binom{u}{v}=\binom{f}{g},\;\mbox{with}\; u\in\dom(A_0^2)\;\mbox{and}\; v,f\in\dom(A_0), g\in\bX,\;\mbox{for}\; \lambda\in\mathbb{C}_{\omega_0}^+. 
	\end{equation}
	By \eqref{eq:3} equation \eqref{2.1.20} is equivalent to
	\begin{equation}\label{2.1.21}
		\begin{cases}v=(\lambda I+B_0)u-f\\Q_0(\lambda)u=(\lambda I+B_0)f+g\end{cases}
	\end{equation}
	By assumption (ii), $(\lambda I+B_0)$ commutes with $Q_0^{-1}(\lambda)$. Thus, the solution of equation \eqref{2.1.20} is given by
	\begin{equation}\label{2.1.22}
		\begin{cases}
			u=(\lambda I+B_0)Q_0^{-1}(\lambda)f+Q_0^{-1}(\lambda)g,\\
			v=(\lambda I+B_0)^2Q_0^{-1}(\lambda)f-f+(\lambda I+B_0)Q_0^{-1}(\lambda)g.
		\end{cases}
	\end{equation}
	We conclude that $\lambda I_{\dom(A_0)\times \bX}-G_0$ is invertible for any $\lambda\in\mathbb{C}_{\omega_0}^+$. Next, we will show that $\big(\lambda I_{\dom(A_0)\times \bX}-G_0\big)^{-1}$ is bounded on $\dom(A_0)\times\bX$. First, using \eqref{2.1.19},
	\begin{align}\label{2.1.23}
		(\lambda I+B_0)Q_0^{-1}(\lambda)&=\tfrac{1}{2}\Big((\lambda I-A_0^+)+(\lambda I-A_0^-)\Big)R(\lambda,A_0^+)R(\lambda,A_0^-)\nonumber\\
		&=\tfrac{1}{2}\Big(R(\lambda,A_0^+)+R(\lambda,A_0^-)\Big)\;\mbox{for any}\;\lambda\in\CC_{\omega_0}^+.
	\end{align}	
	Since $(\lambda I-A_0^\pm)R(\lambda,A_0^\pm)=I$ for any $\lambda\in\CC_{\omega_0}^+$, 
	\begin{equation}\label{2.1.24}
		(\lambda I+B_0)R(\lambda,A_0^+)=I+A_0 R(\lambda,A_0^+),\;(\lambda I+B_0)R(\lambda,A_0^-)=I-A_0 R(\lambda,A_0^-)
	\end{equation}
	for any $\lambda\in\CC_{\omega_0}^+$, which allows us to conclude that 
	\begin{align}\label{2.1.25}
		(\lambda I+B_0)^2Q_0^{-1}(\lambda)-I&=\tfrac{1}{2}\Big((\lambda I+B_0)R(\lambda,A_0^+)+(\lambda I+B_0)R(\lambda,A_0^-)\Big)-I\nonumber\\
		&= \tfrac{1}{2}A_0\Big( R(\lambda,A_0^+)-R(\lambda,A_0^-)\Big)\quad\text{for any }\lambda\in\mathbb{C}_{\omega_0}^+.
	\end{align}
	From \eqref{2.1.22} and \eqref{2.1.25} we obtain 
	\begin{equation}\label{2.1.26}
		(\lambda I_{\dom(A_0)\times\bX}-G_0)^{-1}
		=\begin{bmatrix}
			\tfrac{1}{2}\Big(R(\lambda,A_0^+)+R(\lambda,A_0^-)\Big)\big|_{\dom(A_0)} &
			R(\lambda,A_0^+)R(\lambda,A_0^-)\\
			\tfrac{1}{2}\big(R(\lambda,A_0^+)-R(\lambda,A_0^-)\big)A_0\big|_{\dom(A_0)} &
			\tfrac{1}{2}\big(R(\lambda,A_0^+)+R(\lambda,A_0^-)\big)
		\end{bmatrix}
	\end{equation}
	for any $\lambda\in\CC_{\omega_0}^+$. Since $\dom(A_0^\pm)=\dom(A_0)$ and $A_0$ is closed, we infer that
	$A_0 R(\lambda,A_0^\pm)\in\mathcal{B}(\bX)$ and $\Big(R(\lambda,A_0^+)+R(\lambda,A_0^-)\Big)\big|_{\dom(A_0)}\in\mathcal{B}(\dom(A_0))$ for any $\lambda\in\CC_{\omega_0}^+$. It follows that $A_0R(\lambda,A_0^+)R(\lambda,A_0^-)\in\mathcal{B}(\bX)$, hence $R(\lambda,A_0^+)R(\lambda,A_0^-)\in\mathcal{B}(\bX,\dom(A_0))$. Using the fact that $A_0\big|_{\dom(A_0)}\in\mathcal{B}(\dom(A_0),\bX)$ and $R(\lambda,A_0^\pm)\in\mathcal{B}(\bX)$ we have $\big(R(\lambda,A_0^+)-R(\lambda,A_0^-)\big)A_0\big|_{\dom(A_0)}\in\mathcal{B}(\dom(A_0),\bX)$. We conclude that all four blocks in \eqref{2.1.26} are bounded, which proves 
	$(\lambda I_{\dom(A_0)\times \bX}-G_0)^{-1}\in\mathcal{B}(\dom(A_0)\times \bX)$ for any $\lambda\in\CC_{\omega_0}^+$, thus $\lambda\in\CC_{\omega_0}^+\subset\rho(G_0)$.
	
	Taking Laplace transform in \eqref{eq:4} and using \eqref{2.1.26}, we conclude that
	\begin{equation}\label{2.1.27}
		\mathcal{L}\Bigl\{\cT_0(\cdot)\tbinom{y}{x}\Bigr\}(\lambda)=R(\lambda,G_0)\tbinom{y}{x}
		\;\mbox{for any}\;\lambda\in\mathbb{C}_{\omega_0}^+,\;y\in\dom(A_0),\;x\in\bX,
	\end{equation}
	which proves that $G_0$ is the generator of $\{\cT_0(t)\}_{t\in\R}$.
\end{proof}
\begin{remark}\label{r2.2-new}
	A crucial ingredient in the proof of previous lemma is the fact that $Q_0(\lambda)$ can be factored, see \eqref{2.1.18}. Consequently, we obtain the representation \eqref{2.1.26}, which allows us to find the relatively simple formula \eqref{eq:4}. This fact is due to the block-matrix structure of the operator $G$ defined in \eqref{eq:G}, which can be traced to our choice of substitution $v=u'+Bu$, when passing from the second order equation to the first order system. 
\end{remark}	
We are now ready to show that Hypothesis (ABY) is satisfied provided Hypothesis (HAB) is satisfied. 
\begin{theorem}\label{t2.3}
Assume Hypothesis \emph{(HAB)} and set $\bY=\dom(A_0)$. Then,  Hypothesis \textnormal{(ABY)} is satisfied. In particular, $G$ defined in \eqref{eq:G} generates a $C_0$-group on $\dom(A_0)\times\bX$.
\end{theorem}
\begin{proof}
	Since $B_0$ generates a $C_0$-group and $B_1\in\mathcal{B}(\bX)$ we know that $B=B_0+B_1$ generates a $C_0$-group, hence $\rho(B)\ne\emptyset$. From Hypothesis (HAB) (ii) we infer that $\dom(A_0^2)\subseteq\dom(B_0^2)$, thus $\dom(A) = \dom(A_0^2)$. Moreover, $\dom(B)=\dom(B_0)$ by Hypothesis (HAB)(vi). It follows that $\dom(A) = \dom(A_0^2) \subseteq \dom(A_0) = \bY$. Next, we fix $x \in \dom(A) = \dom(A_0^2)$. Then, $x,\,A_0 x \in \dom(A_0)$.
	From Hypothesis (HAB)(ii) and (iv) we obtain \begin{equation}\label{t2.3.1}
		B\,x = B_0\,x + B_1\,x \in \dom(A_0) + \dom(A_0) = \dom(A_0) = \bY.
	\end{equation}
	But $\bY \subseteq \dom(B_0)=\dom(B)$ by Hypothesis (HAB)(ii), so $x \in \dom(B^2)$.
	Hence $\dom(A) \subseteq \dom(B^2)$, proving (ABY1). From \eqref{t2.3.1} we infer $B(\dom(A)) \subseteq \bY \subseteq \dom(B_0)\dom(B)$, proving (ABY2).
	
	From Lemma~\ref{l2.1} we infer that  $G_0:\dom(A_0^2)\times\dom(A_0)\subseteq \dom(A_0)\times\bX\to\dom(A_0)\times\bX$ defined by $G_0=\left[\begin{smallmatrix}-B_0&I\\A_0^2&-B_0\end{smallmatrix}\right]$ generates a $C_0$-group on $\dom(A_0)\times\bX$. From Hypothesis (HAB)(v) we  obtain
	\begin{equation}\label{t2.3.2}
		B^2 x = B_0^2 x+2B_1 B_0 x+(\widetilde{B}_1+B_1^2)x\;\mbox{for any}\;x\in\dom(A_0^2).
	\end{equation}
	\begin{equation}\label{t2.3.3}
		(A+B^2)x=A_0^2x+W_0x+2B_1B_0x + (\widetilde{B}_1 + B_1^2)x\;\mbox{for any}\;x\in\dom(A_0^2).
	\end{equation}
	From \eqref{t2.3.3} we conclude that $G=G_0+V_0$, where $V_0:\dom(A_0)\times\bX\to\dom(A_0)\times\bX$ is defined by
	\begin{equation}\label{t2.3.4}
		V_0 = \begin{bmatrix}
			-(B_1)\big|_{\dom(A_0)} & 0 \\
			W_0 + \big(2\,B_1\,B_0 + (\widetilde{B}_1 + B_1^2)\big)\big|_{\dom(A_0)} & -B_1
		\end{bmatrix}.	
	\end{equation}	
Since $B_0$ is closed and $\dom(A_0) \subseteq \dom(B_0)$, we conclude that
$B_0\big|_{\dom(A_0)} \in \mathcal{B}(\dom(A_0), \bX)$ by the Closed Graph Theorem.
Since $B_1, \widetilde{B}_1 \in \mathcal{B}(\bX)$, $W_0 \in \mathcal{B}(\dom(A_0), \bX)$,
and $(B_1)\big|_{\dom(A_0)} \in \mathcal{B}(\dom(A_0))$, we infer
	$V_0 \in \mathcal{B}(\dom(A_0) \times \bX)$, thus $G$ generates a $C_0$-group on $\dom(A_0)\times\bX$, proving the theorem.
\end{proof}

\section{Mild Solutions of Damped Waves Equations}\label{s3}
In this section we study classical and mild solutions of the  initial value problem for the damped wave \eqref{AB}. 
Throughout this section we assume Hypotheses (ABY). First, we point out several immediate consequences of Hypothesis (ABY). Throughout this paper we use notation $L:=A+B^2$ for the linear operator with domain $\dom(L)=\dom(A)$.
\begin{remark}\label{r3.1} Assume (ABY1)-(ABY2). Using the Closed Graph Theorem we obtain the following.
\begin{itemize}
\item[(i)] Assumptions $\bY \hookrightarrow\bX$ and $\bY \subseteq \dom(B)$ imply 
\begin{equation}\label{r3.1.1}
B|_\bY \in \mathcal{B}(\bY, \bX).
\end{equation}
\item[(ii)] Assumptions $B(\dom(A)) \subseteq\bY$, that $B$ is closed, and $\bY \hookrightarrow \bX$ imply 
\begin{equation}\label{r3.1.2}
B|_{\dom(A)} \in \mathcal{B}(\dom(A), \bY).
\end{equation}\label{r3.1.3}
\item[(iii)] Since $\dom(A) \subseteq \bY$, $\bY \hookrightarrow\bX$, $\dom(A) \hookrightarrow \bX$, we infer that the inclusion $\dom(A)\hookrightarrow\bY$ is continuous. So, there exists $c_0 > 0$
such that 
\begin{equation}\label{r3.1.4}
\|x\|_\bY \leq c_0 \|x\|_{\dom(A)} \;\text{for any}\; x\in\dom(A).
\end{equation}
\end{itemize}
\end{remark}
In the next lemma we study the resolvent set of the linear operator $G$. We introduce the quadratic pencil 
\begin{equation}\label{defQlam}
	Q(\lambda):\dom(A)\subseteq\bX\to\bX\;\mbox{by}\; Q(\lambda) := \lambda^2 I + 2\lambda B - A,\;\mbox{for}\; \lambda\in\CC.
\end{equation}
\begin{lemma}\label{l3.2}
	Assume \textnormal{(AB1)--(AB2)}. Then $\lambda \in \rho(G)$ if and only if
	$Q(\lambda) := \lambda^2 I + 2\lambda B - A$ is invertible with bounded inverse.
\end{lemma}
\begin{proof} First, we consider the equation
	\begin{equation}\label{3.2.1}
		\lambda \binom{u}{v} - G\binom{u}{v} = \binom{f}{g},
		\quad u \in \dom(A),\; v \in \bY,\; f \in \bY,\; g \in \bX.
	\end{equation}
	From the definition of $G$ in \eqref{eq:G} we immediately see that \eqref{3.2.1} is equivalent to the system
	\begin{equation}\label{3.2.3}
		v = (\lambda I + B)u - f,\; (\lambda^2 I+ 2\lambda B - A)u = (\lambda I+B)f + g. 
	\end{equation}
	
	\noindent\textit{Necessity.}
	Assume that $\lambda \in \rho(G)$.
	Fix $u \in \Ker(Q(\lambda))=\dom(A)$. 
	Set $v = \lambda u + Bu \in\bY$ (by (ABY2)).
	From the equivalence of \eqref{3.2.1} and \eqref{3.2.3} we infer
	\begin{equation}\label{l3.2.4}
		\lambda \binom{u}{v} - G\binom{u}{v} = \binom{0}{0}.
	\end{equation}
	Since $\lambda\in\rho(G)$ we conclude that $u=0$. Hence $\Ker(Q(\lambda)) = \{0\}$.
	Next, we fix $g \in\bX$ and set $\binom{u}{v} = R(\lambda, G)\binom{0}{g}\in\dom(G)=\dom(A)\times\bY$. Thus, $(\lambda I_{\bY\times\bX}-G)\binom{u}{v}=\binom{0}{g}$. By the equivalence of \eqref{3.2.1} and \eqref{3.2.3}, 
	$v = (\lambda I+B)u$ and  $Q(\lambda)u = g$. Thus, $\Range(Q(\lambda)) = \bX$.
	We conclude that $Q(\lambda)$ is invertible. Moreover, by solving for $u$ and $v$ in \eqref{3.2.3} we have
	\begin{equation}\label{3.2.5}
		R(\lambda, G) = \begin{bmatrix}
			Q^{-1}(\lambda)(\lambda I+B)|_\bY & Q^{-1}(\lambda) \\
			\bigl((\lambda I+B)Q^{-1}(\lambda)(\lambda I+B) - I\bigr)|_\bY & (\lambda I+B)Q^{-1}(\lambda)
		\end{bmatrix}.
	\end{equation}
	To prove that the inverse $Q^{-1}(\lambda)$ is bounded we proceed as follows. From \eqref{3.2.5} one can readily check that
	\begin{equation}\label{3.2.6}
		R(\lambda, G)\binom{0}{x} = \begin{pmatrix} Q^{-1}(\lambda)x \\ (\lambda+B)Q^{-1}(\lambda)x \end{pmatrix}
		\;\text{for any}\; x \in \bX.
	\end{equation}
	Thus,
	\begin{equation}\label{3.2.7}
		\|Q^{-1}(\lambda)x\|_\bY \leq \left\| R(\lambda,G)\binom{0}{x} \right\|_{\bY\times \bX}
		\leq \|R(\lambda,G)\| \|x\|\;\text{for any}\; x \in \bX.
	\end{equation}
	Since $\bY \hookrightarrow\bX$ there exists $c_1 > 0$ such that
	$\|y\| \leq c_1 \|y\|_\bY$ for any $y \in \bY$. From \eqref{3.2.7} we conclude that 
	\begin{equation}\label{3.2.8}
		\|Q^{-1}(\lambda)x\| \leq c_1 \|Q^{-1}(\lambda)x\|_\bY \leq c_1 \|R(\lambda,G)\| \|x\|\;\text{for any}\; x \in \bX,
	\end{equation}
	Hence $Q^{-1}(\lambda) \in \mathcal{B}(\bX)$, proving the Necessity. 
	
	\noindent\textit{Sufficiency.} Assume $Q(\lambda) = \lambda^2 I + 2\lambda B - A$ is invertible and $Q^{-1}(\lambda) \in \mathcal{B}(\bX)$. Let $\binom{u}{v} \in \Ker(\lambda I_{\bY\times\bX} - G)$. Since $\dom(G)=\dom(A)\times\bY$ we have $u \in \dom(A)$, $v \in \bY$.
	Using again the equivalence of \eqref{3.2.1} and \eqref{3.2.3} it follows that
	$v = (\lambda I+B)u$ and  $Q(\lambda)u = 0 $, which implies that $u = 0$ and  $v = 0$. Fix $\binom{f}{g} \in\bY\times\bX$. Set $u = Q^{-1}(\lambda)\big((\lambda I+B)f + g\big)\in \dom(A)$ and
	$v = (\lambda I+B)u - f \in\bY$.  From\eqref{3.2.3} we obtain that
	$(\lambda I_{\bY\times\bX} - G)\binom{u}{v} = \binom{f}{g}$.
	Thus, $\Range(\lambda I_{\bY\times\bX} - G) = \bY \times \bX$, and so $\lambda I_{\bY\times \bX} - G$
	is invertible. Moreover, $(\lambda I_{\bY\times\bX} - G)^{-1}$ is given by the RHS of \eqref{3.2.5}.
	
	Next, we prove that $(\lambda I_{\bY\times\bX} - G)^{-1}$ is bounded on $\bY\times\bX$.
	
	We claim that  $Q^{-1}(\lambda) \in \mathcal{B}(\bX, \bY)$. First, we note that $\Range(Q^{-1}(\lambda)) = \dom(Q(\lambda)) = \dom(A) \subseteq \bY$.
	Since $\bX$ and $\bY$ are Banach spaces, it is enough to show that $Q^{-1}(\lambda)$ has closed graph
	in $\bX \times \bY$. Assume $\{x_n\}_{n \geq 1} \subset \bX$, $x \in\bX$, $y \in \bY$ are such that
	\begin{equation}\label{3.2.10}
		x_n \xrightarrow{\bX} x, \; Q^{-1}(\lambda)x_n \xrightarrow{\bY} y\;\mbox{as}\; n\to\infty.
	\end{equation}
	From \eqref{3.2.10} and since $Q^{-1}(\lambda) \in \mathcal{B}(\bX)$ we have $ Q^{-1}(\lambda)x_n \xrightarrow{\bX} Q^{-1}(\lambda)x$ as $n\to\infty$. Similarly, from \eqref{3.2.10} and since
	$\bY \hookrightarrow\bX$ it follows that $ Q^{-1}(\lambda)x_n \xrightarrow{\bX} y$ as $n\to\infty$. Hence, 
	$Q^{-1}(\lambda)x = y$, proving our claim. 
	
	Since $(\lambda I+B)$ is a closed linear operator, $Q^{-1}(\lambda) \in \mathcal{B}(\bX)$, $\Range(Q^{-1}(\lambda)) = \dom(A) \subseteq \dom(B)$ we infer that $(\lambda I+B)Q^{-1}(\lambda) \in \mathcal{B}(\bX)$. From \eqref{r3.1.1} we obtain $(\lambda I+B)|_\bY \in \mathcal{B}(\bY,\bX)$. Since $Q^{-1}(\lambda) \in \mathcal{B}(\bX, \bY)$ we conclude that $Q^{-1}(\lambda)(\lambda I+B)|_\bY \in \mathcal{B}(\bY)$. Finally, 
	since $(\lambda I+B)|_\bY \in \mathcal{B}(\bY,\bX)$ and $(\lambda I+B)Q^{-1}(\lambda) \in \mathcal{B}(\bX)$ we infer that $\bigl((\lambda I+B)Q^{-1}(\lambda)(\lambda I+B) - I\bigr)|_\bY\in \mathcal{B}(\bY,\bX)$. From \eqref{3.2.5} we conclude that 
	$(\lambda I_{\bY\times\bX} - G)^{-1}\in\mathcal{B}(\bY\times\bX)$, proving the lemma.
\end{proof}
Next,  we summarize the basic estimates for the block-matrix representation of the $C_0$-group $\{\cT(t)\}_{t\in\RR}$. Indeed, from (ABY3) we have  
\begin{equation}\label{cal-T}
	\cT(t) = \begin{bmatrix} T_{11}(t) & T_{12}(t) \\ T_{21}(t) & T_{22}(t) \end{bmatrix}, \; t \in \mathbb{R},
\end{equation}
where,
\begin{equation}\label{cal-T-boundedness}
	T_{11}(t) \in \mathcal{B}(\bY),\; T_{22}(t) \in \mathcal{B}(\bX),\;
	T_{12}(t) \in \mathcal{B}(\bX,\bY),\; T_{21}(t) \in \mathcal{B}(\bY,\bX)\;\mbox{for any}\; t\in\RR,
\end{equation}
\begin{equation}\label{cal-T-continuity}
	T_{22}(\cdot)x,\; T_{21}(\cdot)y \in \mathcal{C}(\mathbb{R},\bX),\;
	T_{11}(\cdot)y,\; T_{12}(\cdot)x \in \mathcal{C}(\mathbb{R},\bY)\;\mbox{for any}\; t\in\RR,\, x\in\bX,\, y\in\bY,
\end{equation}
\begin{equation}\label{growth-cal-T}
	\|\cT(t)\|_{\mathcal{B}(\bY\times\bX)} \leq M e^{\omega|t|}, \quad t \in \mathbb{R}.
\end{equation}
From \eqref{cal-T} and \eqref{growth-cal-T} we immediately obtain
\begin{align}\label{growth-Tij}
	\|T_{11}(t)y\|_\bY &\leq M e^{\omega|t|}\|y\|_\bY,\; 
	\|T_{12}(t)x\|_\bY\leq M e^{\omega|t|}\|x\|, \nonumber\\
	\|T_{21}(t)y\| &\leq M e^{\omega|t|}\|y\|_\bY,\;
	\|T_{22}(t)x\| \leq M e^{\omega|t|}\|x\| 
\end{align}
for any $t \in \mathbb{R}$, $x \in \bX$, $y \in \bY$.
\begin{lemma}\label{l3.3}
	Assume \textnormal{(ABY1)--(ABY3)}. Then, $Q(\lambda) = \lambda^2 I + 2\lambda B - A$ is invertible with bounded inverse whenever $|\operatorname{Re}\lambda| > \omega$. Moreover, there exists $\widetilde{M} > 0$ such that
	\begin{equation}\label{Q-inverse-X}
		\|\lambda Q^{-1}(\lambda)x\|+\|BQ^{-1}(\lambda)x\|+\|Q^{-1}(\lambda)x\|_\bY\leq \frac{\widetilde{M}}{|\operatorname{Re}\lambda - \omega|}\|x\|,
	\end{equation}
	\begin{equation}\label{Q-inverse-Y}
		\|\lambda Q^{-1}(\lambda)y\|_\bY +\|Q^{-1}(\lambda)By\|_\bY+\|BQ^{-1}(\lambda)By\|+\|(\lambda^2 Q^{-1}(\lambda) - I)y\|\leq \frac{\widetilde{M}}{|\operatorname{Re}\lambda - \omega|}\|y\|_\bY
	\end{equation}
	for any $x \in\bX$, $y \in \bY$, whenever $|\operatorname{Re}\lambda| > \omega$.
\end{lemma}
\begin{proof}
	From \eqref{growth-cal-T} we infer that
	\begin{equation}\label{3.3.1}
		\{\lambda \in \mathbb{C} : |\operatorname{Re}\lambda| > \omega\} \subseteq \rho(G),\; 
		\|R(\lambda, G)\|_{\mathcal{B}(\bY\times\bX)} \leq \frac{M}{|\operatorname{Re}\lambda - \omega|},
		\;\text{whenever }\; |\operatorname{Re}\lambda| > \omega.
	\end{equation}
	By Lemma~\ref{l3.2} and \eqref{3.3.1} it follows that $Q(\lambda)$ is invertible with bounded inverse, whenever $|\operatorname{Re}\lambda| > \omega$ and \eqref{3.2.5} holds. We fix such a $\lambda$ and $x \in \bX$, $y \in \bY$. From \eqref{3.2.5} and \eqref{3.3.1} we obtain 
	\begin{equation}\label{3.3.3}
		\|Q^{-1}(\lambda)x\|_\bY + \|(\lambda I+B)Q^{-1}(\lambda)x\|
		= \left\|R(\lambda,G)\binom{0}{x}\right\|_{\bY\times\bX}
		\leq \frac{M}{|\operatorname{Re}\lambda - \omega|}\|x\|. 
	\end{equation}
It follows from \eqref{r3.1.1} and \eqref{3.3.3} that
	\begin{equation}\label{3.3.4}
		\|BQ^{-1}(\lambda)x\|
		= \|B|_\bY Q^{-1}(\lambda)x\| \leq \|B|_\bY\|\|Q^{-1}(\lambda)x\|_\bY
		\leq \frac{M\|B|_\bY\|}{|\operatorname{Re}\lambda - \omega|}\|x\|,
	\end{equation}
	\begin{equation}\label{3.3.5}
		\|\lambda Q^{-1}(\lambda)x\|
		\leq \|(\lambda I+B)Q^{-1}(\lambda)x\| + \|BQ^{-1}(\lambda)x\|
		\leq \frac{M(1+\|B|_\bY\|)}{|\operatorname{Re}\lambda - \omega|}\|x\|.
	\end{equation}
	Estimate \eqref{Q-inverse-X} follows from \eqref{3.3.3}, \eqref{3.3.4} and \eqref{3.3.5} by choosing $\widetilde{M}>2M(1+\|B|_\bY\|)$. Using again \eqref{3.2.5} and \eqref{3.3.1} we have 
	\begin{equation}\label{3.3.6}
		\|Q^{-1}(\lambda)(\lambda I+B)y\|_\bY + \|(\lambda I+B)Q^{-1}(\lambda)(\lambda I+B)y - y\|
		= \left\|R(\lambda,G)\binom{y}{0}\right\|_{\bY\times \bX}
		\leq \frac{M}{|\operatorname{Re}\lambda - \omega|}\|y\|_\bY. 
	\end{equation}
	From \eqref{r3.1.1} and \eqref{3.3.3} we infer 
	\begin{equation}\label{3.3.7}
		\|Q^{-1}(\lambda)By\|_\bY \leq \frac{M}{|\operatorname{Re}\lambda - \omega|}\|By\|
		= \frac{M}{|\operatorname{Re}\lambda - \omega|}\|B|_\bY y\|
		\leq \frac{M\|B|_\bY\|}{|\operatorname{Re}\lambda - \omega|}\|y\|_\bY.
	\end{equation}
	Similarly, from \eqref{r3.1.1} and \eqref{3.3.4} we obtain
	\begin{equation}\label{3.3.8}
		\|BQ^{-1}(\lambda)By\|
		\leq \frac{M\|B|_\bY\|}{|\operatorname{Re}\lambda - \omega|}\|By\|
		\leq \frac{M\|B|_\bY\|^2}{|\operatorname{Re}\lambda - \omega|}\|y\|_\bY.
	\end{equation}
It follows	by \eqref{3.3.6} and \eqref{3.3.7} that
	\begin{equation}\label{3.3.9}
		\|\lambda Q^{-1}(\lambda)y\|_\bY
		\leq \|Q^{-1}(\lambda)(\lambda I+B)y\|_\bY + \|Q^{-1}(\lambda)By\|_\bY
		\leq \frac{M(1+\|B|_\bY\|)}{|\operatorname{Re}\lambda - \omega|}\|y\|_\bY.
	\end{equation}
	By \eqref{r3.1.1} and \eqref{3.3.9} we estimate 
	\begin{equation}\label{3.3.10}
		\|\lambda Q^{-1}(\lambda)By\|
		\leq \frac{M(1+\|B|_\bY\|)}{|\operatorname{Re}\lambda - \omega|}\|By\|\leq \frac{M(1+\|B|_\bY\|)\|B|_\bY\|}{|\operatorname{Re}\lambda - \omega|}\|y\|_\bY.
	\end{equation}
	From \eqref{r3.1.1} and \eqref{3.3.9} we have
	\begin{equation}\label{3.3.11}
		\|\lambda BQ^{-1}(\lambda)y\|
		= \|B|_\bY(\lambda Q^{-1}(\lambda)y)\| \leq \|B|_\bY\|\|\lambda Q^{-1}(\lambda)y\|_\bY
		\leq \frac{M(1+\|B|_\bY\|)\|B|_\bY\|}{|\operatorname{Re}\lambda - \omega|}\|y\|_\bY. 
	\end{equation}
	Using \eqref{3.3.6}, \eqref{3.3.8}, \eqref{3.3.10} and \eqref{3.3.11} we conclude that 
	\begin{align}\label{3.3.12}
		\|(\lambda^2 Q^{-1}(\lambda) - I)y - y\|
		&\leq \|(\lambda+B)Q^{-1}(\lambda)(\lambda+B)y - y\|
		+ \|\lambda Q^{-1}(\lambda)By\| \nonumber\\
		&\qquad + \|\lambda BQ^{-1}(\lambda)y\| + \|BQ^{-1}(\lambda)By\| \nonumber\\
		&\leq \frac{M}{|\operatorname{Re}\lambda - \omega|}\|y\|_\bY
		\Bigl(1 + 2\|B|_\bY\|(1+\|B|_\bY\|) + \|B|_\bY\|^2\Bigr)
	\end{align}
	The estimate \eqref{Q-inverse-Y} follows from \eqref{3.3.7}, \eqref{3.3.8}, \eqref{3.3.9} and \eqref{3.3.12}
\end{proof}
\noindent\textbf{Classical solutions of the second order damped equation.}
The goal of this subsection is to prove the existence and uniqueness of classical solutions of \eqref{AB}. We start by recalling the definition 
\begin{definition}\label{d3.1}
	A function $u\in\mathcal{C}^2(\RR,\bX)$ is called a classical solution of the initial value problem \eqref{AB} if $u(t)\in\dom(A)$, $u'(t)\in\dom(B)$ for any $t\in\RR$, $Au(\cdot)\in\mathcal{C}(\RR,\bX)$ and $u$ satisfies  \eqref{AB}.
\end{definition}
\begin{lemma}\label{l3.4}
	Assume \textnormal{(ABY1)--(ABY3)}. If $x \in \dom(A)$, $y \in \bY$ and
	$\binom{u}{v} = \cT(\cdot)\binom{x}{y}$, then
	\begin{itemize}
		\item[(i)] $u \in \mathcal{C}^2(\mathbb{R},\bX) \cap \mathcal{C}^1(\mathbb{R},\bY) \cap \mathcal{C}(\mathbb{R},\dom(A))$;
		\item[(ii)] $u'' + 2Bu' = Au$.
	\end{itemize}
\end{lemma}
\begin{proof} First, we recall the notation $L = A + B^2$. 
	Fix $x \in \dom(A)$, $y \in\bY$. From hypothesis (ABY3) we have $\binom{x}{y} \in \dom(G)$, which implies that
	\begin{equation}\label{3.4.1}
		\binom{u}{v} \in \mathcal{C}^1(\mathbb{R}, \bY\times\bX) \cap \mathcal{C}(\mathbb{R},\dom(A)\times \bY),
	\end{equation}
	\begin{equation}\label{3.4.2}
		\begin{cases} u'(t) = -Bu(t) + v(t) \\ v'(t) = Lu(t) - Bv(t) \end{cases}
		\quad \text{for any } t \in \mathbb{R}. 
	\end{equation}
	The smoothness property \eqref{3.4.1} is equivalent to 
	\begin{equation}\label{3.4.3}
		u \in \mathcal{C}^1(\mathbb{R},\bY) \cap \mathcal{C}(\mathbb{R},\dom(A)),\; v \in \mathcal{C}^1(\mathbb{R},\bX) \cap \mathcal{C}(\mathbb{R},\bY). 
	\end{equation}
	From \eqref{r3.1.1}, \eqref{3.4.1} and \eqref{3.4.2} we obtain
	\begin{equation}\label{3.4.4}
		Bu(\cdot) = v - u' \in \mathcal{C}(\mathbb{R},\bY) \subset \mathcal{C}(\mathbb{R},\bX),\;\mbox{thus,}\;B^2u(\cdot) = B|_\bY(Bu(\cdot)) \in \mathcal{C}(\mathbb{R},\bX).
	\end{equation}
	From \eqref{3.4.1} we know that $u(t)\in\dom(A)$ and $v(t)\in\bY$, which implies that $u'(t)\in\bY\subseteq \dom(B)$ for any $t\in\RR$. Since $u' \in \mathcal{C}(\mathbb{R},\bY)$ by \eqref{3.4.3} and $B|_\bY \in \mathcal{B}(\bY,\bX)$ by \eqref{r3.1.1} it follows that  $Bu'(\cdot) = B|_\bY u'(\cdot) \in \mathcal{C}(\mathbb{R},\bX)$. We note that all hypotheses of Lemma~\ref{App1} are satisfied, which proves that $Bu(\cdot) \in \mathcal{C}^1(\mathbb{R},\bX)$ and $(Bu(\cdot))' = Bu'(\cdot)$. We conclude that $u' = -Bu(\cdot) + v \in \mathcal{C}^1(\mathbb{R},\bX)$, hence $u\in \mathcal{C}^2(\mathbb{R},\bX)$. This fact together with \eqref{3.4.3} prove
	assertion (i). From \eqref{3.4.2} and since $(Bu(\cdot))' = Bu'(\cdot)$, as shown above, we see that (ii) holds.
\end{proof}
We are now ready to prove the existence and uniqueness of classical solutions of \eqref{AB}. Recall \eqref{cal-T}.

\begin{theorem}\label{t3.1}
	Assume \textnormal{(ABY1)--(ABY3)}. If $x \in \dom(A)$, $y \in \bY$, then the initial value problem \eqref{AB} 
	has a unique classical solution given by
	\begin{equation}\label{classical-solution}
		u(t) = T_{11}(t)x + T_{12}(t)(y+Bx), \quad t \in \mathbb{R}.
	\end{equation}
\end{theorem}
\begin{proof}
We note that $x \in \dom(A)$ and $y\in\bY$ imply  $y+Bx\in\bY$. By \eqref{cal-T} $u$ is the upper component of the trajectory $\cT(\cdot)\binom{x}{y}$. We denote by $v$ the lower component of the same trajectory. It follows that $u(0)=x$ and $v(0)=y+Bx$, which implies $ u'(0) = -Bu(0) + v(0) =y$. Applying Lemma~\ref{l3.4} we conclude that $u$ is a classical solution of \eqref{AB}. 
	
Assume $u_1, u_2 \in \mathcal{C}^2(\R, \bX)$ are classical solutions of \eqref{AB} and let $\tu=u_1-u_2$. From Definition~\ref{d3.1} we obtain  $\tu\in\mathcal{C}^2(\RR,\bX)$, $A\tu(\cdot)\in\mathcal{C}(\RR,\bX)$ and $\tu(0)=\tu'(0)=0$. It follows that $B\tu'(\cdot) = \tfrac{1}{2}\bigl(A\tu(\cdot) - \tu''(\cdot)\bigr) \in \mathcal{C}(\RR, \bX)$. Let $w : \R \to \dom(A)$ be the function defined by
	\begin{equation}\label{t3.1.1}
		\quad w(t) = \int_0^t\!\!\left(\int_0^s \tu(\tau)\,\rmd\tau\right)\rmd s
		= \int_0^t (t-s)\,\tu(s)\,\rmd s.
	\end{equation}
	Since $u \in \mathcal{C}^2(\RR, \bX)$ and $Au(\cdot) \in \mathcal{C}(\RR, \bX)$, we have
	$u \in \mathcal{C}(\RR, \dom(A))$, thus
	\begin{equation}\label{t3.1.2}
		w \in \mathcal{C}^2(\RR, \dom(A)),\;\;
		w'(t) = \int_0^t \tu(s)\,\rmd s,\quad w''(t) = \tu(t)\;\mbox{for any}\; t\in\RR.
	\end{equation}
	Since $B$ is a closed linear operator and $u'(t) \in \dom(B)$ for any
	$t \in \RR$, $B\tu'(\cdot) \in \mathcal{C}(\R, \bX)$,  and so from \cite[Proposition~1.1.7]{ABHN},  we obtain 
	\begin{equation}\label{t3.1.3}
		\int_0^t B\tu'(s)\,\rmd s = B\int_0^t \tu'(s)\,\rmd s = B\tu(t)\;\text{for any}\; t \in \R.
	\end{equation}
	Since $A \in \mathcal{B}(\dom(A), \bX)$ we have
	\begin{equation}\label{t3.1.4}
		\int_0^t A\tu(s)\,\rmd s
		= \int_0^t A\big|_{\dom(A)}\, \tu(s)\,\rmd s
		= A\int_0^t \tu(s)\,\rmd s
		= Aw'(t)\quad\text{for any }t \in \R.
	\end{equation}
It follows from \eqref{t3.1.3} and \eqref{t3.1.4} that
	\begin{equation}\label{t3.1.5}
		\tu'(t) + B\big|_{\dom(A)}\,\tu(t)= \int_0^t \bigl(\tu''(s) + B\tu'(s)\bigr)\,\rmd s
		= \int_0^t A\tu(s)\,\rmd s
		= Aw'(t)\;\text{for any}\; t \in \R.
	\end{equation}
	Since $B\big|_{\dom(A)} \in \mathcal{B}(\dom(A), \bY)$ and $A \in \mathcal{B}(\dom(A), \bX)$,
	integrating again in \eqref{t3.1.5} we infer
	\begin{align}\label{t3.1.6}
		w''(t) + 2B\,w'(t)
		&= \tu(t) + B\big|_{\dom(A)}\int_0^t \tu(s)\,\rmd s= \int_0^t\!\bigl(\tu'(s) + B\big|_{\dom(A)}\,\tu(s)\bigr)\,\rmd s
		= \int_0^t A\,w'(s)\,\rmd s\nonumber\\
		&= A\int_0^t w'(s)\,\rmd s= Aw(t)\;\text{for any}\;t \in \R.
	\end{align}
	Let $z : \R \to \bY$ be the function defined by 
	$z(t)=w'(t)+Bw(t)$.
	Since $w \in \mathcal{C}^2(\R, \dom(A))$ and $B\big|_{\dom(A)} \in \mathcal{B}(\dom(A), \bY)$ we obtain 
	\begin{equation}\label{t3.1.7}
		z \in \mathcal{C}^1(\R,\bY),\;\mbox{and}\;
		z'(t) = w''(t) + B\big|_{\dom(A)}\,w'(t),\; t \in \RR.
	\end{equation}
	Using \eqref{t3.1.6} and \eqref{t3.1.7} we evaluate
	\begin{align}\label{t3.1.8}
		z'(t)&= w''(t) + B\big|_{\dom(A)}\,w'(t)= w''(t) + Bw'(t)= Aw(t) - Bw'(t)
		\nonumber\\
		&= Aw(t) - B\bigl(z(t) - Bw(t)\bigr)= (A + B^2)\,w(t) - B\,z(t)\;\text{for any}\;t \in \R.
	\end{align}
	From \eqref{t3.1.2} and \eqref{t3.1.7} we conclude that $\begin{pmatrix}w\\z\end{pmatrix} \in \mathcal{C}^1(\R,\dom(A)\times \bY)
	= \mathcal{C}^1(\R,\dom(G))$ and 
	\begin{equation}\label{t3.1.9}
		\begin{pmatrix}w\\z\end{pmatrix}' = G\begin{pmatrix}w\\z\end{pmatrix},\quad
		\begin{pmatrix}w\\z\end{pmatrix} (0)= \begin{pmatrix}0\\0\end{pmatrix}.
	\end{equation}
	Since $G$ generates a $C_0$-group we infer that
	$\begin{pmatrix}w(t)\\z(t)\end{pmatrix} = \begin{pmatrix}0\\0\end{pmatrix}$
	for any $t \in \R$. Hence,
	\begin{equation}\label{t3.1.10}
		\tu(t) = w''(t) = 0\;\text{for any}\; t \in \R,
	\end{equation}
	proving that $u_1 = u_2$.
\end{proof}	
\noindent\textbf{Mild solutions. Generalized Cosine and Sine Families.}
In this subsection we prove the existence and uniqueness of Laplace Transform mild solutions of \eqref{AB}. First, we recall several results concerning the Laplace Transform of $C_0$-groups. 
\begin{remark}\label{r3.8}
Since $\{\cT(t)\}_{t \in \mathbb{R}}$ is a $C_0$-group with the generator $G$, from  \eqref{growth-cal-T} we obtain
\begin{equation}\label{r3.8.1}
		\bigl(\mathcal{L}\,\cT(\cdot)\tbinom{y}{x}\bigr)(\lambda)
		= R(\lambda, G)\tbinom{y}{x}, 
		\bigl(\mathcal{L}\,\cT(-\cdot)\tbinom{y}{x}\bigr)(\lambda)
		= R(\lambda, -G)\tbinom{y}{x} = -R(-\lambda, G)\tbinom{y}{x}, 
\end{equation}
for any $\lambda \in \mathbb{C}_\omega^+$, $\tbinom{y}{x} \in \bY \times \bX$. From \eqref{3.2.5}, \eqref{cal-T} and \eqref{r3.8.1} we infer that
\begin{align}\label{r3.8.2}
		\bigl(\mathcal{L}\,T_{11}(\cdot)y\bigr)(\lambda) &= Q^{-1}(\lambda)(\lambda I+B)y, 
		\bigl(\mathcal{L}\,T_{12}(\cdot)x\bigr)(\lambda) = Q^{-1}(\lambda)x,\\
		\bigl(\mathcal{L}\,T_{21}(\cdot)y\bigr)(\lambda) &= (\lambda I+B)Q^{-1}(\lambda)(\lambda I+B)y - y, 
		\bigl(\mathcal{L}\,T_{22}(\cdot)x\bigr)(\lambda)= (\lambda I+B)Q^{-1}(\lambda)x, \nonumber \\
		\bigl(\mathcal{L}\,T_{11}(-\cdot)y\bigr)(\lambda) &= Q^{-1}(-\lambda)(\lambda I-B)y, 
		\bigl(\mathcal{L}\,T_{12}(-\cdot)x\bigr)(\lambda) = -Q^{-1}(-\lambda)x, \nonumber \\
		\bigl(\mathcal{L}\,T_{21}(-\cdot)y\bigr)(\lambda) &= -(\lambda I-B)Q^{-1}(-\lambda)(\lambda I-B)y + y, 
		\bigl(\mathcal{L}\,T_{22}(-\cdot)x\bigr)(\lambda)= (\lambda I-B)Q^{-1}(-\lambda)x\nonumber 
	\end{align}
	for any $\lambda \in \mathbb{C}_\omega^+$, $x \in \bX$, $y \in \bY$.
\end{remark}
\begin{remark}\label{r3.9}
Taking the Laplace Transform and using \eqref{r3.8.2} one can readily check that
\begin{equation}\label{r3.9.1}
T_{11}(t)y = T_{12}(t)By + T_{22}(t)y - BT_{12}(t)y\;\mbox{for any}\;t\in\RR, y\in\bY. 
\end{equation}
From \eqref{r3.1.1}, \eqref{cal-T-boundedness} and \eqref{cal-T-continuity}
we conclude that the family of strongly continuous operators $\{T_{11}(t)\}_{t\in\mathbb{R}}$ on $\mathcal{B}(\bY)$ can be extended to a strongly continuous family of operators from $\mathcal{B}(\dom(B),\bX)$.
\end{remark}
We are now ready to prove the existence and uniqueness of Laplace Transform mild solutions of \eqref{AB}, showing in the process that each classical solution is a Laplace Transform mild solution.
\begin{theorem}\label{t3.10}
	Assume \textnormal{(ABY1)--(ABY3)}. Then, for each $x \in \dom(B)$ and $y \in \bX$ \eqref{AB} has a unique mild solution given by
	\begin{equation}\label{Laplace-Mild-sol}
		u(t) = T_{22}(t)x - BT_{12}(t)x + T_{12}(t)(2Bx+y). 
	\end{equation}
\end{theorem}
\begin{proof}
First, we show that the function defined by \eqref{Laplace-Mild-sol} is a Laplace Transform mild solution. From \eqref{r3.1.1}, \eqref{cal-T-boundedness} and \eqref{cal-T-continuity} we obtain $u \in \mathcal{C}(\mathbb{R}, \bX)$. Moreover, it follows from \eqref{cal-T-continuity} and \eqref{growth-Tij} that $u \in \mathcal{C}_{0,-\omega}(\mathbb{R},\bX)$. Using \eqref{r3.8.2} we evaluate
\begin{align}\label{t3.10.1}
(\mathcal{L}u)(\lambda)
&= (\lambda I+B)Q^{-1}(\lambda)x - BQ^{-1}(\lambda)x + Q^{-1}(\lambda)(2Bx+y) = Q^{-1}(\lambda)(\lambda x + 2Bx + y);\nonumber\\ 
(\mathcal{L}u(-\cdot))(\lambda)
&= (\lambda I-B)Q^{-1}(-\lambda)x + BQ^{-1}(-\lambda)x - Q^{-1}(-\lambda)(2Bx+y)\nonumber\\&= Q^{-1}(-\lambda)(\lambda x - 2Bx - y)\;\mbox{for any}\;\lambda \in \mathbb{C}_\omega^+.
\end{align}
Since $Q(\pm\lambda)$ is invertible with bounded inverse for any $\lambda \in \mathbb{C}_\omega^+$, from \eqref{t3.10.1} we see that equation \eqref{Def-Lap-Mild} holds true, thus the function defined in \eqref{Laplace-Mild-sol} is a Laplace Transform mild solution of \eqref{AB}.
	
Assume $u \in \mathcal{C}_{0,-\nu_1}(\mathbb{R},X)$ and $v \in \mathcal{C}_{0,-\nu_2}(\mathbb{R},X)$ are Laplace Transform mild solutions of \eqref{AB}. Let $\nu_0 = \max\{\omega, \nu_1, \nu_2\}$. Then $u - v \in \mathcal{C}_{0,-\nu_0}(\mathbb{R},X)$ and
	\begin{equation}\label{t3.10.2}
		Q(\lambda)(\mathcal{L}(u-v))(\lambda) = 0,\; Q(-\lambda)(\mathcal{L}(u(-\cdot)-v(-\cdot)))(\lambda) = 0\;\mbox{for any}\;\lambda \in \mathbb{C}_{\nu_0}^+
	\end{equation}
	Using again Lemma~\ref{l3.3}, in particular the fact that $Q(\pm\lambda)$ is invertible with bounded inverse for any $\lambda \in \mathbb{C}_{\nu_0}^+$, from \eqref{t3.10.2} we infer that $u|_{\mathbb{R}_+} = v|_{\mathbb{R}_+}$ and $u(-\cdot)|_{\mathbb{R}_+} = v(-\cdot)|_{\mathbb{R}_+}$. Hence $u = v$.
\end{proof}
To conclude this section we define the generalized cosine and sine families associated to \eqref{AB}.
\begin{definition}\label{d3.11}
	We define $C_{A,B} : \mathbb{R} \to \mathcal{B}(\dom(B),\bX)$ and $S_{A,B} : \mathbb{R} \to \mathcal{B}(\bX)$ as follows: $\CAB(\cdot)x$ is the unique Laplace Transform mild solution of 
	\begin{equation}\label{IVP-x0} 
		\begin{cases} u'' + 2Bu' = Au, \\ u(0) = x\in\dom(B), u'(0) = 0. \end{cases}
	\end{equation}
	Similarly, $\SAB(\cdot)x$ is the unique Laplace Transform mild solution of 
	\begin{equation}\label{IVP-0y} 
		\begin{cases} u'' + 2Bu' = Au, \\ u(0) =0, u'(0) =x\in\bX. \end{cases}
	\end{equation}
\end{definition}
\begin{remark}\label{r3.12} From Theorem~\ref{t3.10} we have
	\begin{equation}\label{def-CAB}
		\CAB(t)x = T_{22}(t)x - BT_{12}(t)x + 2T_{12}(t)Bx,\;\mbox{for any}\; x \in \dom(B);
	\end{equation}
	\begin{equation}\label{def-SAB}
		\SAB(t)x = T_{12}(t)x,\;\mbox{for any}\;x\in\bX. 
	\end{equation}
	Reformulating Theorem~\ref{t3.10}, \eqref{def-CAB} and \eqref{def-SAB}, we infer that for any $x\in\dom(B)$, $y\in\bX$ \eqref{AB} has a unique Laplace Transform mild solution given by $\CAB(\cdot)x+\SAB(\cdot)y$.
\end{remark}

\section{Properties of the generalized cosine/sine functions.}\label{s4}
In this section we discuss the properties of the generalized cosine and sine families. One of our goals is to prove that many well-known properties of a cosine family and its associated sine family known in the undamped case can be recovered in the case of 
$\{\CAB(t)\}_{t\in\RR}$ and $\{\SAB(t)\}_{t\in\RR}$, defined in \eqref{def-CAB} and \eqref{def-SAB}. Throughout this section we assume Hypothesis (ABY). 
\begin{lemma}\label{l4.1}
	Assume \textnormal{(ABY1)--(ABY3)}. Then, 
	\begin{itemize}
		\item[(i)] $\CAB(t)\bY \subseteq \bY$ for any $t \in \mathbb{R}$;
		\item[(ii)] $\CAB(t)\dom(A) \subseteq \dom(A)$ for any $t \in \mathbb{R}$;
		\item[(iii)] $\SAB(t)\bX \subseteq \bY \subseteq \dom(B)$ for any $t \in \mathbb{R}$;
		\item[(iv)] $\SAB(t)\bY\subseteq\dom(A)$ for any $t \in \mathbb{R}$;
		\item[(v)] $A\SAB(t)\dom(A) \subseteq\bY \subseteq \dom(B)$ for any $t \in \mathbb{R}$.
	\end{itemize}
\end{lemma}
\begin{proof} From \eqref{r3.9.1} and \eqref{def-CAB} we have:
	\begin{equation}\label{4.1.1}
		\CAB(t)y = T_{11}(t)y + T_{12}(t)By\;\text{for any}\; t \in \mathbb{R},\, y \in\bY. 
	\end{equation}
	\noindent Assertion (i) follows from \eqref{cal-T-boundedness} and \eqref{4.1.1}.
	
	\noindent(ii) Fix $x \in \dom(A)$. From (AB2) we have $Bx \in \bY$, hence $\tbinom{x}{Bx} \in \dom(A) \times \bY = \dom(G)$, which implies that $\cT(t)\tbinom{x}{Bx} \in \dom(G) = \dom(A) \times \bY$ for any $t \in \mathbb{R}$. From \eqref{cal-T} and \eqref{4.1.1} we conclude that $\CAB(t)x = T_{11}(t)x + T_{12}(t)Bx \in \dom(A)$ for any $t \in \mathbb{R}$, proving (ii).
	Assertion (iii) follows from \eqref{cal-T-boundedness}, \eqref{def-SAB} and (ABY2).
	
	(iv) Fix $y \in \bY$. Then $\tbinom{0}{y} \in \dom(A) \times \bY = \dom(G)$ and thus
	\begin{equation}\label{4.1.2}
		\tbinom{T_{12}(t)y}{T_{22}(t)y} = \cT(t)\tbinom{0}{y} \in \dom(G) = \dom(A) \times \bY \;\text{for any}\; t \in \mathbb{R}.
	\end{equation}
	Hence, $\SAB(t)y = T_{12}(t)y \in \dom(A)$, proving (iv). 
	
	(v) Again we recall the notation $L = A + B^2$ and fix $x \in \dom(A)$. Then, $\tbinom{0}{x} \in \dom(A) \times \bY = \dom(G)$ and
	\begin{equation}\label{4.1.3}
		G\binom{0}{x}
		= \begin{bmatrix} -B & I \\ L & -B \end{bmatrix}\binom{0}{x}
		= \binom{x}{-Bx} \in \dom(G). 
	\end{equation}
	Hence $\binom{0}{x} \in \dom(G^2)$, which implies 
	\begin{equation}\label{4.1.4}
		\binom{T_{12}(t)x}{T_{22}(t)x} = \cT(t)\binom{0}{x} \in \dom(G^2),\;\mbox{thus}\;
		G\binom{T_{12}(t)x}{T_{22}(t)x} \in \dom(A) \times \bY	\; \text{for any}\; t \in \mathbb{R}. 
	\end{equation}
	It follows that
	\begin{equation}\label{4.1.5}
		LT_{12}(t)x - BT_{22}(t)x \in \bY \;\text{for any}\; t \in \mathbb{R}. 
	\end{equation}
	Plugging \eqref{def-CAB} and \eqref{def-SAB}  in \eqref{4.1.5} we obtain
	\begin{equation}\label{4.1.6}
		A\SAB(t)x - BC_{A,B}(t)x - 2BS_{A,B}(t)Bx \in \bY \;\text{for any}\; t \in \mathbb{R}. 
	\end{equation}
	From (ii), (iv) and hypothesis (ABY1)-(ABY2) we have $Bx \in \bY$, hence
	\begin{equation}\label{4.1.7}
		\CAB(t)x,\; \SAB(t)Bx \in \dom(A), \; \text{thus}\; BC_{A,B}(t)x,\; BS_{A,B}(t)Bx \in \bY\;\text{for any}\; t \in \mathbb{R}. 
	\end{equation}
	assertion (v) follows shortly from \eqref{4.1.6} and \eqref{4.1.7}.
\end{proof}
In the next two lemmas we study the differentiability properties of trajectories of the generalized cosine and sine families. To prove these results we recall
that if $g \in \mathcal{C}_{0,-\omega}(\mathbb{R},\bX)$ and $f : \mathbb{R} \to \bX$ is defined by $f(t) = \int_0^t g(s)\,\rmd s$, then
\begin{equation}\label{Laplace-pm}
	(\mathcal{L}f(-\cdot))(\lambda) = -\lambda^{-1}\mathcal{L}\{g(-\cdot)\}\;\mbox{for any}\;\lambda\in\CC_\omega^+.
\end{equation}
\begin{lemma}\label{l4.2}
	Assume hypotheses \textnormal{(ABY1)--(ABY3)}. Then,
	\begin{itemize}
		\item[(i)] $\SAB(\cdot)x \in \mathcal{C}^1(\mathbb{R},\bX) \cap \mathcal{C}(\mathbb{R},\bY)$ for any $x \in \bX$, and
		\begin{equation}\label{diff-SAB-1}
			\bigl(\SAB(t)x\bigr)' = T_{22}(t)x - B\SAB(t)x\;\mbox{for any}\; t\in\RR, x\in\bX;
		\end{equation}
		\item[(ii)] Formula \eqref{diff-SAB-1} can be also written as 
		\begin{equation}\label{diff-SAB-2}
			\bigl(\SAB(t)x\bigr)' = \CAB(t)x - 2\SAB(t)Bx \;\mbox{for any}\; t\in\RR, x\in\dom(B);
		\end{equation}
		\item[(iii)] $\SAB(\cdot)y \in \mathcal{C}^2(\mathbb{R},\bX) \cap \mathcal{C}^1(\mathbb{R},\bY) \cap \mathcal{C}(\mathbb{R},\dom(A))$ for any $y \in \bY$, and
		\begin{equation}\label{diff-SAB-3}
			\bigl(\SAB(t)y\bigr)' = \CAB(t)y - 2\SAB(t)By, 
			\bigl(\SAB(t)y\bigr)''= A\SAB(t)y - 2BC_{A,B}(t)y + 4B\SAB(t)By 
		\end{equation}
		for any $t \in \mathbb{R}$, $y \in \bY$.
		\item[(iv)] The generalized sine also satisfies the identity
		\begin{equation}\label{diff-SAB-4}
			\bigl(\SAB(t)x\bigr)'' + 2\bigl(\SAB(t)Bx\bigr)' = \SAB(t)Ax \;\mbox{for any}\; t\in\RR, x\in\dom(A).
		\end{equation}
	\end{itemize}
\end{lemma}
\begin{proof}
	(i) Fix $x \in \bX$. From \eqref{cal-T-continuity} and \eqref{def-SAB} we infer that $\SAB(\cdot)x\in\mathcal{C}(\mathbb{R},\bY)$. Let $f_x : \mathbb{R} \to X$ be defined by
	\begin{equation}\label{4.2.1}
		f_x(t) = \int_0^t \bigl(T_{22}(s)x - B\SAB(s)x\bigr)\,\rmd s, \quad t \in \mathbb{R}. 
	\end{equation}
	From \eqref{r3.1.1}, \eqref{cal-T-continuity} and \eqref{def-SAB} we infer $T_{22}(\cdot)x,\, B\SAB(\cdot)x \in \mathcal{C}(\mathbb{R},\bX)$, which implies 
	\begin{equation}\label{4.2.2}
		f_x \in \mathcal{C}^1(\mathbb{R},\bX) \; \text{and} \; f_x'(t) = T_{22}(t)x - B\SAB(t)x\;\text{for any}\; t\in\RR. 
	\end{equation}
	Using \eqref{r3.8.2} and \eqref{def-SAB} we evaluate
	\begin{equation}\label{4.2.3}
		(\mathcal{L}f_x)(\lambda)
		= \lambda^{-1}\mathcal{L}\bigl(T_{22}(\cdot)x - B\SAB(\cdot)x\bigr)(\lambda)
		=\lambda^{-1}\bigl((\lambda I+B)Q^{-1}(\lambda)x - BQ^{-1}(\lambda)x\bigr)
		= Q^{-1}(\lambda)x
	\end{equation}
	for any $\lambda \in \mathbb{C}_\omega^+$. Using again \eqref{r3.1.1},  \eqref{r3.8.2}, \eqref{def-SAB} and \eqref{Laplace-pm} we obtain 
	\begin{align}\label{4.2.4}
		(\mathcal{L}f_x(-\cdot))(\lambda)
		&=-\lambda^{-1}\mathcal{L}\bigl(T_{22}(-\cdot)x - B\SAB(-\cdot)x\bigr)(\lambda)
		\nonumber\\
		&=-\lambda^{-1}\bigl((\lambda-B)Q^{-1}(-\lambda)x - B(-Q^{-1}(-\lambda)x)\bigr)=-Q^{-1}(-\lambda)x\;\mbox{for any}\;\lambda \in \mathbb{C}_\omega^+.
	\end{align}
	From \eqref{r3.8.2}, \eqref{def-SAB}, \eqref{4.2.3} and \eqref{4.2.4} we conclude that $f_x = \SAB(\cdot)x$. Assertion (i) follows shortly from \eqref{4.2.2}.
	
	(ii) Identity \eqref{diff-SAB-2} follows from (i) and \eqref{def-CAB}.
	
	(iii) Fix $y\in\bY$. By Theorem~\ref{t3.1} $\SAB(\cdot)y = T_{12}(\cdot)y \in \mathcal{C}^2(\mathbb{R},\bX) \cap \mathcal{C}^1(\mathbb{R},\bY) \cap \mathcal{C}(\mathbb{R},\dom(A))$.
	Identity \eqref{diff-SAB-3} follows from \eqref{diff-SAB-2} since $\bY \subseteq \dom(B)$. Moreover, from \eqref{diff-SAB-2} and since $\SAB(\cdot)y$ is a classical solution of \eqref{IVP-0y} we obtain
	\begin{align}\label{4.2.5}
		\bigl(\SAB(t)y\bigr)''
		&= A\SAB(t)y - 2B\bigl(\SAB(t)y\bigr)'= A\SAB(t)y - 2B\bigl(\CAB(t)y - 2\SAB(t)By\bigr) \nonumber\\
		&= A\SAB(t)y - 2BC_{A,B}(t)y + 4B\SAB(t)By \;\;\mbox{for any}\;t\in\RR,
	\end{align}
	proving (iii).
	
	iv) Fix $x \in \dom(A)$. Since $\SAB(0)=0$ and $\SAB'(0)=I$ (in the strong sense), by taking Laplace Transform we have
	\begin{align}\label{4.2.6}
		\mathcal{L}\Bigl((\SAB(\cdot)x)''&+ 2(\SAB(\cdot)Bx)' - \SAB(\cdot)Ax\Bigr)(\lambda)
		= \lambda^2\mathcal{L}(\SAB(\cdot)x)(\lambda) - \lambda\SAB(0)x- \SAB'(0)x \notag\\
		&\qquad\qquad + 2\lambda\mathcal{L}(\SAB(\cdot)Bx)(\lambda)-\SAB(0)Bx - Q^{-1}(\lambda)Ax \notag\\
		&= \lambda^2 Q^{-1}(\lambda)x - x + 2\lambda Q^{-1}(\lambda)Bx - Q^{-1}(\lambda)Ax \notag\\
		&= Q^{-1}(\lambda)(\lambda^2 I+ 2\lambda B - A)x - x = Q^{-1}(\lambda)Q(\lambda)x - x = 0\;\mbox{for any}\;\lambda \in \mathbb{C}_\omega^+. 
	\end{align}
	Similarly from \eqref{r3.8.2} we obtain
	\begin{align}\label{4.2.7}
		\mathcal{L}\Bigl((\SAB(\cdot)x)''(-\cdot)&+ 2(\SAB(\cdot)Bx)'(-\cdot) - \SAB(-\cdot)Ax\Bigr)(\lambda)\notag\\
		&= \mathcal{L}((\SAB(-\cdot)x)'')(\lambda)- 2\mathcal{L}((\SAB(-\cdot)Bx)')(\lambda) + Q^{-1}(-\lambda)Ax\notag\\ 
		&= \lambda^2\mathcal{L}(\SAB(-\cdot)x)(\lambda) - \lambda\SAB(0)x + \SAB'(0)x \notag\\
		&\qquad\qquad - 2\Bigl(\lambda\mathcal{L}(\SAB(-\cdot)Bx)(\lambda) - \SAB(0)Bx\Bigr) + Q^{-1}(-\lambda)Ax \notag\\
		&= -\lambda^2 Q^{-1}(-\lambda)x + x + 2\lambda Q^{-1}(-\lambda)Bx + Q^{-1}(-\lambda)Ax \notag\\
		&= x - Q^{-1}(-\lambda)(\lambda^2 I - 2\lambda B - A)x = x - Q^{-1}(-\lambda)Q(-\lambda)x = 0
	\end{align}
	for any $\lambda \in \mathbb{C}_\omega^+$. The identity \eqref{diff-SAB-4} follows from \eqref{4.2.6} and \eqref{4.2.7}.
\end{proof}
\begin{lemma}\label{l4.3}
	Assume hypotheses \textnormal{(ABY1)--(ABY3)}. Then,
	\begin{itemize}
		\item[(i)] $\CAB(\cdot)y \in \mathcal{C}^1(\mathbb{R},\bX) \cap \mathcal{C}(\mathbb{R},\bY)$ for any $y \in \bY$, and
		\begin{equation}\label{diff-CAB-1}
			\bigl(\CAB(t)y\bigr)' = A\SAB(t)y - 2BC_{A,B}(t)y + 4B\SAB(t)By + 2\bigl(\SAB(t)By\bigr)' 
		\end{equation}
		for any $t \in \mathbb{R}$, $y \in \bY$;
		\item[(ii)] $\CAB(\cdot)x \in \mathcal{C}^2(\mathbb{R},\bX) \cap \mathcal{C}^1(\mathbb{R},\bY) \cap \mathcal{C}(\mathbb{R},\dom(A))$ for any $x \in \dom(A)$, and
		\begin{equation}\label{diff-CAB-2}
			\bigl(\CAB(t)x\bigr)' = A\SAB(t)x - 2BC_{A,B}(t)x + 4BS_{A,B}(t)Bx + 2\CAB(t)Bx - 4\SAB(t)B^2x 
		\end{equation}
		\begin{align}\label{diff-CAB-3}
			\bigl(\CAB(t)x\bigr)'' &= (A+4B^2)\CAB(t)x + 4BC_{A,B}(t)Bx + 8BS_{A,B}(t)B^2x\nonumber\\& - 8B^2\SAB(t)Bx - 2BA\SAB(t)x\;\mbox{for any}\;t\in\RR, x\in\dom(A). 
		\end{align}
	\end{itemize}
\end{lemma}
\begin{proof}
(i) Fix $y \in\bY$. From Lemma~\ref{l4.2} (i) and (iii), we infer that $\CAB(\cdot)y = (\SAB(\cdot)y)' + 2\SAB(\cdot)By \in \mathcal{C}^1(\mathbb{R},\bX) \cap \mathcal{C}(\mathbb{R},\bY)$. Formula \eqref{diff-CAB-1} follows from \eqref{diff-SAB-3}.
	
(ii) It follows from Theorem~\ref{t3.1}  that $\CAB(\cdot)x \in \mathcal{C}^2(\mathbb{R},\bX) \cap \mathcal{C}^1(\mathbb{R},\bY) \cap \mathcal{C}(\mathbb{R},\dom(A))$ for any $x \in \dom(A)$. Formula \eqref{diff-CAB-2} follows from \eqref{diff-CAB-1} and \eqref{diff-SAB-3} with $y = Bx \in \bY$, for each $x \in \dom(A)$. Formula \eqref{diff-CAB-3} follows from \eqref{diff-CAB-2} since $\CAB(\cdot)x$ is a classical solution of \eqref{IVP-x0}.
\end{proof}
To simplify our notations we introduce $C_{A,B}' : \mathbb{R} \to \mathcal{B}(\bY,\bX)$, $S_{A,B}' : \mathbb{R} \to \mathcal{B}(\bX)$ by
\begin{align}\label{C-S-AB-deriv}
	C_{A,B}'(t)y &= \bigl(\CAB(t)y\bigr)', \; t \in \mathbb{R},\, y \in \bY, \\
	S_{A,B}'(t)x &= \bigl(\SAB(t)x\bigr)', \; t \in \mathbb{R},\, x \in \bX. 
\end{align}
\begin{remark}\label{r4.4}
	Assume \textnormal{(ABY1)--(ABY3)}. Then, 
	\begin{itemize}
		\item[(i)] From hypothesis (ABY3), Lemma~\ref{l4.1} and \eqref{diff-SAB-2} we infer that 
		\begin{equation}\label{r4.4.1}
			\SAB'(t)\bY\subseteq\bY \subseteq \dom(B)\;\mbox{and}\; \SAB'(t)\dom(A) \subseteq \dom(A)\;\mbox{for any}\; t \in \mathbb{R};
		\end{equation}	
		\item[(ii)] Similarly, from (ABY3), Lemma~\ref{l4.1} and \eqref{diff-SAB-2} we have
		\begin{equation}\label{r4.4.2}
			\CAB'(t)\dom(A) \subseteq \bY\;\mbox{for any}\; t \in \mathbb{R}.
		\end{equation}	
	\end{itemize}	
\end{remark}	
In the next lemma we express the blocks $T_{ij}(\cdot)$, $i,j = 1,2$ from the block decomposition of the $C_0$-group $\{\cT(t)\}$ given in \eqref{cal-T} in terms of the generalized sine and cosine families and their derivatives.
\begin{lemma}\label{4.4}
	Assume \textnormal{(ABY1)--(ABY3)}. Then,
	\begin{equation}\label{cT-row 1}
		T_{11}(t)y= \CAB(t)y - \SAB(t)By,\, T_{12}(t)x = \SAB(t)x \;\text{for any}\; t \in \mathbb{R},\, x\in\bX,\, y \in \bY;
	\end{equation}
	\begin{equation}\label{T21-Y}
		T_{21}(t)y= A\SAB(t)y - BC_{A,B}(t)y + 3B\SAB(t)By + (\SAB(t)By)' \; \text{for any}\; t \in \mathbb{R},\, y \in \bY;
	\end{equation}
	\begin{equation}\label{T21-X}
		T_{21}(t)x= A\SAB(t)x - BC_{A,B}(t)x + 3B\SAB(t)Bx 
		+ \CAB(t)Bx - 2\SAB(t)B^2x
	\end{equation}
	for any $t \in \mathbb{R}$, $x \in \dom(A)$;
	\begin{equation}\label{T22-X}
		T_{22}(t)x = (\SAB(t)x)' + 3\SAB(t)x \;\text{for any}\; t \in \mathbb{R},\, x\in\bX;
	\end{equation}
	\begin{align}\label{T22-Y}
		T_{22}(t)y &= \CAB(t)y + B\SAB(t)y - 2\SAB(t)By 
		\;\text{for any}\;  t \in \mathbb{R},\, y \in \bY. 
	\end{align}
\end{lemma}
\begin{proof}
Identities \eqref{cT-row 1}	 follow from \eqref{r3.9.1}, \eqref{def-CAB} and \eqref{def-SAB}. Identity \eqref{T22-X}
follows from \eqref{diff-SAB-1}. Identity \eqref{T22-Y} follows from \eqref{T22-X} and \eqref{diff-SAB-3}. Next, we prove \eqref{T21-Y}. We recall that by Lemma~\ref{l3.2} and \eqref{3.3.1}, $Q(\lambda)$ is invertible with bounded inverse whenever $|\operatorname{Re}\lambda| > \omega$.
	Moreover,
	\begin{equation}\label{4.4.1}
		AQ^{-1}(\lambda) = \lambda^2 Q^{-1}(\lambda) - I + 2\lambda BQ^{-1}(\lambda)\;\mbox{whenever}\; \pm\lambda \in \mathbb{C}_\omega^+
	\end{equation}
	We claim that 
	\begin{equation}\label{4.4.2}
		T_{21}(t)y - T_{22}(t)By = A\SAB(t)y - BC_{A,B}(t)y + 2B\SAB(t)By\; \text{for any}\; t \in \mathbb{R},\, y \in \bY;
	\end{equation}
	Indeed, taking Laplace transform and using \eqref{r3.8.2} and \eqref{4.4.1}, we evaluate 
	\begin{align}\label{4.4.3}
		\mathcal{L}\big(T_{21}(\cdot)y - T_{22}(\cdot)By\big)(\lambda)
		&= (\lambda I+B)Q^{-1}(\lambda)(\lambda I+B)y - y - (\lambda I+B)Q^{-1}(\lambda)By \notag\\
		&= (\lambda I+B)Q^{-1}(\lambda)(\lambda y + By - By) - y
		= \lambda(\lambda I+B)Q^{-1}(\lambda)y - y \notag\\
		&= \lambda^2 Q^{-1}(\lambda)y - y + \lambda BQ^{-1}(\lambda)y,
	\end{align}
	\begin{align}\label{4.4.4}
		\mathcal{L}\big(A\SAB(\cdot)y &- BC_{A,B}(\cdot)y + 2B\SAB(\cdot)By\big)(\lambda)
		= AQ^{-1}(\lambda)y - BQ^{-1}(\lambda)(\lambda I+2B)y + 2BQ^{-1}(\lambda)By \notag\\
		&= \lambda^2 Q^{-1}(\lambda)y - y + 2\lambda BQ^{-1}(\lambda)y - \lambda BQ^{-1}(\lambda)y
		- 2BQ^{-1}(\lambda)By + 2BQ^{-1}(\lambda)By \notag\\
		&= \lambda^2 Q^{-1}(\lambda)y - y + \lambda BQ^{-1}(\lambda)y, 
	\end{align}
	\begin{align}\label{4.4.5}
		\mathcal{L}\big(T_{21}(-\cdot)y - T_{22}(-\cdot)By\big)(\lambda)
		&= -(\lambda I-B)Q^{-1}(-\lambda)(\lambda I-B)y + y + (\lambda I-B)Q^{-1}(-\lambda)By \notag\\
		&= -(\lambda I-B)Q^{-1}(-\lambda)(\lambda y - By + By) + y
		= -\lambda(\lambda I-B)Q^{-1}(-\lambda)y + y \notag\\
		&= -\lambda^2 Q^{-1}(-\lambda)y + y + \lambda BQ^{-1}(-\lambda)y, 
	\end{align}
	\begin{align}\label{4.4.6}
		\mathcal{L}\big(A\SAB(-\cdot)y &- BC_{A,B}(-\cdot)y + 2B\SAB(-\cdot)By\big)(\lambda)\notag\\
		&= -AQ^{-1}(-\lambda)y - BQ^{-1}(-\lambda)(\lambda I-2B)y - 2BQ^{-1}(-\lambda)By \notag\\
		&= -\lambda^2 Q^{-1}(-\lambda)y + y + 2\lambda BQ^{-1}(-\lambda)y
		- \lambda BQ^{-1}(-\lambda)y + 2BQ^{-1}(-\lambda)By \notag\\
		&\quad - 2BQ^{-1}(-\lambda)By
		- 2BQ^{-1}(-\lambda)By \notag\\
		&= -\lambda^2 Q^{-1}(-\lambda)y + y + \lambda BQ^{-1}(-\lambda)y 
	\end{align}
	for any $\lambda \in \mathbb{C}_\omega^+$, $y \in \bY$. Claim \eqref{4.4.2} follows from \eqref{4.4.3}--\eqref{4.4.6}. Identity \eqref{T21-Y} follows from \eqref{T22-X} and \eqref{4.4.2}. Identity \eqref{T21-X} follows from \eqref{T21-Y} and \eqref{diff-SAB-2}.
\end{proof}	
In the next lemma we collect several exponential growth estimates for the generalized cosine and sine families and their first order derivatives. 
\begin{lemma}\label{l4.5}
	Assume \textnormal{(ABY1)--(ABY3)}. Then there exists $\overline{M} > 0$ such that
	\begin{align}
		\|\SAB(t)x\|_\bY &\leq \overline{M}e^{\omega|t|}\|x\|\;\;\text{for any } t \in \mathbb{R},\, x \in \bX,\label{SAB-XY} \\
		\|B\SAB(t)x\| &\leq \overline{M}e^{\omega|t|}\|x\| \;\;\text{for any } t \in \mathbb{R},\, x \in \bX, \label{BSAB-XX}\\
		\|B\SAB(t)By\| &\leq \overline{M}e^{\omega|t|}\|y\|_\bY \;\;\text{for any } t \in \mathbb{R},\, y \in \bY, \label{BSABB-YX}\\
		\|A\SAB(t)y\| &\leq \overline{M}e^{\omega|t|}\|y\|_\bY\;\;\text{for any } t \in \mathbb{R},\, y \in \bY, \label{ASAB-YX}\\
		\|S_{A,B}'(t)x\| &\leq \overline{M}e^{\omega|t|}\|x\|\;\;\text{for any } t \in \mathbb{R},\, x \in \bX, \label{SAB-prime-XX} \\
		\|S_{A,B}'(t)y\|_\bY &\leq \overline{M}e^{\omega|t|}\|y\|_\bY\;\;\text{for any } t \in \mathbb{R},\, y \in \bY, \label{SAB-prime-YY} \\
		\|\CAB(t)x\| &\leq \overline{M}e^{\omega|t|}\|x\|_{\dom(B)}\;\;\text{for any } t \in \mathbb{R},\, x \in \dom(B), \label{CAB-XX}\\
		\|\CAB(t)y\|_\bY &\leq \overline{M}e^{\omega|t|}\|y\|_\bY\;\;\text{for any } t \in \mathbb{R},\, y \in \bY, \label{CAB-YY}\\
		\|B\CAB(t)y\| &\leq \overline{M}e^{\omega|t|}\|y\|_\bY\;\;\text{for any } t \in \mathbb{R},\, y \in \bY, \label{BCAB-YX}\\
		\|C_{A,B}'(t)y\| &\leq \overline{M}e^{\omega|t|}\|y\|_\bY\;\;\text{for any } t \in \mathbb{R},\, y \in \bY\label{CAB-prime-YX},\\
		\|C_{A,B}'(t)x\|_\bY &\leq Me^{\omega|t|}\|Ax\|, \;\;\text{for any } t \in \mathbb{R}, x \in \dom(A)\label{CAB-prime-domainAY},\\
		\|A\CAB(t)x\| &\leq \overline{M}e^{\omega|t|}\|Ax\|\;\;\text{for any } t \in \mathbb{R}, x \in \dom(A)\label{ACAB-prime-domainAX}.
	\end{align}
\end{lemma}
\begin{proof}
Estimate \eqref{SAB-XY}	follows from \eqref{growth-Tij} and \eqref{def-SAB}. Estimate \eqref{BSAB-XX} follows from \eqref{r3.1.1} and \eqref{SAB-XY}. Similarly, estimate \eqref{BSABB-YX} follows \eqref{r3.1.1} and \eqref{BSAB-XX}.
	
From \eqref{growth-Tij} and \eqref{diff-SAB-2} we obtain that 
	\begin{equation}\label{4.5.1}
		\|S_{A,B}'(t)x\| = \|T_{22}(t)x - BT_{12}(t)x\| \leq \|T_{22}(t)x\| + \|B|_\bY T_{12}(t)x\|
		\leq M(1+\|B|_\bY\|)e^{\omega|t|}\|x\|
	\end{equation}
for any $t \in \mathbb{R}$, $x \in\bX$, proving \eqref{SAB-prime-XX}. Next, we recall that there exists $c_1 > 0$ such that
$\|y\| \leq c_1 \|y\|_\bY$ for any $y \in \bY$, since $\bY \hookrightarrow\bX$.
It follows from \eqref{r3.1.1}, \eqref{growth-Tij} and \eqref{def-CAB} that
	\begin{align}\label{4.5.2}
		\|\CAB(t)x\|
		&= \|T_{22}(t)x - BT_{12}(t)x + 2T_{12}(t)Bx\| \leq \|T_{22}(t)x\| + \|B|_\bY T_{12}(t)x\|_\bY + 2\|T_{12}(t)Bx\| \notag\\
		&\leq Me^{\omega|t|}\|x\|(1+\|B|_\bY\|) + 2Mc_1 e^{\omega|t|}\|Bx\| \leq M\max\{1+\|B|_\bY\|,\, c_1\}\,e^{\omega|t|}(\|x\|+\|Bx\|) 
	\end{align}
	for any $t \in \mathbb{R}$, $x \in \dom(B)$, proving \eqref{CAB-XX}. Similarly, using \eqref{r3.1.1}, \eqref{growth-Tij} and \eqref{4.1.1} we estimate
	\begin{align}\label{4.5.3}
		\|\CAB(t)y\|_\bY
		&= \|T_{11}(t)y + T_{12}(t)By\|_\bY \leq \|T_{11}(t)y\|_\bY + \|T_{12}(t)By\|_\bY \notag\\
		&\leq Me^{\omega|t|}\|y\|_\bY + Me^{\omega|t|}\|By\|
		\leq M(1+\|B|_\bY\|)e^{\omega|t|}\|y\|_\bY 
	\end{align}
	for any $t \in \mathbb{R}$, $y \in \bY$, proving \eqref{CAB-YY}. Estimate \eqref{BCAB-YX} follows from \eqref{r3.1.1} and \eqref{CAB-YY}. 
	Using again \eqref{r3.1.1}, \eqref{growth-Tij}, \eqref{diff-SAB-2} and \eqref{4.5.3} we infer that
	\begin{align}\label{4.5.4}
		\|S_{A,B}'(t)y\|_\bY&= \|\CAB(t)y - 2\SAB(t)By\|_\bY \leq \|\CAB(t)y\|_\bY + 2\|T_{12}(t)By\|_\bY\notag\\
		&
		\leq M(1+\|B|_\bY\|)e^{\omega|t|}\|y\|_\bY + 2Me^{\omega|t|}\|By\|\leq M(1+3\|B|_\bY\|)e^{\omega|t|}\|y\|_\bY 
	\end{align}
	for any $t \in \mathbb{R}$, $y \in \bY$, proving \eqref{SAB-prime-YY}. 
	From \eqref{r3.1.1}, \eqref{growth-Tij}, \eqref{T21-Y},  \eqref{BSABB-YX}, \eqref{SAB-prime-XX} and \eqref{BCAB-YX} we conclude that
	\begin{align}\label{4.5.5}
		\|A\SAB(t)y\|&=\|T_{21}(t)y+B\CAB(t)y-3B\SAB(t)By+\SAB'(t)By\|\nonumber\\
		&\leq\|T_{21}(t)y\|+\|B\CAB(t)y\|+3\|B\SAB(t)By\|+\|\SAB'(t)By\|\nonumber\\
		&\leq Me^{\omega|t|}\|y\|_\bY+ 4\overline{M}e^{\omega|t|}\|y\|_\bY+\overline{M}e^{\omega|t|}\|By\|\leq \big(M+4\overline{M}+\overline{M}\|B_\bY\|\big)e^{\omega|t|}\|y\|_\bY
	\end{align}
	for any $t \in \mathbb{R}$, $y \in \bY$. Estimate \eqref{ASAB-YX} follows by replacing $\overline{M}$ by $M+4\overline{M}+\overline{M}\|B_\bY\|$. Finally, using \eqref{r3.1.1}, \eqref{growth-Tij}, \eqref{T21-Y}, \eqref{BSABB-YX}, \eqref{ASAB-YX}, \eqref{SAB-prime-XX} and \eqref{BCAB-YX} we estimate
	\begin{align}\label{4.5.6}
		\|\CAB'(t)y\|&=\|A\SAB(t)y-2B\CAB(t)y+4B\SAB(t)By+2\SAB'(t)By\|\nonumber\\&\leq\|A\SAB(t)y\|+2\|B\CAB(t)y\|+4\|B\SAB(t)By\|+2\|\SAB'(t)By\|\nonumber\\&\leq (6\overline{M}+2\overline{M}\|B_\bY\|)e^{\omega|t|}\|y\|_\bY
	\end{align}
	for any $t \in \mathbb{R}$, $y \in \bY$. Estimate \eqref{CAB-prime-YX} follows by replacing $\overline{M}$ by $6\overline{M}+2\overline{M}\|B_\bY\|$.
	
	Next, we fix $x \in \dom(A)$. Then $Bx \in \bY$, thus $\binom{x}{Bx} \in \dom(A) \times \bY = \dom(G)$. We infer that 
	$\cT(t)\binom{x}{Bx} \in \dom(G) = \dom(A) \times \bY$ for any $t\in\RR$. From \eqref{cal-T}, \eqref{T21-X} and \eqref{T22-Y} we obtain 
	\begin{equation}\label{4.5.7}
		\cT(t)\binom{x}{Bx}
		= \begin{pmatrix} T_{11}(t)x + T_{12}(t)Bx \\ T_{21}(t)x + T_{22}(t)Bx \end{pmatrix}
		= \begin{pmatrix} \CAB(t)x \\ C_{A,B}'(t)x + B\CAB(t)x \end{pmatrix} \;\mbox{for any}\; t \in \mathbb{R},
	\end{equation}	
	and thus
	\begin{align}\label{4.5.8}
		G\cT(t)\binom{x}{Bx}
		&= \begin{bmatrix} -B & I \\ L & -B \end{bmatrix}
		\begin{pmatrix} \CAB(t)x \\ C_{A,B}'(t)x + B\CAB(t)x \end{pmatrix}
		= \begin{pmatrix} C_{A,B}'(t)x \\ A\CAB(t)x - BC_{A,B}'(t)x \end{pmatrix}\nonumber\\
		&  = \begin{pmatrix} C_{A,B}'(t)x \\ A\CAB(t)x - BC_{A,B}'(t)x \end{pmatrix} \;\mbox{for any}\; t \in \mathbb{R}.	
	\end{align}	
	Moreover, one can readily check that 
	\begin{equation}\label{4.5.9}
		G\binom{x}{Bx}
		= \begin{bmatrix} -B & I \\ L & -B \end{bmatrix}\binom{x}{Bx}
		= \binom{0}{Ax}.	
	\end{equation}	
	From \eqref{growth-cal-T}, \eqref{4.5.8} and \eqref{4.5.9} we obtain
	\begin{align}\label{4.5.10}
		\|C_{A,B}'(t)x\|_\bY &+ \|A\CAB(t)x - BC_{A,B}'(t)x\|
		= \|G\cT(t)\tbinom{x}{Bx}\|_{\bY\times\bX} = \|\cT(t)G\tbinom{x}{Bx}\|_{\bY\times\bX} \nonumber\\
		&= \left\|\cT(t)\tbinom{0}{Ax}\right\|_{\bY\times\bX}
		\leq Me^{\omega|t|}\left\|\tbinom{0}{Ax}\right\|_{\bY\times\bX}
		= Me^{\omega|t|}\|Ax\|\;\mbox{for any}\; t \in \mathbb{R}.	
\end{align}
Estimate \eqref{CAB-prime-domainAY} follows immediately from \eqref{4.5.10}. Moreover, from \eqref{r3.1.1} and  \eqref{CAB-prime-domainAY} \eqref{4.5.10} we conclude that 
\begin{align}\label{4.5.11}
		\|A\CAB(t)x\|
		&\leq \|A\CAB(t)x - BC_{A,B}'(t)x\| + \|BC_{A,B}'(t)x\| \leq Me^{\omega|t|}\|Ax\| + \|B|_Y\|\|C_{A,B}'(t)x\|_\bY \nonumber\\
		&\leq Me^{\omega|t|}\|Ax\| + \|B|_\bY\|Me^{\omega|t|}\|Ax\|
		= (1+\|B|_\bY\|)Me^{\omega|t|}\|Ax\|
	\end{align}
	for any $t \in \mathbb{R}$, proving the estimate \eqref{ACAB-prime-domainAX}.
\end{proof}	
In the next lemma we show that $\dom(A)$ is dense in $\bX$ and $\bY$, with the respective norms. 
\begin{lemma}\label{l4.6}
	Assume \textnormal{(ABY1)--(ABY3)}. Then,
	\begin{itemize}
		\item[(i)] $\displaystyle\lim_{t \to 0} \frac{1}{t} \SAB(t)x = x$ in $\bX$ for any $x \in \bX$;
		\item[(ii)] $\displaystyle\lim_{t \to 0} \frac{1}{t} \SAB(t)y = y$ in $\bY$ for any $y \in \bY$;
		\item[(iii)] $\displaystyle\lim_{t \to 0} \frac{1}{t^m} \SAB^m(t)x = x$ in $\bX$ for any $x \in \bX$ and $m\in\NN$;
		\item[(iv)] $\dom(A)$ is dense in $\bX$ in $\bX$-norm;
		\item[(v)] $\dom(A)$ is dense in $\bY$ in $\bY$-norm.
	\end{itemize}
\end{lemma}
\begin{proof}
Assertion (i) follows from Lemma~\ref{l4.2}(i) since $\SAB(0)x = T_{22}(0)x = x$ for any $x\in\bX$. Assertion (ii) follows from Lemma~\ref{l4.2}(ii). From Lemma~\ref{l4.2}(i) and \eqref{SAB-prime-XX} we have
\begin{align}\label{4.6.1}
		\left\| \frac{1}{t} \SAB(t) x \right\|
		&= \left\| \frac{1}{t} \int_0^t \SAB'(s) x \, \rmd s \right\|
		\leq \frac{1}{|t|} \int_0^{|t|} \|\SAB'(s) x\| \, \rmd s\nonumber\\
		&\leq \frac{\overline{M}}{|t|} \left( \int_0^{|t|} e^{\omega s} \, \rmd s \right) \|x\|= \overline{M} \frac{e^{\omega|t|} - 1}{\omega|t|} \|x\|
		\; \text{for any}\; t \neq 0,\, x \in \bX.
	\end{align}
	It follows that
	\begin{align}\label{4.6.2}
		\left\| \frac{1}{t^{m+1}} \SAB^{m+1}(t) x - x \right\|&\leq\left\| \frac{1}{t^{m+1}} \SAB^{m+1}(t) x-\frac{1}{t} \SAB(t) x\right\|+\left\| \frac{1}{t} \SAB(t) x - x \right\|\nonumber\\
		&  \leq \left\| \frac{1}{t} \SAB(t) \Big( \tfrac{1}{t^m}\SAB^m(t)x - x \Big) \right\|
		+ \left\| \frac{1}{t} \SAB(t) x - x \right\|\nonumber\\
		& \leq \left( \overline{M} \cdot \frac{e^{\omega|t|}-1}{\omega|t|} + 1 \right)\left\|\tfrac{1}{t^m}\SAB^m(t)x - x \right\|+
		\left\| \frac{1}{t} \SAB(t) x - x \right\|
	\end{align}
	for any $t\ne 0$, $x \in\bX$. Assertion (iii) follows from (i) and \eqref{4.6.2} by induction. From Lemma~\ref{l4.1}(iii) and (iv),  we obtain $\SAB^2(t)x = \SAB(t)\SAB(t)x \in \dom(A)$ for any $t \in \mathbb{R}$, $x \in\bX$. Assertion (iv) follows shortly from (iii). Assertion (v) follows from (ii) and Lemma~\ref{l4.1}(i).
\end{proof}
In the case of cosine families and their associated sine families it is well-known that by integrating their trajectories we obtain smoother functions, see, e.g., \cite[Proposition 3.14.5]{ABHN}. In the next lemma we aim to prove a similar result for the case of the generalized cosine and sine $\{\CAB(t)\}_{t\in\RR}$ and $\{\SAB(t)\}_{t\in\RR}$ given in \eqref{def-CAB} and \eqref{def-SAB}, respectively.
\begin{lemma}\label{l4.7}
	Assume \textnormal{(ABY1)--(ABY3)}. Then,
	\begin{itemize}
		\item[(i)] $\displaystyle\int_0^t \SAB(s)x \, \rmd s \in \dom(A)$ and
		\begin{equation}\label{A-int-SAB}
			A \int_0^t \SAB(s)x \, \rmd s = \SAB'(t)x + 2B\SAB(t)x-x\;\mbox{for any}\;t\in\RR,\,x\in\bX;
		\end{equation}
		\item[(ii)] $\displaystyle\int_0^t \CAB(s)x \, \rmd s \in \bY$ and
		\begin{equation}\label{B-int-CAB}
			B\int_0^t \CAB(s)x \, \rmd s = B\SAB(t)x + 2\int_0^t B\SAB(s)Bx \, \rmd s\;\mbox{for any}\;t\in\RR,\,x\in\dom(B);
		\end{equation}
		\item[(iii)] $\displaystyle\int_0^t \CAB(s)y \, ds \in \dom(A)$ and
		\begin{align}\label{B-int-CAB-caseY}
			&A\int_0^t \CAB(s)y \, \rmd s=A\SAB(t)y+2\SAB'(t)By+4B\SAB(t)By-2By\nonumber\\
			&= \CAB'(t)y + 2BC_{A,B}(t)y -2By \;\mbox{for any}\;t\in\RR,\,x\in\bY;
		\end{align}
		\item[(iv)] $\displaystyle\int_0^t (t-s)\SAB(s)x \, \rmd s \in \dom(A)$ and
		\begin{equation}\label{A-int-t-SAB}
			A\int_0^t (t-s)\SAB(s)x \, \rmd s = \SAB(t)x + 2\int_0^t B\SAB(s)x \, \rmd s - tx\;\mbox{for any}\;t\in\RR,\,x\in\bX;
		\end{equation}
		\item[(v)] $\displaystyle\int_0^t (t-s)\CAB(s)x \, \rmd s \in \dom(A)$ and
		\begin{equation}\label{A-int-t-CAB}
			A\int_0^t (t-s)\CAB(s)x \, \rmd s = \CAB(t)x + 2B\int_0^t \CAB(s)x \, \rmd s - x - 2tBx
		\end{equation}
		for any $t \in \mathbb{R}$, $x \in \dom(B)$.
	\end{itemize}
\end{lemma}
\begin{proof}
	(i) Fix $x \in \bX$. Then $\binom{0}{x} \in \bY \times \bX$, which implies that
	\begin{equation}\label{4.7.1}
		\int_0^t \mathcal{T}(s)\binom{0}{x} \, \rmd s \in \dom(G) = \dom(A) \times \bY,\,\;	G\int_0^t \mathcal{T}(s)\binom{0}{x} \, \rmd s = \mathcal{T}(t)\binom{0}{x} - \binom{0}{x}
	\end{equation}\label{4.7.2}
	for any $t \in \mathbb{R}$. Rewriting this statement using \eqref{cal-T} we have
	\begin{equation}
		\int_0^t \SAB(s)x \, \rmd s \in \dom(A),\;\int_0^t T_{22}(s)x \, \rmd s \in \bY\; \text{for any}\; t \in \mathbb{R}
	\end{equation}
	and
	\begin{equation}\label{4.7.3}	
		\begin{cases}
			\displaystyle-B\int_0^t \SAB(s)x \, \rmd s + \int_0^t T_{22}(s)x \, \rmd s = \SAB(t)x \\
			\displaystyle\int_0^t \SAB(s)x \, \rmd s - B\int_0^t T_{22}(s)x \, \rmd s = T_{22}(t)x - x
		\end{cases}\; \text{for any}\; t \in \mathbb{R}.
	\end{equation} 
	Eliminating $\displaystyle\int_0^t T_{22}(s)x \, \rmd s$ from (13) we obtain
	\begin{equation}\label{4.7.4}
		A\int_0^t \SAB(s)x \, \rmd s = T_{22}(t)x + B\SAB(t)x - x
	\end{equation}
	for any $t \in \mathbb{R}$. The identity \eqref{A-int-SAB} follows from  \eqref{diff-SAB-1} and \eqref{4.7.4}, proving assertion (i).
	
	(ii) From \eqref{diff-SAB-2} we have
	\begin{equation}\label{4.7.5}
		\int_0^t \CAB(s)x \, \rmd s
		= \int_0^t \bigl[\SAB'(s)x + 2\SAB(s)Bx\bigr] \, \rmd s
		= \SAB(t)x + 2\int_0^t \SAB(s)Bx \, \rmd s, 
	\end{equation}
	for any $t \in \mathbb{R}$, $x \in \dom(B)$. Assertion (ii) follows from \eqref{4.7.5}, Lemma~\ref{l4.1}(iii) and $B|_\bY \in \mathcal{B}(\bY,\bX)$.
	
	(iii) From \eqref{4.7.5}, (i) and Lemma~\ref{l4.1}(iv) we infer that
	$\int_0^t \CAB(s)y \, \rmd s \in \dom(A)$ for any $t \in \mathbb{R}$, $y \in \bY$. Identity \eqref{B-int-CAB-caseY} follows from \eqref{diff-CAB-2} and  \eqref{4.7.5}.
	
	(iv) Fix $x \in\bX$. By applying Lemma~\ref{App2} for the trajectory $\cT(\cdot)\binom{0}{x} $ and the function $\varphi:\RR\to\RR$ defined by $\varphi(t)=t$, we obtain
	\begin{equation}\label{4.7.6}
		\int_0^t (t-s)\SAB(s)x \, \rmd s \in \dom(A),\,\int_0^t (t-s)T_{22}(s)x \, \rmd s \in \bY
		\; \text{for any}\; t \in \mathbb{R},
	\end{equation}
	and
	\begin{equation}\label{4.7.7}
		\begin{cases}
			\displaystyle -B\int_0^t (t-s)\SAB(s)x \, \rmd s + \int_0^t (t-s) T_{22}(s)x \, \rmd s
			= \int_0^t \SAB(s)x \, \rmd s \\
			\displaystyle \int_0^t (t-s)\SAB(s)x \, \rmd s - B\int_0^t (t-s)T_{22}(s)x \, \rmd s
			= \int_0^t T_{22}(s)x \, \rmd s - tx
		\end{cases}\; \text{for any}\; t \in \mathbb{R}.
	\end{equation}
	Eliminating $\displaystyle\int_0^t (t-s) T_{22}(s)x \, \rmd s$ from \eqref{4.7.7} we have
	\begin{equation}\label{4.7.8}
		A\int_0^t (t-s)\SAB(s)x \, \rmd s
		= B\int_0^t \SAB(s)x \, \rmd s + \int_0^t T_{22}(s)x \, \rmd s - tx\; \text{for any}\; t \in \mathbb{R}.
	\end{equation}
	From \eqref{diff-SAB-1} and \eqref{4.7.8} we conclude
	\begin{align}\label{4.7.9}
		A\int_0^t (t-s)\SAB(s)x \, \rmd s
		&= B\int_0^t \SAB(s)x \, \rmd s
		+ \int_0^t \bigl(\SAB'(s)x + B\SAB(s)x\bigr) \, \rmd s - tx\nonumber\\
		&= \SAB(t)x + 2\int_0^t B\SAB(s)x \, ds - tx
		\quad \text{for any } t \in \mathbb{R},
	\end{align}
	proving (iv).

	(v) Fix $x \in \dom(B)$. From \eqref{diff-SAB-2}, (i), (ii), (iii) it follows that
	\begin{align}\label{4.7.10}
		\int_0^t &(t-s)\CAB(s)x \, \rmd s
		= \int_0^t (t-s)\SAB'(s)x \, \rmd s + 2\int_0^t (t-s)\SAB'(s)Bx \, \rmd s\nonumber\\
		&= (t-s)\SAB(s)x \Big|_{s=0}^{s=t}+\int_0^t \SAB(s)x \, \rmd s
		+ 2\int_0^t (t-s)\SAB(s)Bx \, \rmd s\nonumber\\
		&=\int_0^t \SAB(s)x \, \rmd s
		+ 2\int_0^t (t-s)\SAB(s)Bx \, \rmd s
		\in \dom(A) \; \text{for any}\; t \in \mathbb{R}.
	\end{align}
	From \eqref{diff-SAB-2}, \eqref{A-int-SAB}, \eqref{B-int-CAB}, \eqref{A-int-t-SAB} and  \eqref{4.7.10} we conclude that
	\begin{align*}
		A\int_0^t &(t-s)\CAB(s)x \, \rmd s
		= A\int_0^t \SAB(s)x \, \rmd s + 2A\int_0^t (t-s)\SAB(s)Bx \, \rmd s \nonumber\\
		&= \SAB'(t)x + 2B\SAB(t)x - x
		+ 2\bigl(\SAB(t)Bx + 2\int_0^t B\SAB(s)Bx \, \rmd s - tBx\bigr) \nonumber\\
		&= \CAB(t)x - 2\SAB(t)Bx + 2B\SAB(t)x
		+ 2\SAB(t)Bx - x - 2tBx
		+ 4\int_0^t B\SAB(s)Bx \, \rmd s \nonumber\\
		&= \CAB(t)x + 2\!\left(B\SAB(t)x + 2\int_0^t B\SAB(s)Bx \, \rmd s\right) - x - 2tBx \nonumber\\
		&= \CAB(t)x + 2B\int_0^t \CAB(s)x \, \rmd s - x - 2tBx, \;\text{for any}\; t \in \mathbb{R},
	\end{align*}
	proving assertion (v).
\end{proof}
We are now ready to prove that the two concepts of mild solutions introduced in Definition~\ref{Def-mild} and Definition~\ref{Def-Int-Mild} are equivalent. To prove this fact it is enough to show that \eqref{AB} has a unique Integral mild solution for any $x\in\dom(B)$ and $y\in\bX$ expressed in the terms of the generalized cosine and sine families. 
We recall the following standard integration fact,
\begin{equation}\label{integral-transform}
	\int_0^t \!\int_0^s (s-\tau)^n f(\tau)\,d\tau\,ds
	= \frac{1}{n+1} \int_0^t (t-s)^{n+1} f(s)\,ds
\end{equation}
for any $t \in \mathbb{R}$ and $f \in \mathcal{C}(\mathbb{R}, \bW)$, where $\bW$ is a Banach space.
\begin{theorem}\label{t4.8}
	Assume hypotheses \textnormal{(ABY1)--(ABY3)}. Then, for each $x \in \dom(B)$ and $y \in\bX$ the initial value problem \eqref{AB} has a unique integral mild solution given by
	\begin{equation}\label{Integral-Mild-C-S-AB}
		u_{x,y}(t) = \CAB(t)x + \SAB(t)y, \; t \in \mathbb{R}.
	\end{equation}
\end{theorem}
\begin{proof} \textit{Existence.} Fix $x \in \dom(B)$, $y \in \bX$ and let $u_{x,y} : \mathbb{R} \to \bX$ be the function defined by \eqref{Integral-Mild-C-S-AB}. From Lemma~\ref{l4.7} (i), (ii), we infer that
	\begin{equation}\label{4.8.1}
		\int_0^t u_{x,y}(s)\,\rmd s
		= \int_0^t \CAB(s)x\,\rmd s
		+\int_0^t \SAB(s)y\,\rmd s
		\in \bY \subseteq \dom(B)\;\text{for any}\; t \in \mathbb{R}.
	\end{equation}
	Moreover, using \eqref{4.7.5}
	\begin{align}\label{4.8.2}
		B\int_0^t u_{x,y}(s)\,\rmd s
		&= B\int_0^t \CAB(s)x\,\rmd s + B\int_0^t \SAB(s)y\,\rmd s\nonumber\\
		&= B\SAB(t)x + 2\int_0^t B\SAB(s)Bx\,\rmd s
		+ \int_0^t B\SAB(s)y\,\rmd s\nonumber\\
		&= B\SAB(t)x + \int_0^t B\SAB(s)(y + 2Bx)\,\rmd s
		\; \text{for any}\; t \in \mathbb{R}.
	\end{align}
	Since $\SAB(\cdot)x,\, \SAB(\cdot)(y+2Bx) \in \mathcal{C}(\mathbb{R},\bY)$, $B|_\bY \in \mathcal{B}(\bY,\bX)$, from \eqref{4.8.2} we conclude that
	\begin{equation}\label{4.8.3}
		t \mapsto B\int_0^t u_{x,y}(s)\,\rmd s \;\in\; \mathcal{C}(\mathbb{R}, \bX).
	\end{equation}
	From Lemma~\ref{l4.7} (iii), (iv) it follows that
	\begin{equation}\label{4.8.4}
		\int_0^t (t-s)u_{x,y}(s)\,\rmd s
		=\int_0^t (t-s)\CAB(s)x\,\rmd s
		+ \int_0^t (t-s)\SAB(s)y\,\rmd s
		\in \dom(A)\; \text{for any}\; t \in \mathbb{R}.
	\end{equation}
	Using \eqref{B-int-CAB}, \eqref{A-int-t-SAB} and \eqref{A-int-t-CAB} we obtain
	\begin{align}\label{4.8.5}
		A\int_0^t &(t-s)u_{x,y}(s)\,\rmd s
		= A\int_0^t (t-s)\CAB(s)x\,\rmd s + A\int_0^t (t-s)\SAB(s)y\,\rmd s\nonumber\\
		&= \CAB(t)x + 2B\int_0^t \CAB(s)x\,\rmd s + \SAB(t)y + 2\int_0^t B\SAB(s)y\,\rmd s - x - 2tBx - ty\nonumber\\
		&= \CAB(t)x + 2\Big(B\SAB(t)x+2\int_0^t B\SAB(s)Bx\,\rmd s \Big)+ \SAB(t)y + 2\int_0^t B\SAB(s)y\,\rmd s\nonumber\\
		&\qquad\qquad\qquad - x - 2tBx - ty\nonumber\\
		&=\CAB(t)x+\SAB(t)y+2B\SAB(t)x+4\int_0^t B\SAB(s)Bx\,\rmd s+2\int_0^t B\SAB(s)y\,\rmd s\nonumber\\
		&\qquad\qquad\qquad - x - 2tBx - ty\nonumber\\
		&=\CAB(t)x+\SAB(t)y+2B\SAB(t)x+2\int_0^t B\SAB(s)(y+2Bx)\,\rmd s-x-t(y+2Bx)
	\end{align}
	for any $t \in \mathbb{R}$. Since $\CAB(\cdot)x \in \mathcal{C}(\mathbb{R},\bX)$, $\SAB(\cdot)z \in \mathcal{C}(\mathbb{R},\bY)$ for any $z \in \bX$, $B|_\bY \in \mathcal{B}(\bY,\bX)$, from \eqref{4.8.5} we obtain that
	\begin{equation}\label{4.8.6}
		t \mapsto A\int_0^t (t-s)u_{x,y}(s)\,\rmd s \;\in\; \mathcal{C}(\mathbb{R}, \bX).
	\end{equation}
	Summarizing, from \eqref{4.8.1} and \eqref{4.8.4} we see that \eqref{domain-cond} holds for $u = u_{x,y}$. Similarly, from \eqref{4.8.3} and \eqref{4.8.6} we have that \eqref{cont-cond} holds for $u = u_{x,y}$.
	Using \eqref{Integral-Mild-C-S-AB}, \eqref{4.8.2}, \eqref{4.8.5} we compute
	\begin{align}\label{4.8.7}
		u_{x,y}(t)&+ 2B\int_0^t u_{x,y}(s)\,\rmd s - A\int_0^t (t-s)u_{x,y}(s)\,\rmd s\nonumber\\
		&= \CAB(t)x + \SAB(t)y
		+ 2\Big(B\SAB(t)x + \int_0^t B\SAB(s)(y+2Bx)\,ds\Big)\nonumber\\
		&\quad - \CAB(t)x - \SAB(t)y
		- 2BS_{A,B}(t)x
		- 2\int_0^t B\SAB(s)(y+2Bx)\,ds + x + t(y+2Bx)\nonumber\\
		&= x + t(y + 2Bx) \; \text{for any}\; t \in \mathbb{R}.
	\end{align}
	Hence, $u_{x,y}$ satisfies \eqref{Int-Mild}, which proves that $u_{x,y}$ is an integral mild solution of \eqref{AB}.
	
	\noindent\textbf{Uniqueness.} Assume that $u_1, u_2 : \mathbb{R} \to\bX$ are two integral mild solutions of \eqref{AB}. Set $u_0 = u_1 - u_2$. Then $u_0$ is an integral mild solution of \eqref{AB} with $u_0(0)=u_0'(0)=0$. Thus $u_0$ satisfies \eqref{domain-cond}, \eqref{cont-cond} and
	\begin{equation}\label{4.8.8}
		u_0(t) + 2B\int_0^t u_0(s)\,\rmd s = A\int_0^t (t-s)u_0(s)\,\rmd s \; \text{for any}\; t \in \mathbb{R}.
	\end{equation}
	Let $v_0, w_0 : \mathbb{R} \to \dom(A)$ be the functions defined by
	\begin{equation}\label{4.8.9}
		v_0(t) = \int_0^t (t-s)u_0(s)\,\rmd s = \int_0^t\!\!\left(\int_0^s u_0(\tau)\,\rmd\tau\right)\rmd s,\;
		w_0(t) = \int_0^t (t-s)v_0(s)\,\rmd s = \int_0^t\!\!\int_0^s v_0(\tau)\,\rmd\tau\,\rmd s.
	\end{equation}
	Since $v_0 \in \mathcal{C}(\mathbb{R},\bX)$ and $Av_0(\cdot) \in \mathcal{C}(\mathbb{R},\bX)$ we have $v_0 \in \mathcal{C}(\mathbb{R},\dom(A))$, which implies that $w_0 \in \mathcal{C}^2(\mathbb{R},\dom(A))$.  Differentiating in \eqref{4.8.9} and using \eqref{integral-transform} we obtain
	\begin{equation}\label{4.8.10}
		w_0'(t) = \int_0^t v_0(s)\,\rmd s = \frac{1}{2}\int_0^t (t-s)^2 u_0(s)\,\rmd s,\;w_0''(t) = v_0(t) = \int_0^t (t-s)u_0(s)\,\rmd s
	\end{equation}
	for any $t \in \mathbb{R}$. 
It follows from \eqref{4.8.8} that
	\begin{align}\label{4.8.11}
		&w_0''(t) + 2Bw_0'(t) - Aw_0(t)=\nonumber\\
		&
		= \int_0^t (t-s)u_0(s)\,\rmd s + B\int_0^t (t-s)^2 u_0(s)\,\rmd s
		- A\int_0^t\!\!\int_0^s\!\!\int_0^\tau v_0(\xi)\,\rmd\xi\,\rmd\tau\,\rmd s\nonumber\\
		&= \int_0^t\!\!\int_0^s u_0(\tau)\,\rmd\tau\,\rmd s
		+ 2B\int_0^t\!\!\int_0^s (s-\tau)u_0(\tau)\,\rmd\tau\,\rmd s
		- A\int_0^t\!\!\int_0^s\!\!\int_0^\tau (\tau-\xi)u_0(\xi)\,\rmd\xi\,\rmd\tau\,\rmd s\nonumber\\
		&= \int_0^t\!\!\int_0^s u_0(\tau)\,\rmd\tau\,\rmd s
		+ 2B\int_0^t\!\!\int_0^s\!\!\int_0^\tau u_0(\xi)\,\rmd\xi\,\rmd\tau\,\rmd s
		- A\int_0^t\!\!\int_0^s\!\!\int_0^\tau (\tau-\xi)u_0(\xi)\,\rmd\xi\,\rmd\tau\,\rmd s\nonumber\\
		&
		= \int_0^t\!\!\int_0^s\!\Big(u_0(\tau) + 2B\int_0^\tau u_0(\xi)\,\rmd\xi
		- A\int_0^\tau (\tau-\xi)u_0(\xi)\,\rmd\xi\Big)\rmd\tau\,\rmd s
		= \int_0^t\!\!\int_0^s 0\,d\tau\,ds = 0 
	\end{align}
	for any $t \in \mathbb{R}$. Since $w_0(0) = w_0'(0) = 0$, we conclude that $w_0$ is a classical solution of \eqref{AB} with $x=0$ and $y=0$.
	From Theorem~\ref{t3.1} we infer that $w_0(t) = 0$ for any $t \in \mathbb{R}$. By \eqref{4.8.10} we have $v_0(t) = 0$ for any $t \in \mathbb{R}$. Finally, by \eqref{4.8.9}, we conclude that $u_0(t) = v_0''(t) = 0$ for any $t \in \mathbb{R}$, proving the theorem.
\end{proof}
\begin{remark}\label{r4.9}
	From Theorem~\ref{t3.10} and Theorem~\ref{t4.8} we infer that \eqref{AB} has a unique Laplace Transform mild solution and a unique integral mild solution given by formulas \eqref{Laplace-Mild-sol} and \eqref{Integral-Mild-C-S-AB}, respectively. Since these formulas are the same, we conclude that the two concepts of mild solutions introduced in Definition~\ref{Def-mild} and Definition~\ref{Def-Int-Mild} coincide.
\end{remark}

\section{The Inhomogeneous Equation and Trigonometric Identities.}\label{s5}
In this section we discuss the inhomogeneous equation \eqref{ABf} associated to \eqref{AB}.
We study classical and mild solutions of \eqref{ABf}. We approach mild solutions of \eqref{ABf} in the same fashion we did  for the homogeneous equation \eqref{AB}. One option is to use the Laplace Transform, the other is to integrate the equation twice. We recall the notation
\begin{equation}\label{measurable-exp-grow}
	\mathcal{M}_{-\nu}(\mathbb{R},\bX)
	= \big\{ f \in \mathcal{M}(\mathbb{R},\bX) : \sup_{t \in \mathbb{R}} e^{-\nu|t|}\|f(t)\| < \infty \big\}. 
\end{equation}
\begin{definition}\label{d5.1} Assume \textnormal{(ABY1)--(ABY3)}. 
	\begin{itemize}
		\item[(i)] If $f \in \mathcal{C}(\mathbb{R}, \bX)$ then $u \in \mathcal{C}^2(\mathbb{R}, \bX)$ is called a \emph{classical solution} of \eqref{ABf} provided $u(t) \in \dom(A)$, $u'(t) \in \dom(B)$, $Au(\cdot) \in \mathcal{C}(\mathbb{R}, \bX)$ and $u$ satisfies \eqref{ABf}.
		\item[(ii)] If $f\in\mathcal{M}_{-\nu}(\mathbb{R},\bX)$, then $u \in \mathcal{C}_{0,-\nu}(\mathbb{R}, \bX)$ is called a \emph{Laplace Transform mild solution} of \eqref{ABf} provided $(\mathcal{L}u)(\lambda)$, $(\mathcal{L}u(-\cdot))(\lambda) \in \dom(A)$ and
		\begin{equation}\label{ABf-LTm} 
			Q(\lambda)\,(\mathcal{L}u)(\lambda) = \lambda x + 2Bx + y + (\mathcal{L}f)(\lambda),\,
			Q(-\lambda)\,(\mathcal{L}u(-\cdot))(\lambda) = \lambda x - 2Bx - y + (\mathcal{L}f(-\cdot))(\lambda)
		\end{equation}
		for any $\lambda \in \mathbb{C}^+_{\max\{\nu, \omega\}}$.
		\item[(iii)] If $f \in L^1_{\mathrm{loc}}(\mathbb{R}, \bX)$ then $u \in \mathcal{C}(\mathbb{R}, \bX)$ is called an \emph{integral mild solution} of \eqref{ABf} provided
		\begin{equation}\label{ABf-Im-1}  
			\int_0^t u(s)\,\rmd s \in \dom(B), \quad \int_0^t (t-s)u(s)\,\rmd s \in \dom(A),
		\end{equation}
		\begin{equation}\label{ABf-Im-2}
			t \mapsto B\int_0^t u(s)\,\rmd s, \quad t \mapsto A\int_0^t (t-s)u(s)\,\rmd s \;\in\; \mathcal{C}(\mathbb{R}, \bX),
		\end{equation}
		\begin{equation}\label{ABf-Im-3}   
			u(t) + 2B\int_0^t u(s)\,\rmd s = A\int_0^t (t-s)u(s)\,\rmd s + x
			+ t(y + Bx) + \int_0^t (t-s)\,f(s)\,\rmd s,\;\mbox{for any}\;t\in\RR.
		\end{equation}
	\end{itemize}	
\end{definition}	
Our first goal of this section is to prove that the solution of \eqref{ABf} is given by
\begin{equation}\label{ABf-Sol} 
	u(t) = C_{A,B}(t)\,x + S_{A,B}(t)\,y + \int_0^t S_{A,B}(t-s)\,f(s)\,\rmd s.
\end{equation}
Assuming some hypotheses on the initial conditions and $f$, we will see below that the function defined in \eqref{ABf-Sol} is a classical or mild solution of \eqref{ABf}. First, we will discuss the properties of $S_{A,B} * f$, the convolution term in \eqref{ABf-Sol}.

\begin{lemma}\label{l5.2} Assume \textnormal{(ABY1)--(ABY3)}. Then, the following holds.
	\begin{itemize}
		\item[(i)] If $f \in L^1_{\mathrm{loc}}(\mathbb{R}, \bX)$ then $S_{A,B} * f \in \mathcal{C}(\mathbb{R}, \bY)$.
		\item[(ii)] If $f \in \mathcal{M}_{-\nu}(\mathbb{R}, \bX)$ then $S_{A,B} * f \in \mathcal{C}_{0,-\gamma}(\mathbb{R}, \bX)$ for any $\gamma > \max\{\omega, \nu\}$.
		\item[(iii)] If $f \in \mathcal{C}^1(\mathbb{R}, \bX)$ then $S_{A,B} * f \in \mathcal{C}^2(\mathbb{R}, \bX)\cap \mathcal{C}^1(\mathbb{R}, \bY)$ and
		\begin{equation}\label{Sf-prime} \quad (S_{A,B} * f)' = S_{A,B}' * f,\;  (S_{A,B} * f)'' = S_{A,B}'(\cdot)\,f(0) + S_{A,B}' * f', 
		\end{equation}
		\item[(iv)] If $f \in \mathcal{C}(\mathbb{R}, \bY)$ then $S_{A,B} * f \in \mathcal{C}^2(\mathbb{R}, \bX)\cap \mathcal{C}^1(\mathbb{R}, \bY)$ and
		\begin{equation}\label{Sf-primeY}
			(S_{A,B} * f)' = (S_{A,B}(\cdot)|_\bY)' * f,\;  (S_{A,B} * f)'' = (S_{A,B}(\cdot)|_\bY)'' * f + f.
		\end{equation}
	\end{itemize}
\end{lemma}
\begin{proof}
	Assertion (i) follows from \cite[Proposition 1.3.4]{ABHN}, since $S_{A,B}: \mathbb{R} \to \mathcal{B}(\bX, \bY)$ is strongly continuous by Lemma~\ref{l4.2}(i), first on $\mathbb{R}_+$ and then on $\mathbb{R}_-$. From \eqref{SAB-XY} and since $\bY \hookrightarrow \bX$, there exists $c_1 > 0$ such that
	$\|y\| \leq c_1\|y\|_\bY$ for any $y\in\bY$, which implies that 
	\begin{align}\label{5.2.1}
		\left\|\int_0^t S_{A,B}(t-s)\,f(s)\,\rmd s\right\| &\le \left|\int_0^t \overline{M}\,e^{\omega|t-s|}\,\|f(s)\|\,\rmd s\right|\le \overline{M}c_1K_f\left|\int_0^t e^{\omega|t-s|}\,e^{\nu|s|}\,\rmd s\right|\nonumber\\
		&\le \overline{M}c_1K_f\,|t|\,e^{\max\{\omega,\nu\}|t|}\;\mbox{for any}\;t\in\RR.
	\end{align}
	Here $K_f=\sup_{t \in \mathbb{R}} e^{-\nu|t|}\|f(t)\| < \infty$. Assertion (ii) follows shortly.
	
	(iii) Let $K_1: \mathbb{R} \to \mathcal{B}(\bX)$ be the operator-valued function defined by $K_1 = S_{A,B}$. From Lemma~\ref{l4.2}(i) we have that $K_1(\cdot)\,x \in \mathcal{C}^1(\mathbb{R}, \bX)$ for any $x \in \bX$. From Lemma~\ref{App3} for $\bX_1 = \bX_2 = \bX$, $g = f$ we obtain that $S_{A,B} * f \in \mathcal{C}^1(\mathbb{R}, \bX)$ and
	\begin{equation}\label{5.2.2} 
		(S_{A,B} * f)' = S_{A,B}' * f = S_{A,B}(\cdot)\,f(0) + S_{A,B} * f'.
	\end{equation}
	From (i) we infer that $S_{A,B} * f' \in \mathcal{C}(\mathbb{R}, \bY)$. Since $S_{A,B}(\cdot)\,f(0) \in \mathcal{C}(\mathbb{R}, \bY)$ by Lemma 4.2(i) it follows that $(S_{A,B} * f)' \in \mathcal{C}(\mathbb{R}, \bY)$, which implies that $S_{A,B} * f \in \mathcal{C}^1(\mathbb{R}, \bY)$. Applying again Lemma~\ref{App3} for $K_1 = S_{A,B}$, $\bX_1 = \bX_2 = \bX$, $g = f' \in \mathcal{C}(\mathbb{R}, \bX)$ we obtain $S_{A,B} * f' \in \mathcal{C}^1(\mathbb{R}, \bX)$ and $(S_{A,B} * f')' = S_{A,B}' * f'$. Since $S_{A,B}(\cdot)\,f(0) \in \mathcal{C}^1(\mathbb{R}, \bX)$ by Lemma~\ref{l4.2}(i) we conclude that $S_{A,B} * f \in \mathcal{C}^2(\mathbb{R}, \bX)$ and
	\begin{equation}\label{5.2.3} (S_{A,B} * f)'' = S_{A,B}'(\cdot)\,f(0) + S_{A,B}' * f'.
	\end{equation}
	Identities \eqref{Sf-prime} follow from \eqref{5.2.2} and \eqref{5.2.3}.
	
	(iv) Let $K_2: \mathbb{R} \to \mathcal{B}(\bY)$, $K_2(t) = S_{A,B}(t)|_\bY$. From Lemma~\ref{l4.2}(iii) we have $K_2(\cdot)\,y \in \mathcal{C}^1(\mathbb{R}, \bY)$ for any $y \in \bY$. Since $f \in \mathcal{C}(\mathbb{R}, \bY)$, from Lemma~\ref{App3} for $\bX_1 = \bX_2 = \bY$ and $g = f$ we obtain that $S_{A,B} * f \in \mathcal{C}^1(\mathbb{R}, \bY) \subseteq \mathcal{C}^1(\mathbb{R}, \bX)$ and $(S_{A,B} * f)' =(S_{A,B}(\cdot)|_\bY*f)'= S_{A,B}' * f$.
	
	Let $K_3: \mathbb{R} \to \mathcal{B}(\bY, \bX)$, $K_3(t) = S_{A,B}'(t)|_\bY$. Using again Lemma~\ref{l4.2}(iii) we know that $K_3(\cdot)\,y \in \mathcal{C}^1(\mathbb{R}, \bX)$ for any $y\in\bY$ and $K_3'(t)\,y = (S_{A,B}(t)\,y)''$ for any $t \in \mathbb{R}$. From Lemma~\ref{App3} with $\bX_1 = \bY$, $\bX_2 = \bX$, $g = f$, we infer that $S_{A,B}'(\cdot)|_\bY * f \in \mathcal{C}^1(\mathbb{R}, \bX)$ and 
	\begin{equation}\label{5.2.4} 
		(S_{A,B}'(\cdot)|_\bY * f)' = S_{A,B}''(\cdot)|_\bY * f + f.
	\end{equation}
	Hence, $S_{A,B} * f \in \mathcal{C}^2(\mathbb{R}, \bX)$.
	The identity \eqref{Sf-primeY} follows from \eqref{diff-SAB-3} and \eqref{5.2.4}.
\end{proof}

\begin{lemma}\label{l5.3} Assume \textnormal{(ABY1)--(ABY3)}. Then,
	\begin{itemize}
		\item[(i)] If $f \in \mathcal{C}^1(\mathbb{R}, \bX)$ then $S_{A,B} * f \in \mathcal{C}(\mathbb{R}, \dom(A))$;
		\item[(ii)] If $f \in \mathcal{C}(\mathbb{R}, \bY)$ then $S_{A,B} * f \in \mathcal{C}(\mathbb{R}, \dom(A))$.
	\end{itemize}	
\end{lemma}
\begin{proof}
	(i)  Let $K_4(t):\bX \to \dom(A)$ be the operator valued function defined by $$K_4(t)\,x = \displaystyle\int_0^t S_{A,B}(s)\,x\,\rmd s.$$ 
	By Lemma 4.8(i) $K_4$ is well defined and
	\begin{equation}\label{5.3.1} 
		AK_4(t)\,x = S_{A,B}'(t)\,x + 2B\SAB(t)\,x - x\;\mbox{for any}\;t\in\RR,\,x\in\bX.
	\end{equation}
	From \eqref{BSAB-XX}, \eqref{SAB-prime-XX} and \eqref{5.3.1} we obtain $K_4(t) \in \mathcal{B}(\bX, \dom(A))$ for any $x \in \bX$. Moreover, since $B|_\bY \in \mathcal{B}(\bY, \bX)$, by \eqref{r3.1.1}, $S_{A,B}(\cdot)\,x \in \mathcal{C}^1(\mathbb{R}, \bX) \cap \mathcal{C}(\mathbb{R}, \bY)$, by Lemma~\ref{l4.2}(i), from (5.3.1) we conclude that $K_4: \mathbb{R} \to \mathcal{B}(\bX, \dom(A))$ is strongly continuous. From \cite[Proposition 1.3.4]{ABHN} we infer that
	\begin{equation}\label{5.3.2} 
		K_4 * f' \in \mathcal{C}(\mathbb{R}, \dom(A)).
	\end{equation}
	From \eqref{Sf-prime} we immediately see that
	\begin{align}\label{5.3.3} 
		(S_{A,B} * f)(t)&= \int_0^t \bigl[S_{A,B}(s)\,f(0) + (S_{A,B} * f')(s)\bigr]\,\rmd s=K_4(t)\,f(0) + \bigl(1 * (S_{A,B} * f')\bigr)(t)\nonumber\\
		&= K_4(t)\,f(0) + \bigl((1 * S_{A,B}) * f'\bigr)(t)
		= K_4(t)\,f(0) + (K_4 * f')(t) \;\text{for any}\; t \in \mathbb{R}.
	\end{align}
	Assertion (i) follows from \eqref{5.3.2}, \eqref{5.3.3} and since $K_4(\cdot)\,f(0) \in \mathcal{C}(\mathbb{R}, \dom(A))$.
	
	(ii) It follows from Lemma~\ref{l4.2}(iii) and \eqref{ASAB-YX} that $S_{A,B}(\cdot)|_\bY$ is a strongly continuous operator-valued function taking values in $\mathcal{B}(\bY, \dom(A))$. Since $f \in \mathcal{C}(\mathbb{R}, \bY)$, from \cite[Proposition 1.3.4]{ABHN} we obtain that $S_{A,B} * f \in \mathcal{C}(\mathbb{R}, \dom(A))$, proving (ii).
\end{proof}
\begin{lemma}\label{l5.4}
	Assume \textnormal{(ABY1)--(ABY3)} and let $f \in L^1_{\mathrm{loc}}(\mathbb{R}, \bX)$.  Then, $u_f=\SAB*f$ satisfies \eqref{ABf-Im-1} and \eqref{ABf-Im-2}.
\end{lemma}
\begin{proof}
	We recall the notation $1$ for the function identically equal to $1$ and $1_\RR$ the identity function on $\RR$. We define $g_1=1*f$ and $g_2=1_\RR*f$.
	Using the properties of convolution we see that
	\begin{equation}\label{5.4.1} 
		\int_0^t u_f(s)\,\rmd s = (1 * u_f)(t) = \bigl(1 * (S_{A,B} * f)\bigr)(t)
		= \bigl(S_{A,B} * (1 * f)\bigr)(t) = (S_{A,B} * g_1)(t)
	\end{equation}
	for any $t \in \mathbb{R}$.
	Similarly,
	\begin{equation}\label{5.4.2}
		\int_0^t (t-s)\,u_f(s)\,ds = (1_{\mathbb{R}} * u_f)(t) = \bigl(1_{\mathbb{R}} * (S_{A,B} * f)\bigr)(t)
		= \bigl(S_{A,B} * (1_{\mathbb{R}} * f)\bigr)(t) = (S_{A,B} * g_2)(t) 
	\end{equation}
	for any $t \in \mathbb{R}$. Since $f \in L^1_{\mathrm{loc}}(\mathbb{R}, \bX)$ it follows that
	\begin{equation}\label{5.4.3}
		g_1 \in \mathcal{C}(\mathbb{R}, \bX), \; g_2 \in \mathcal{C}^1(\mathbb{R}, \bX), \;g_2' = g_1.
	\end{equation}
	From Lemma~\ref{l5.2}(i) and \eqref{5.4.3} we obtain
	$S_{A,B} * g_1 \in \mathcal{C}(\mathbb{R}, \bY)$. Since $B|_\bY \in \mathcal{B}(\bY, \bX)$ by \eqref{r3.1.1}, it follows that $B(S_{A,B} * g_1)(\cdot) \in \mathcal{C}(\mathbb{R}, \bX)$. It follows from Lemma~\ref{l5.3}(i) and \eqref{5.4.3} that $S_{A,B} * g_2 \in \mathcal{C}(\mathbb{R}, \dom(A))$, hence $A(S_{A,B} * g_2)(\cdot) \in \mathcal{C}(\mathbb{R}, \bX)$. Using \eqref{5.4.1} and \eqref{5.4.2} we infer that $u_f =S_{A,B}*f$ satisfies \eqref{ABf-Im-1} and \eqref{ABf-Im-2}.
\end{proof}
To prove the next theorem we need the following elementary result
\begin{equation}\label{5.4.4} 
	(S_{A,B} * g)(-\cdot) = -S_{A,B}(-\cdot) * g(-\cdot), \;\mbox{for any}\; g \in \mathcal{M}_{-\nu}(\mathbb{R}, \bX),
\end{equation}
which combined with \eqref{r3.8.2} yields
\begin{equation}\label{5.4.5} 
	\bigl(\mathcal{L}(S_{A,B} * g)(-\cdot)\bigr)(\lambda) = Q^{-1}(-\lambda)\,(\mathcal{L}\,g(-\cdot))(\lambda)\;\mbox{for any}\; g \in \mathcal{M}_{-\nu}(\mathbb{R}, \bX),\,\lambda \in \mathbb{C}^+_{\max\{\omega,\nu\}}
\end{equation}
\begin{theorem}\label{t5.5}
	Assume \textnormal{(ABY1)--(ABY3)}. Then,
	\begin{itemize}
		\item[(i)] If $x \in \dom(A)$, $y \in \bY$, $f \in \mathcal{C}^1(\mathbb{R}, \bX)$ then the function $u$ defined in \eqref{ABf-Sol} belongs to $\mathcal{C}^2(\mathbb{R}, \bX) \cap \mathcal{C}^1(\mathbb{R}, \bY) \cap \mathcal{C}(\mathbb{R}, \dom(A))$ and is the unique classical solution of \eqref{ABf}.
		\item[(ii)] If $x \in \dom(A)$, $y \in \bY$, $f \in \mathcal{C}(\mathbb{R}, \bY)$ then the function $u$ defined in \eqref{ABf-Sol} belongs to $\mathcal{C}^2(\mathbb{R}, \bX) \cap \mathcal{C}^1(\mathbb{R}, \bY) \cap \mathcal{C}(\mathbb{R}, \dom(A))$ and is the unique classical solution of \eqref{ABf}.
		\item[(iii)] If $x \in \dom(B)$, $y \in \bX$, $f \in \mathcal{M}_{-\nu}(\mathbb{R}, \bX)$, then the function defined in \eqref{ABf-Sol} is the unique Laplace Transform mild solution of \eqref{ABf} .
		\item[(iv)] If $x \in \dom(B)$, $y \in \bX$, $f \in L^1_{\mathrm{loc}}(\mathbb{R}, \bX)$ then the function defined in \eqref{ABf-Sol} is the unique Integral mild solution of \eqref{ABf} .
	\end{itemize}
\end{theorem}
\begin{proof}
	Let $u_{x,y}: \mathbb{R} \to \bX$, $u_{x,y}(t) = C_{A,B}(t)\,x + S_{A,B}(t)\,y$. From Theorem~\ref{t3.1} and \eqref{cT-row 1} we know that $u_{x,y}$ is a classical solution of \eqref{AB}. Moreover, from Lemma~\ref{l3.4} we have
	\begin{equation}\label{5.5.1} 
		u_{x,y} \in \mathcal{C}^2(\mathbb{R}, \bX) \cap \mathcal{C}^1(\mathbb{R}, \bY) \cap \mathcal{C}(\mathbb{R}, \dom(A)), \;\mbox{for any}\; x \in \dom(A), y \in \bY.
	\end{equation}
	Fix $f \in \mathcal{C}^1(\mathbb{R}, \bX)$. From Lemma~\ref{l5.2}(iii) and Lemma~\ref{l5.3}(i) we obtain
	\begin{equation}\label{5.5.2} 
		S_{A,B} * f \in \mathcal{C}^2(\mathbb{R}, \bX) \cap \mathcal{C}^1(\mathbb{R}, \bY) \cap \mathcal{C}(\mathbb{R}, \dom(A)).
	\end{equation}
	Using \eqref{A-int-SAB}, \eqref{Sf-prime}, \eqref{5.3.1} and \eqref{5.3.3} we evaluate
	\begin{align}\label{5.5.3}
		&(S_{A,B} * f)''(t) + 2B(S_{A,B} * f)'(t) - A(S_{A,B} * f)(t) = S_{A,B}'(t)\,f(0) + (S_{A,B}' * f')(t) \nonumber\\
		&\qquad+ 2B\Bigl(S_{A,B}(t)\,f(0) + (S_{A,B} * f')(t)\Bigr) - A\Bigl(\int_0^t S_{A,B}(s)\,f(0)\,\rmd s + (K_4 * f')(t)\Bigr)\nonumber\\
		&= S_{A,B}'(t)\,f(0) + 2B\SAB(t)\,f(0) - A\int_0^t S_{A,B}(s)\,f(0)\,\rmd s+\bigl((S_{A,B}' + 2B\SAB - AK_4) * f'\bigr)(t) \nonumber\\
		&= f(0) + (I * f')(t) = f(0) + \int_0^t f'(s)\,ds = f(t),\;\mbox{for any}\; t\in\RR.
	\end{align}
	Moreover,
	\begin{equation}\label{5.5.4} 
		(S_{A,B} * f)(0) = 0, \; (S_{A,B} * f)'(0) = 0.
	\end{equation}
	From \eqref{5.5.2}, \eqref{5.5.3} and \eqref{5.5.4} we infer that $S_{A,B} * f$ is a classical solution of \eqref{ABf} with $x = y = 0$. We conclude that $u = u_{x,y} + S_{A,B} * f$ is a classical solution of \eqref{ABf}. Moreover, from \eqref{5.5.1} and \eqref{5.5.2} we have $u = u_{x,y} + S_{A,B} * f \in \mathcal{C}^2(\mathbb{R}, \bX) \cap \mathcal{C}^1(\mathbb{R}, \bY) \cap \mathcal{C}(\mathbb{R}, \dom(A))$. If $u_1$ and $u_2$ are two classical solutions of \eqref{ABf} then $u_1 - u_2$ is a classical solution of \eqref{AB} with $x = y = 0$. Hence, $u_1 = u_2$ by Theorem~\ref{t3.1}. Assertion (i) is now proved.
	
	Fix $f \in \mathcal{C}(\mathbb{R}, \bY)$. From Lemma~\ref{l5.2}(iv) and Lemma~\ref{l5.3}(ii) we have
	\begin{equation}\label{5.5.5}  S_{A,B} * f \in \mathcal{C}^2(\mathbb{R}, \bX) \cap \mathcal{C}^1(\mathbb{R}, \bY) \cap \mathcal{C}(\mathbb{R}, \dom(A)).
	\end{equation}
	Using \eqref{Sf-primeY} and Lemma~\ref{l4.2}(iii) we compute
	\begin{align}\label{5.5.6}
		&(S_{A,B} * f)''(t) + 2B(S_{A,B} * f)'(t) - A\,S_{A,B} * f \nonumber\\
		&\qquad= (S_{A,B}(\cdot)|_\bY)'' * f + f + 2B\bigl(S_{A,B}(\cdot)|_\bY\bigr)' * f - A\,S_{A,B}(\cdot)|_\bY * f \nonumber\\
		&\qquad= f + \Bigl((S_{A,B}(\cdot)|_\bY)'' + 2B(S_{A,B}(\cdot)|_\bY)' - A\,S_{A,B}(\cdot)|_\bY\Bigr) * f \nonumber\\
		&\qquad= f + 0 * f = f.
	\end{align}
	Yet again one can readily see that $(S_{A,B} * f)(0) = 0$ and  $(S_{A,B} * f)'(0) = 0$. From \eqref{5.5.5} and \eqref{5.5.6} we obtain that $S_{A,B} * f$ is a classical solution of \eqref{ABf} with $x = y = 0$. Hence, $u = u_{x,y} + S_{A,B} * f$ is a classical solution of \eqref{ABf}. Also, from \eqref{5.5.1} and \eqref{5.5.5} we have $u = u_{x,y} + S_{A,B} * f \in \mathcal{C}^2(\mathbb{R}, \bX) \cap \mathcal{C}^1(\mathbb{R}, \bY) \cap \mathcal{C}(\mathbb{R}, \dom(A))$. The uniqueness follows again from the uniqueness of the associated homogeneous equation \eqref{AB} proved in Theorem~\ref{t3.1}.
	
	(iii) Fix $x \in \dom(B)$, $y \in \bX$, $f \in \mathcal{M}_{-\nu}(\mathbb{R}, \bX)$. Then, by Theorem~\ref{t3.10}, \eqref{IVP-x0} and \eqref{IVP-0y} we obtain $u_{x,y} = C_{A,B}(\cdot)\,x + S_{A,B}(\cdot)\,y$ is a Laplace Transform mild solution of \eqref{AB}, hence $u_{x,y}$ satisfies \eqref{Def-Lap-Mild}. It follows from \eqref{5.4.5} that
	\begin{equation}\label{5.5.8} \bigl(\mathcal{L}(S_{A,B} * f)\bigr)(\lambda) = Q^{-1}(\lambda)\,(\mathcal{L}f)(\lambda), \; \bigl(\mathcal{L}(S_{A,B} * f)(-\cdot)\bigr)(\lambda) = Q^{-1}(-\lambda)\,(\mathcal{L}f(-\cdot))(\lambda),
	\end{equation}
	for any $\lambda \in \mathbb{C}^+_{\max\{\omega, \nu\}}$. Since $u_{x,y}$ satisfies \eqref{Def-Lap-Mild}, by \eqref{5.5.8} we conclude that $u = u_{x,y} + S_{A,B} * f$ satisfies \eqref{ABf-LTm}, that is, $u = u_{x,y} + S_{A,B} * f$ is a Laplace Transform mild solution of \eqref{ABf}. If $u_1$ and $u_2$ are two Laplace Transform mild solutions of \eqref{ABf} then from \eqref{ABf-LTm} we immediately infer that $(\mathcal{L}(u_1 - u_2))(\lambda) = 0$ and $(\mathcal{L}(u_1 - u_2)(-\cdot))(\lambda) = 0$ for any $\lambda \in \mathbb{C}^+_{\max\{\omega,\nu\}}$, which implies that $u_1 = u_2$, proving (iii).
	
	(iv) Fix $x \in \dom(B)$, $y \in \bX$, $f \in L^1_{\mathrm{loc}}(\mathbb{R}, X)$. From Theorem~\ref{t4.8} we know that $u_{x,y} = C_{A,B}(\cdot)\,x + S_{A,B}(\cdot)\,y$ is an integral mild solution, hence $u_{x,y}$ satisfies \eqref{domain-cond}, \eqref{cont-cond} and \eqref{Int-Mild}. By Lemma~\ref{l5.4}, $S_{A,B} * f$ satisfies conditions \eqref{ABf-Im-1} and \eqref{ABf-Im-2}. Hence,
	\begin{equation}\label{5.5.9}
		u = u_{x,y} + S_{A,B} * f \text{ satisfies } \eqref{ABf-Im-1} \text{ and } \eqref{ABf-Im-2}.
	\end{equation}
	We note that \eqref{A-int-t-SAB} can be rewritten as
	\begin{equation}\label{5.5.10}
		S_{A,B} + 2B(1 * S_{A,B}) - A(1_{\mathbb{R}} * S_{A,B}) = 1_{\mathbb{R}} \cdot I.
	\end{equation}
	From \eqref{5.4.1}, \eqref{5.4.2} and \eqref{5.5.10} we evaluate
	\begin{align}\label{5.5.11}
		(S_{A,B} * f)&(t)+2B\int_0^t (S_{A,B} * f)(s)\,\rmd s - A\int_0^t (t-s)(S_{A,B} * f)(s)\,\rmd s \nonumber\\
		&= \Bigl(S_{A,B} + 2B(1 * S_{A,B}) - A(1_{\mathbb{R}} * S_{A,B})\Bigr) * f\,(t) = (1_{\mathbb{R}} \cdot I * f)(t) = \int_0^t (t-s)\,f(s)\,\rmd s 
	\end{align}
	for any $t \in \mathbb{R}$. Since $u_{x,y}$ satisfies \eqref{Int-Mild}, from \eqref{5.5.11} we obtain $u = u_{x,y} + S_{A,B} * f$ satisfies \eqref{ABf-Im-3}. From \eqref{5.5.9}  we conclude that $u = u_{x,y} + S_{A,B} * f$ is an integral mild solution of \eqref{ABf}. If $u_1$ and $u_2$ are two integral mild solutions of \eqref{ABf} then $u_1 - u_2$ is an integral mild solution of \eqref{AB} with $x = y = 0$. Hence, by Theorem~\ref{t4.8}, we infer that $u_1 = u_2$, proving (iv).
\end{proof}
\begin{corollary}\label{c5.6}
	Assume \textnormal{(ABY1)--(ABY3)}. Then,
	\begin{itemize}
		\item[(i)] If $f\in\mathcal{C}(\RR,\bX)$ and $u$ is a classical solution of \eqref{ABf} such that $u\in\mathcal{C}^2(\mathbb{R}, \bX) \cap \mathcal{C}^1(\mathbb{R}, \bY) \cap \mathcal{C}(\mathbb{R}, \dom(A))$ then $u$ satisfies \eqref{ABf-Sol}; 
		\item[(ii)] If $h$ is a classical solution of \eqref{AB} then
		\begin{align}\label{5.6.7}  
			h(t+s) + h(t-s)&= 2C_{A,B}(s)\,h(t) - 4\int_0^s S_{A,B}(s-\xi)\,B\,h'(t-\xi)\,d\xi\nonumber\\
			&=2C_{A,B}(s)\,h(t)-4S_{A,B}(s)\,B\,h(t) +4 \int_0^s S_{A,B}'(s-\xi)\,B\,h(t-\xi)\,d\xi,
		\end{align}
		\begin{align}\label{5.6.8} 
			h(t+s) - h(t-s)&= 2S_{A,B}(s)\,h'(t) + 4\int_0^s S_{A,B}(s-\xi)\,B\,h'(t-\xi)\,d\xi\nonumber\\
			&=2S_{A,B}(s)\,h'(t)+4S_{A,B}(s)\,B\,h(t)-4 \int_0^s S_{A,B}'(s-\xi)\,B\,h(t-\xi)\,d\xi 
		\end{align}
		for any $t, s \in \mathbb{R}$.
	\end{itemize}
\end{corollary}
\begin{proof}(i) Assume $u$ is a classical solution of \eqref{ABf} such that $u\in\mathcal{C}^2(\mathbb{R}, \bX) \cap \mathcal{C}^1(\mathbb{R}, \bY) \cap \mathcal{C}(\mathbb{R}, \dom(A))$. Then, integrating twice in \eqref{ABf} we immediately see that $u$ is an integral mild solution of \eqref{ABf}. From Theorem~\ref{t5.5}(iv) we conclude that it satisfies \eqref{ABf-Sol}, proving (i).
	
	(ii) Let $h$ be a classical solution of \eqref{AB} and let $F_{t,\pm}: \mathbb{R} \to\bX$ be defined by
	\begin{equation}\label{5.6.1}
		F_{t,\pm}(s) = h(t+s) \pm h(t-s), \;\mbox{for}\; t \in \mathbb{R}.
	\end{equation}
	Then, from Theorem~\ref{t3.1} we infer that
	\begin{equation}\label{5.6.2}
		F_{t,\pm} \in \mathcal{C}^2(\mathbb{R}, \bX) \cap \mathcal{C}^1(\mathbb{R}, \bY) \cap \mathcal{C}(\mathbb{R}, \dom(A))\;\mbox{for any}\; t \in \mathbb{R}.
	\end{equation}
	Moreover,
	\begin{equation}\label{5.6.3}
		F_{t,\pm}'(s) = h'(t+s) \mp h'(t-s), \; F_{t,\pm}''(s) = h''(t+s) \pm h''(t-s),\;\mbox{for any}\; t,s \in \mathbb{R}.
	\end{equation}
	It follows that
	\begin{equation}\label{5.6.4} 
		F_{t,\pm}''(s) + 2B\,F_{t,\pm}'(s) - A\,F_{t,\pm}(s) = \mp 4B\,h'(t-s)\;\mbox{for any}\; t,s \in \mathbb{R}.
	\end{equation}
	Also,
	\begin{equation}\label{5.6.5} 
		F_{t,+}(0) = 2h(t), \; F_{t,+}'(0) = 0, \; F_{t,-}(0) = 0, \; F_{t,-}'(0) = 2h'(t)\;\mbox{for any}\; t \in \mathbb{R}.
	\end{equation}
	From \eqref{5.6.2} and since $B|_\bY \in \mathcal{B}(\bY, \bX)$ we can see that $B\,h'(t-\cdot) \in \mathcal{C}(\mathbb{R}, \bX)$ for any $t \in \mathbb{R}$. From (i), \eqref{5.6.2} and \eqref{5.6.4} we obtain
	\begin{equation}\label{5.6.6} 
		F_{t,\pm}(s) = C_{A,B}(s)\,F_{t,\pm}(0) + S_{A,B}(s)\,F_{t,\pm}'(0) \mp 4\int_0^s S_{A,B}(s-\xi)\,B\,h'(t-\xi)\,\rmd\xi\;\mbox{for any}\; t,s \in \mathbb{R}.
	\end{equation}
	Integrating by parts one can readily check that
	\begin{equation}\label{5.6.9}
		\int_0^s S_{A,B}(s-\xi)\,B\,h'(t-\xi)\,\rmd\xi = S_{A,B}(s)\,B\,h(t) - \int_0^s S_{A,B}'(s-\xi)\,B\,h(t-\xi)\,\rmd\xi
	\end{equation}
	The identities \eqref{5.6.7} and \eqref{5.6.8} follow shortly from \eqref{5.6.5}, \eqref{5.6.6} and \eqref{5.6.9}.
\end{proof}

\noindent\textbf{Trigonometric Identities for Generalized Cosine and Sine.}
In this subsection we look for various ``trigonometric" type functional equations satisfied by generalized cosine and sine operator families $\{\CAB(t)\}_{t\in\RR}$ and $\{\SAB(t)\}_{t\in\RR}$ defined in Section~\ref{s3}. We cannot proceed exactly as in the case of cosine and sine families since the generalized cosine and sine are not even and odd, respectively. Moreover, differentiating certain identity only complicates computations since many formulas have several additional terms, all of which would vanish in the case $B=0$. Instead we will use the definition of the generalized cosine and sine operator families as solutions of \eqref{IVP-x0} and \eqref{IVP-0y} and use the results on the inhomogeneous equation. 
\begin{lemma}\label{l5.1}
	Assume \textnormal{(ABY1)--(ABY3)}. Then,
	\begin{itemize}
		\item[(i)] $\SAB(t+s) = \CAB(t)\SAB(s) + \SAB(t)\SAB'(s), \;\mbox{for any}\; t, s \in \mathbb{R}.$
		\item[(ii)] $\CAB(t+s)|_\bY = \CAB(t)\CAB(s)|_\bY + \SAB(t)\CAB'(s)|_\bY, \;\mbox{for any}\; t, s \in \mathbb{R}.$
	\end{itemize}
\end{lemma}

\begin{proof}
	(i) Fix $x \in \dom(A)$, $s \in \mathbb{R}$ and define $f_{s,x} : \mathbb{R} \to\bX$ by $f_{s,x}(t) = \SAB(t+s)x$. Since $\SAB(\cdot)x$ is a classical solution of \eqref{IVP-0y}, we infer that $f_{s,x}$ is a classical solution of 
	\eqref{AB} with $f_{s,x}(0)=\SAB(s)x$ and $f_{s,x}'(0)=\SAB'(s)x$. Moreover,  $Af_{s,x}(\cdot) \in \mathcal{C}(\mathbb{R}, \bX)$ by Lemma~\ref{l4.2} (iii). From Theorem~\ref{t3.1} we conclude that
	$f_{s,x}(t) = \CAB(t)\SAB(s)x + \SAB(t)\SAB'(s)x$
	for any $t \in \mathbb{R}$. Hence, $\SAB(t+s)x = \CAB(t)\SAB(s)x + \SAB(t)\SAB'(t)x$. for any $t, s \in \mathbb{R}$, $x \in \dom(A)$. Hence, assertion (i) holds in the special case when $x\in\dom(A)$. The general case follows since $\SAB(t), \SAB(s), \SAB(t+s) \in \mathcal{B}(\bX,\bY)$, $\SAB'(t) \in \mathcal{B}(\bX)$, $\CAB(t)|_\bY \in \mathcal{B}(\bY)$, by Lemma~\ref{l4.5} and $\dom(A)$ is dense in  $\bX$ by Lemma~\ref{l4.6} (iv).
	
	(ii) Fix again $x \in \dom(A)$, $s \in \mathbb{R}$ and let $g_{s,x} : \mathbb{R} \to\bX$ be the function defined by $g_{s,x}(t) = \CAB(t+s)x$. From Lemma~\ref{l4.3}(ii) we have that $\CAB(\cdot)x$ is a classical solution of \eqref{IVP-x0}, thus $g_{s,x}$ is a classical solution of \eqref{AB} with 
	$g_{s,x}(0)=\CAB(s)x$ and $g_{s,x}'(0)=\CAB'(s)x$. Moreover, $Ag_{s,x}(\cdot) \in \mathcal{B}(\mathbb{R}, \bX)$.  Using again Theorem~\ref{t3.1} we infer that $\CAB(t+s)x = \CAB(t)\CAB(s)x + \SAB(t)\CAB'(s)x$ for any $t,s\in\mathbb{R}$, $x \in \dom(A)$.
	From Lemma~\ref{l4.5} we know that $\CAB(t)|_\bY, \CAB(s)|_\bY, \CAB(t+s)|_\bY \in \mathcal{B}(\bY)$, $\CAB'(t) \in \mathcal{B}(\bY,\bX)$, $\SAB(t) \in \mathcal{B}(\bX,\bY)$. Since $\dom(A)$ is dense in $\bY$ in the $\bY$-norm by Lemma~\ref{l4.6}(v), we infer that assertion (ii) holds for any $x \in\bY$, proving the lemma.
\end{proof}
\begin{lemma}\label{l5.6}
	Assume \textnormal{(ABY1)--(ABY3)}. Then,
	\begin{itemize}
		\item[(i)] $C_{A,B}(t+s)\,x - C_{A,B}(t-s)\,x = 2S_{A,B}(s)\,C_{A,B}'(t)\,x + 4\displaystyle\int_0^s S_{A,B}(s-\xi)\,BC_{A,B}'(t-\xi)\,x\,\rmd\xi$
		for any $t, s \in \mathbb{R}$, $x \in \dom(A)$;
		\item[(ii)] $C_{A,B}(t+s)\,y - C_{A,B}(t-s)\,y = 2S_{A,B}(s)\,C_{A,B}'(t)\,y- 4\displaystyle\int_0^s S'_{A,B}(s-\xi)\,BC_{A,B}(t-\xi)\,y\,\rmd\xi+ 4S_{A,B}(s)\,BC_{A,B}(t)\,y $
		for any $t, s \in \mathbb{R}$, $y \in \bY$;
		\item[(iii)] $C_{A,B}(t+s)\,x + C_{A,B}(t-s)\,x = 2C_{A,B}(s)\,C_{A,B}'(t)\,x - 4\displaystyle\int_0^s S_{A,B}(s-\xi)\,BC_{A,B}'(t-\xi)\,x\,\rmd\xi$
		for any $t, s \in \mathbb{R}$, $x \in \dom(A)$;
		\item[(iv)] $C_{A,B}(t+s)\,y + C_{A,B}(t-s)\,y = 2C_{A,B}(s)\,C_{A,B}(t)\,y+ 4\displaystyle\int_0^s S'_{A,B}(s-\xi)\,BC_{A,B}(t-\xi)\,y\,\rmd\xi- 4S_{A,B}(s)\,BC_{A,B}(t)\,y$
		for any $t, s \in \mathbb{R}$, $y \in \bY$;
		\item[(v)] $S_{A,B}(t+s)\,y - S_{A,B}(t-s)\,y = 2S_{A,B}(s)\,S_{A,B}'(t)\,y + 4\displaystyle\int_0^s S_{A,B}(s-\xi)\,BS_{A,B}'(t-\xi)\,y\,\rmd\xi$
		for any $t, s \in \mathbb{R}$, $y \in \bY$;
		\item[(vi)] $S_{A,B}(t+s)\,x - S_{A,B}(t-s)\,x = 2S_{A,B}(s)\,S_{A,B}'(t)\,x - 4\displaystyle\int_0^s S_{A,B}'(s-\xi)\,BS_{A,B}(t-\xi)\,x\,\rmd\xi+4S_{A,B}(s)\,BS_{A,B}(t)\,x $
		for any $t, s \in \mathbb{R}$, $x \in \bX$;
		\item[(vii)] $S_{A,B}(t+s)\,y + S_{A,B}(t-s)\,y = 2C_{A,B}(s)\,S_{A,B}'(t)\,y - 4\displaystyle\int_0^s S_{A,B}(s-\xi)\,BS_{A,B}'(t-\xi)\,y\,\rmd\xi$
		for any $t, s \in \mathbb{R}$, $y \in \bY$;
		\item[(viii)] $S_{A,B}(t+s)\,x + S_{A,B}(t-s)\,x = 2C_{A,B}(s)\,S_{A,B}(t)\,x + 4\displaystyle\int_0^s S_{A,B}'(s-\xi)\,BS_{A,B}(t-\xi)\,x\,\rmd\xi- 4S_{A,B}(s)\,BS_{A,B}(t)\,x  $ for any $t, s \in \mathbb{R}$, $y \in \bX$.
	\end{itemize}
\end{lemma}
\begin{proof}
	Assertions (i) and (iii) follow from \eqref{5.6.7} and \eqref{5.6.8} for $h = C_{A,B}(\cdot)\,x$ with $x \in \dom(A)$. Similarly, assertions (v) and (vii) follow from \eqref{5.6.7} and \eqref{5.6.8} for $h = S_{A,B}(\cdot)\,y$ with $y \in\bY$. Assertions (ii) and (iv) follow from (i) and (iii) respectively, \eqref{5.6.7} and \eqref{5.6.8} and since $\dom(A)$ is dense in $\bY$ in the $\bY$-norm and all the terms in (i) and (iii) are operators from $\mathcal{B}(\bY)$. Assertions (vi) and (viii) follow from (v) and (vii), respectively, \eqref{5.6.7} and \eqref{5.6.8}, and since $\bY$ is dense in $\bX$ and all the terms in (vi) and (viii) are operators from $\mathcal{B}(\bX)$.
\end{proof}

\section{Special Cases and Examples}\label{s8}
In this section we discuss several special cases, (H1), (H2) and (H3), when Hypothesis (HAB) for the operators $A$, $B$ holds and present several examples of concrete damped waves models that satisfy our hypotheses. Even though we are especially interested in the case when $B$ is unbounded, our setup in (H1) includes the case when $B$ is bounded. This special case covers several classical examples, see Example~\ref{e8.1}, Example~\ref{e8.2} and Example~\ref{e8.6}.

\noindent\textbf{Hypothesis (H1).} \textit{We assume that $A_0:\dom(A_0)\subseteq\bX\to\bX$ is the generator of a $C_0$-group; $B\in\mathcal{B}(\bX)$ is such that $B(\dom(A_0))\subseteq\dom(A_0)$ and moreover, 
	$B\big|_{\dom(A_0)}\in\mathcal{B}(\dom(A_0))$; $W_0\in\mathcal{B}(\dom(A_0),\bX)$ and $A=A_0^2+W_0$ with $\dom(A)=\dom(A_0^2)$}.

The next special case is related to the case of linear elastic systems with structural damping, see, e.g., \cite{CR,N}. In this specific type of examples the linear operator $A$ can be represented as a quadratic polynomial depending on the leading order term $B_0$ of $B$. This particular subclass of operators that satisfy Hypothesis (HAB) include the interesting case of an equation whose associated undamped equation is not well-posed, see Example~\ref{e8.4} below, where the diffusion coefficient can be negative.

\noindent\textbf{Hypothesis (H2).}  \textit{We assume that $B_0:\dom(B_0)\subseteq\bX\to\bX$ is a $C_0$-group generator, and the operators $B_1,\widetilde{B}_1,E_0,E_1\in\mathcal{B}(\bX)$, and a constant $\gamma>-1$ are such that
	\begin{itemize}
		\item[(i)] $B_1(\dom(B_0))\subseteq\dom(B_0)$ and $(B_1)\big|_{\dom(B_0)}\in\mathcal{B}(\dom(B_0))$,
		\item[(ii)] $(B_0 B_1-B_1 B_0)x=\widetilde{B}_1 x$ for any $x\in\dom(B_0)$,
		\item[(iii)] $A=\gamma B_0^2+E_1 B_0+E_0$, $B=B_0+B_1$, with $\dom(A)=\dom(B_0^2)$ and $\dom(B)=\dom(B_0)$.
\end{itemize}}
The next special case is specific to higher order PDEs, when the ``leading order term" $A_0$ of $A$ has a degree equal to at least four times the degree of the leading order term of $B$.

\noindent\textbf{Hypothesis (H3).} \textit{ We assume that $A_0:\dom(A_0)\subseteq\bX\to\bX$, $B_0:\dom(B_0)\subseteq\bX\to\bX$ are
	generators of $C_0$-groups, and operators $B_1,\widetilde{B}_1,E_0,E_1,E_2\in\mathcal{B}(\bX)$ and $W_0\in\mathcal{B}(\dom(A_0),\bX)$ are such that
	\begin{itemize}
		\item[(i)] $A_0 B_0=B_0 A_0$,
		\item[(ii)] $\dom(A_0)\subseteq\dom(B_0^2)$,
		\item[(iii)] $A_0^\pm=\pm A_0-B_0$ generate $C_0$-groups on $\bX$,
		\item[(iv)] $(B_1)\big|_{\dom(A_0)}\in\mathcal{B}(\dom(A_0))$,
		\item[(v)] $(B_0 B_1-B_1 B_0)x=\widetilde{B}_1 x$ for any $x\in\dom(A_0^2)$,
		\item[(vi)] $A=A_0^2+E_2 B_0^2+E_1 B_0+E_0+W_0$, $B=B_0+B_1$ with $\dom(A)=\dom(A_0^2)$ and $\dom(B)=\dom(B_0)$.
\end{itemize}}

\begin{lemma}\label{l8.1}
	Each of the Hypotheses \textnormal{(H1), (H2), (H3)} is a special case of Hypothesis \textnormal{(HAB)}. 
\end{lemma}
\begin{proof}
	First, we assume Hypothesis (H1) and set $B_0 = 0$. Then one can readily check that Hypothesis (HAB)(i), (ii), (iii), (v) and (vi) hold true. Moreover, since $B_1 = B$ and $B_1\big|_{\dom(A)} \in \mathcal{B}(\dom(A))$,
	we see that Hypothesis (HAB)(iv) holds true.
	
	Next, we assume Hypothesis (H2) and set $A_0 = \sqrt{1+\gamma}\, B_0$.  We note that
	$\sqrt{1+\gamma} \in \mathbb{R}$ since $\gamma > -1$, which implies that $A_0$ generates a
	$C^0$-group on $\bX$. Clearly, Hypothesis (HAB)(i) holds. One can readily check that
	$\dom(A_0) = \dom(B_0)$, $\dom(A_0^2) = \dom(B_0^2)$, hence Hypothesis (HAB)(ii) holds true.
	In this case $A_0^\pm = (\pm\sqrt{1+\gamma} - 1)\,B_0$, thus Hypothesis (HAB)(iii) holds true.
	Conditions (iv) and (v) of Hypothesis (HAB) are part of Hypothesis (H2) as well.
	Setting $W_0: \dom(A_0) \to \bX$, $W_0 = E_1 B_0 + E_0$, we infer that
	$W_0 \in \mathcal{B}(\dom(A_0), \bX)$. Finally, $A_0^2 - B_0^2 = \gamma B_0^2$,
	which shows that Hypothesis (HAB)(vi) holds true.
	
	Finally, we assume Hypothesis (H3). Conditions (i), (iii), (iv), (v) of (HAB) are included in Hypothesis (H3).
	Fix $\mu\in\rho(A_0)$. Since $A_0$ and $B_0$ commute, 
	$B_0 x=(\mu-A_0)^{-1}B_0(\mu-A_0)x\in\dom(A_0)$
	for any $x\in\dom(A_0^2)$, proving $B_0(\dom(A_0^2))\subseteq\dom(A_0)$. Since $\dom(A_0) \subseteq \dom(B_0^2) \subseteq \dom(B_0)$,
	we infer that Hypothesis (HAB)(ii) holds true. 
	
	We note that
	\begin{equation}\label{2.4.1}
		B_0^2 x=B_0^2(\mu I-A_0)^{-1}(\mu I-A_0)x,\; B_0 x=B_0(\mu I-A_0)^{-1}(\mu I-A_0)x
	\end{equation}
	for any $x\in\dom(A_0)$, $\mu\in\rho(A_0)$. Since $B_0$ is a $C_0$-group generator, $B_0^2$ generates an analytic semigroup, so $B_0,B_0^2$
	are closed linear operators. Moreover, using  $\dom(A_0)\subseteq\dom(B_0^2)\subseteq\dom(B_0)$ it follows that 
	$B_0 R(\mu,A_0),B_0^2 R(\mu,A_0)\in\mathcal{B}(\bX)$ for any $\mu\in\rho(A_0)$. Clearly $\rho(A_0)\ne\emptyset$ since $A_0$ generates a $C_0$-group. Since $(\mu I-A_0)\big|_{\dom(A_0)}\in\mathcal{B}(\dom(A_0),\bX)$
	we conclude that
	\begin{equation}\label{bound}
		B_0\big|_{\dom(A_0)}, B_0^2\big|_{\dom(A_0)}\in \mathcal{B}(\dom(A_0), \bX).
	\end{equation}
	We define
	$\widetilde{W}_0: \dom(A_0) \to \bX$ by
	$
	\widetilde{W}_0 = (E_2 + I)\,B_0^2 + E_1\,B_0 + E_0 + W_0$. Since $W_0 \in \mathcal{B}(\dom(A_0), \bX)$ from \eqref{bound}
	we obtain $\widetilde{W}_0 \in \mathcal{B}(\dom(A_0), \bX)$. Moreover,
	\begin{equation}\label{rep-A}
		A_0^2 - B_0^2 + \widetilde{W}_0 = A_0^2 + W_0 = A,
	\end{equation}
	which proves that Hypothesis (HAB)(vi) holds true.
\end{proof}

\begin{example}\label{e8.1}
	We consider the initial value problem
	\begin{equation}\label{eq:1}
		\begin{cases}
			\partial_t^2 u + 2\,b(\xi)\,\partial_t u = \gamma\,\Delta_\xi u + \vec{a}(\xi)\cdot\nabla_\xi u + a_0(\xi)\,u,\\
			u(\xi,0) = f(\xi),\; u_t(\xi,0) = g(\xi).
		\end{cases}
	\end{equation}
	Here $\gamma > 0$, $\vec{a} \in \cC_\rmb(\R^k,\C^k)$, $a_0 \in \cC_\rmb(\R^k,\C)$,
	$b \in \cC^1_\rmb(\R^k,\C)$, $k\in\NN$. System \eqref{eq:1} is a special case of \eqref{AB} that satisfies hypothesis (H1). Indeed, here $A: H^2(\R^k) \to L^2(\R^k)$,
	$B: L^2(\R^k) \to L^2(\R^k)$,
	$A = \gamma\,\Delta_\xi + \vec{a}(\cdot)\cdot\nabla_\xi + a_0(\cdot)$ and $B=M_{b(\cdot)}$ the operator of
	multiplication by $b(\cdot)$. One can readily check that $A = A_0^2 + W_0$, where
	$A_0, W_0: H^1(\R^k) \to L^2(\R^k)$ are defined by
	$A_0 = i\sqrt{\gamma}\,(-\Delta_\xi)^{1/2}$, $W_0 = \vec{a}(\cdot)\cdot\nabla_\xi + a_0(\cdot)$. Since $\Delta$ is a self-adjoint, negative definite operator, we infer that $(-\Delta_\xi)^{1/2}$
	is a self-adjoint, positive definite operator, thus $A_0$ generates a $C^0$-group on
	$L^2(\R^k)$. Since $\vec{a} \in \mathcal{C}_\rmb(\R^k,\C^k)$, $a_0 \in \mathcal{C}_\rmb(\R^k,\C)$ we
	immediately conclude that $W_0 \in \mathcal{B}(H^1(\R^k), L^2(\R^k))$. Moreover, since
	$b \in \mathcal{C}^1_b(\R^k,\C)$, we obtain that $B\big|_{H^1(\R^k)} \in \mathcal{B}(H^1(\R^k))$,
	$B \in \mathcal{B}(L^2(\R^k))$, hence hypothesis (H1) is satisfied. 
\end{example}
Models analogous to \eqref{eq:1}, corresponding to the case when the domain $\R^k$ is replaced
by a bounded domain, are also very common in the literature; see, e.g., \cite{BH,CST1, CPSST,CZ,GH,KLT,OZP}
or even when $\RR^k$ is replaced by a Riemannian
manifold; see, e.g., \cite{AL,L}. Cases when the linear operator $A$ is a fourth-order operator are discussed in \cite{Soba}.

Model \eqref{eq:1} can be considered on $L^p(\R^k)$, $p \in (1,\infty)\setminus\{2\}$,
only in the case when $k=1$. The main reason is that the linear operator $i(-\Delta)^{1/2}$,
defined via Fourier multipliers, does not generate a $C^0$-group on $L^p(\R^k)$ unless $p=2$ or $k=1$; see, e.g.,
\cite{Sogge,Stein,Taylor}.

\begin{example}\label{e8.2}
	Consider again \eqref{eq:1} in the case $k=1$. In the case when $f \in W^{1,p}(\R)$,
	$g \in L^p(\R)$ we can model the system on $L^p(\R)$ as follows. We define
	$A: W^{2,p}(\R) \subseteq L^p(\R) \to L^p(\R)$ by
	$A = \gamma\,\partial_\xi^2 + \vec{a}(\cdot)\,\partial_\xi + a_0(\cdot)$ and $B$ be the operator
	of multiplication by $b(\cdot)$ on $L^p(\R)$. Then $A = A_0 ^2+ W_0$ where
	$A_0, W_0: W^{1,p}(\R) \to L^p(\R)$ are defined by $A_0 = \sqrt{\gamma}\,\partial_\xi$ and
	$W_0 = \vec{a}(\cdot)\,\partial_\xi + a_0(\cdot)$. The operator $A_0$ generates on $L^p(\R)$ the $C^0$-group $\{T_0(t)\}_{t\in\R}$ given by
	$T_0(t)f = f(\cdot + \sqrt{\gamma}\,t)$. Since $\vec{a}(\cdot),\,a_0(\cdot) \in \mathcal{C}_\rmb(\R)$
	and $b \in \mathcal{C}_\rmb^1(\R)$ we have $W_0 \in \mathcal{B}(W^{1,p}(\R), L^p(\R))$,
	$B\big|_{W^{1,p}(\R)} \in \mathcal{B}(W^{1,p}(\R))$ and $B \in \mathcal{B}(L^p(\R))$. Hence Hypothesis (H1) is satisfied. 	
\end{example}	
Examples~\ref{e8.1} and ~\ref{e8.2} are special cases of \eqref{AB} with $B$ being bounded. We now turn our attention to the case when $B$ is unbounded and start with the following example where $B$ is an order one differential operator.	
\begin{example}\label{l8.3}
	We consider the linear model
	\begin{equation}\label{8.3.0}
		\begin{cases}
			\partial_t^2 u + 2\,\vec{b}_0 \cdot \nabla_\xi\,\partial_t u + 2\,b(\xi)\,u_t
			= \gamma\,\Delta_\xi + \vec{a}(\xi)\cdot\nabla u + a_0(\xi)\,u,\\
			u(\xi,0) = f(\xi),\;
			u_t(\xi,0) = g(\xi).
		\end{cases}
	\end{equation}
	Here $\gamma > 0$, $\vec{b}_0 \in \R^k$, $\vec{a} \in \cC_\rmb(\R^k,\C^k)$,
	$a_0 \in \cC_\rmb(\R^k,\C)$, $b \in \cC^1_\rmb(\R^k,\C^k)$, $f \in H^1(\R^k)$, $g \in L^2(\R^k)$.	We can rewrite \eqref{8.3.0} in the form \eqref{AB}, where $A:H^2(\RR^k)\to L^2(\R^k)$ and $B:\{u \in L^2(\R^k):\, \vec{b}_0\cdot\nabla u \in L^2(\R^k)\}\to L^2(\RR^k)$ defined by
	\begin{equation}\label{8.3.1}
		A=\gamma\,\Delta_\xi + \vec{a}(\cdot)\cdot\nabla_\xi u + a_0(\cdot),\;B=\vec{b}_0\cdot\nabla_\xi+b(\cdot).
	\end{equation}
	We now show that this model satisfies hypothesis (HAB). We define $A_0, W_0: H^1(\R^k) \subseteq L^2(\R^k) \to L^2(\R^k)$,
	$B_0: \{u \in L^2(\R^k):\, \vec{b}_0\cdot\nabla_\xi u \in L^2(\R^k)\} \subseteq L^2(\R^k) \to L^2(\R^k)$,
	$B_1: L^2(\R^k) \to L^2(\R^k)$ by
	\begin{equation}\label{8.3.2}
		A_0 = \mathcal{F}^{-1} M_{if_0} \mathcal{F},\;W_0 = \vec{a}(\cdot)\cdot\nabla u + a_0(\cdot),B_0 = \vec{b}_0\cdot\nabla_\xi u,\; B_1 = M_{b(\cdot)},
	\end{equation}
	where $f_0:\R^k\to\R$ is the function defined by $f_0(\eta) = \big((\vec{b}_0\cdot\eta)^2 + \gamma\,|\eta|^2\big)^{1/2}$.
	Taking Fourier transform, we can immediately see that $B_0 = \mathcal{F}^{-1} M_{g_0} \mathcal{F}$,
	where $g_0:\R^k \to \R$ is defined by $g_0(\eta) = \vec{b}_0\cdot\eta$. Since $f_0, g_0$ are real-valued functions,
	one can readily check that $A_0, B_0$, $A_0^\pm = \pm A_0 - B_0$ generate $C^0$-groups on
	$L^2(\R^k)$. Moreover, since
	\begin{equation}\label{8.3.3}
		|g_0(\eta)| \leq |\vec{b}_0|\,|\eta| \leq \frac{|\vec{b}_0|}{\sqrt{\gamma}}\,f_0(\eta)
		\;\text{for any } \eta\in\R,
	\end{equation}
	we infer that $\dom(A_0) \subseteq \dom(B_0)$, $B_0(\dom(A_0^2)) \subseteq \dom(A_0)$.
	Clearly $A_0$ and $B_0$ commute and
	\begin{equation}\label{8.3.4}
		A_0^2-B_0^2 = \mathcal{F}^{-1} M_{-f_0^2 + g_0^2} \mathcal{F}
		= \mathcal{F}^{-1} M_{-\gamma|\cdot|^2} \mathcal{F}
		= \gamma\,\Delta_\xi,
	\end{equation}
	hence $A = A_0^2-B_0^2+ W_0$. Defining $\widetilde{B}_1$ to be the operator of multiplication
	by $\vec{b}_0\cdot\nabla_\xi b(\cdot)$ we have $B_0 B_1 u - B_1 B_0 u = \widetilde{B}_1 u$ for any
	$u \in H^2(\R^k) = \dom(A_0^2)$. Since $\vec{a} \in \cC_\rmb(\RR^k,\CC^k)$, $a_0 \in \cC_\rmb(\RR^k,\CC)$, $b \in \cC^1_\rmb (\RR^k,\CC)$,
	we obtain $W_0 \in \mathcal{B}(H^1(\R^k), L^2(\R^k))$,
	$B_1, \widetilde{B}_1 \in \mathcal{B}(L^2(\R^k))$,
	$(B_1)\big|_{H^1(\R^k)} \in \mathcal{B}(H^1(\R^k))$. Hence,  Hypothesis (HAB) is
	satisfied. 
\end{example}	
We note that in the model \eqref{8.3.0} above we can replace the terms $\vec{b}_0\cdot\nabla_\xi u_t$
and $\gamma\,\Delta_\xi$, respectively, by two Fourier multipliers
$\mathcal{F}^{-1} M_{ig_0} \mathcal{F}$ and $\mathcal{F}^{-1} M_{-f_0^2} \mathcal{F}$,
where $f_0, g_0: \R^k \to \R$ are locally integrable functions  
such that $|g_0| \leq \tilde{c}\,|f_0|$ for some $\tilde{c} > 0$.	
\begin{example}\label{e8.4}
	An example analogous to \eqref{8.3.0} is its 1-dimensional version,
	\begin{equation}\label{eq:11}
		\begin{cases}
			\partial_t^2 u + 2\,\partial_t\partial_\xi u + 2\,b(\xi)\,\partial_t u
			= \gamma\,\partial_\xi^2 u + a_1(x)\,\partial_\xi u + a_0(\xi)\,u,\\
			u(\xi,0) = f(\xi),\; u_t(\xi,0) = g(\xi).
		\end{cases}
	\end{equation}	
	Here $a_1, a_0 \in \cC_\rmb(\RR,\CC)$, $b \in \cC^1_\rmb(\RR,\CC)$, $\gamma > -1$,
	$f \in W^{1,p}(\R)$, $g \in L^p(\R)$. This model is a special case of \eqref{AB} that satisfies hypothesis (H2). Let $A: W^{2,p}(\R) \subseteq L^p(\R) \to L^p(\R)$,
	$B, B_0: W^{1,p} \to L^p(\R)$, $\widetilde{B}_1, E_0, E_1: L^p(\R) \to L^p(\R)$ by
	\begin{equation}\label{8.4.1}
		A = \gamma\,\partial_\xi^2 + a_1(\cdot)\,\partial_\xi + a_0(\cdot),\;
		B = \partial_\xi + b(\cdot),\;
		B_0 = \partial_\xi.
	\end{equation}
	$\widetilde{B}_1, E_0, E_1$ are the operators of multiplication by $b'$, $a_0$, $a_1$,
	respectively. Since $B_0$ generates the translation group on $L^p(\R)$, we immediately see
	that the conditions of hypothesis (H2) hold true.
	We note that this 1-dimensional model allows for the parameter $\gamma$ to be negative.	
\end{example}

\begin{example}\label{e8.5}
	Another example that satisfies Hypothesis (HAB) is the linear model
	\begin{equation}\label{8.5.0}
		\begin{cases}
			\partial_t^2 u + 2\,\vec{b}_0\cdot\nabla_\xi\,\partial_t u + 2\,b(\xi)\,\partial_t u
			= -\gamma\,\Delta_\xi^2 u + \nabla_\xi\cdot(\mathbf{a}(\xi)\nabla u) + \vec{a}(\xi)\cdot\nabla u
			+ a_0(\xi)\,u,\\
			u(\xi,0) = f(\xi),\;
			u_t(\xi,0) = g(\xi).
		\end{cases}
	\end{equation}
	Here $\gamma > 0$, $\vec{b}_0 \in \R^k$, $\mathbf{a} \in \cC_\rmb(\R^k\C^{k\times k})$,
	$\vec{a} \in \cC_\rmb(\R^k,\C^k)$, $a_0 \in \cC_\rmb(\R^k,\C)$, $b \in \cC^2_\rmb(\R^k,\C)$,
	$f\in H^2(\R^k)$, $g \in L^2(\R^k)$.
	The model \eqref{8.5.0} fits the general framework \eqref{AB}. In this specific case $A: H^4(\R^k) \subseteq L^2(\R^k) \to L^2(\R^k)$,
	$B: \{u \in L^2(\R^k):\, \vec{b}_0\cdot\nabla_\xi u \in L^2(\R^k)\}\subseteq L^2(\R^k) \to L^2(\R^k)$, are defined by 
	\begin{equation}\label{8.5.1}
		A= -\gamma\,\Delta_\xi^2 + \nabla_\xi\cdot(\mathbf{a}(\cdot)\nabla_\xi) + \vec{a}(\cdot)\cdot\nabla_\xi + a_0(\cdot),\;B=\vec{b}_0\cdot\nabla_\xi+b(\cdot).
	\end{equation}
	Next, we introduce
	$A_0, W_0: H^2(\R^k)\subseteq L^2(\R^k)\to L^2(\R^k)$, $B_0: \{u \in L^2(\R^k):\, \vec{b}_0\cdot\nabla_\xi u \in L^2(\R^k)\}\subseteq L^2(\R^k) \to L^2(\R^k)$ and $B_1: L^2(\RR^k)\to L^2(\RR^k)$ by
	\begin{equation}\label{8.5.2}
		A_0= i\sqrt{\gamma}\,\Delta_\xi,\;W_0 = \nabla_\xi\cdot(\mathbf{a}(\cdot)\nabla_\xi) + \vec{a}(\cdot)\cdot\nabla_\xi + a_0(\cdot),\;B_0 = \vec{b}_0\cdot\nabla_\xi,\; B_1 = \vec{b}_0\cdot\nabla_\xi b(\cdot)
	\end{equation}
	Taking Fourier transform we obtain that
	$A_0 = \mathcal{F}^{-1} M_{if_0} \mathcal{F}$,
	$B_0 = \mathcal{F}^{-1} M_{ig_0} \mathcal{F}$,
	where $f_0, g_0: \R^k \to \R$ are defined by
	\begin{equation}\label{8.5.3}
		f_0(\eta) = \sqrt{\gamma}\,|\eta|^2,\qquad g_0(\eta) = \vec{b}_0\cdot\eta.
	\end{equation}
	It follows that $A_0$ and $B_0$ commute. Moreover, since $f_0$ and $g_0$ are real-valued
	we immediately infer that $A_0, B_0, A_0^\pm:= \pm A_0 - B_0$ generate $C^0$-groups.
	One can readily check that
	$\dom(A_0) \subseteq \dom(B_0^2)$. The conditions
	$(B_1)\big|_{H^2(\R^k)} \in \mathcal{B}(H^2(\R^k))$
	and $B_1 \in \mathcal{B}(L^2(\R^k))$ follow from the assumption
	$b \in \cC^2_\rmb(\R^k,\C)$. Finally, $B_1 \in \mathcal{B}(L^2(\R^k))$ and
	by setting $\widetilde{B}_1$ to be the operator of multiplication
	by $\vec{b}_0\cdot\nabla_\xi b(\cdot)$, $B_0 B_1 u - B_1 B_0 u = \widetilde{B}_1 u$ for any $u \in H^2(\R^k)$
	arguing as in Example~\ref{l8.3}.
\end{example} 
\begin{example}\label{e8.6}
	In \cite{MS} the authors study a damped fractional Klein--Gordon
	equation that is a special case of \eqref{AB}. The system reads as follows:
	\begin{equation}\label{8.6.0}
		\begin{cases}
			\partial_t^2 u + b(\xi)\,\partial_t u = -(-\partial_\xi^2)^\alpha u - m\,u,\\
			u(\xi,0) = f(\xi),\; 
			u_t(\xi,0) = g(\xi), 
		\end{cases}
	\end{equation}
	Here $\alpha \in (0,1]$, $m > 0$, $f \in H^{2\alpha}(\R)$,
	$g \in H^\alpha(\R)$. The linear operators $A:H^{2\alpha}(\RR)\subseteq L^2(\RR)\to L^2(\RR)$ are defined by $A=-(-\partial_\xi^2)^\alpha$ and $M=M_{b(\cdot)}$, the operator of multiplication by $b\in\cC_\rmb(\RR)$.  System \eqref{8.6.0} satisfies hypothesis (H1). Indeed, one can readily check that $A_0:H^{\alpha}(\RR)\subseteq L^2(\RR)\to L^2(\RR)$ defined by $A_0 = i(-\partial_\xi^2)^{\alpha/2}$is skew-adjoint on $L^2(\RR)$, hence it is a group generator on $L^2(\R)$. Since $A=A_0^2$ we immediately see that hypothesis (H1) holds in this case.
\end{example}
Next, we look at two examples of coupled systems of damped wave equations.
\begin{example}\label{e8.7}
	Consider the system
	\begin{equation}\label{8.7.0}
		\begin{cases}
			\partial_t^2 u + 2\beta_1\,\partial_\xi\,\partial_t v
			= \gamma\,\partial_\xi^2 u + a_{11}(\xi)\,\partial_\xi u + a_{12}(\xi)\,\partial_\xi v
			+ b_{11}(\xi)\,u + b_{12}(\xi)\,v,\\
			\partial_t^2 v + 2\beta_2\,\partial_\xi\,\partial_t u
			= \gamma\,\partial_\xi^2 v + a_{21}(\xi)\,\partial_\xi u + a_{22}(\xi)\,\partial_\xi v
			+ b_{21}(\xi)\,u + b_{22}(\xi)\,v,\\
			u(\xi,0) = f_1(\xi),\; v(\xi,0) = f_2(\xi),
			u_t(\xi,0) = g_1(\xi),\; v_t(\xi,0) = g_2(\xi).
		\end{cases}
	\end{equation}
	Here $\gamma > 0$, $\beta_1, \beta_2 \in \R$, $\beta_1\beta_2 > 0$,
	$a_{jk}(\cdot), b_{jk}(\cdot) \in \cC_\rmb(\RR,\CC)$,
	$f_1, f_2 \in H^1(\R)$, $g_1, g_2 \in L^2(\R)$.
	The system \eqref{8.6.0} could be written as
	\begin{equation}\label{8.7.1}
		\partial_t^2\begin{pmatrix}u\\v\end{pmatrix}
		+ 2\begin{bmatrix}0 & \beta_1\partial_\xi\\\beta_2\partial_\xi & 0\end{bmatrix}
		\partial_t\begin{pmatrix}u\\v\end{pmatrix}
		= \begin{bmatrix}\gamma\partial_\xi^2 & 0\\ 0 & \gamma\partial_\xi^2\end{bmatrix}
		\begin{pmatrix}u\\v\end{pmatrix}
		+ \mathbf{a}(\xi)\,\partial_\xi\begin{pmatrix}u\\v\end{pmatrix}
		+ \mathbf{b}(\xi)\begin{pmatrix}u\\v\end{pmatrix},
	\end{equation}
	where $\mathbf{a} = (a_{jk}),
	\mathbf{b} = (b_{jk}) \in \mathcal{C}_\rmb(\R,\C^{2\times2})$.
	Equation \eqref{8.7.1} is a special case of \eqref{AB}.
	Here $A: H^2(\R,\C^2) \subseteq L^2(\R,\C^2) \to L^2(\R,\C^2)$ and $B: H^1(\R,\C^2) \subseteq L^2(\R,\C^2) \to L^2(\R,\C^2)$,
	are defined by
	\begin{equation}\label{8.7.2}
		B = \begin{bmatrix}0 & \beta_1\partial_\xi\\\beta_2\partial_\xi & 0\end{bmatrix},\qquad
		A = \begin{bmatrix}\gamma\partial_\xi^2 & 0\\ 0 & \gamma\partial_\xi^2\end{bmatrix}
		+ \mathbf{a}(\cdot)\,\partial_\xi + \mathbf{b}(\cdot).
	\end{equation}
	Set $B_0=B$ and let $A_0, W_0: H^1(\R,\C^2) \subseteq L^2(\R,\C^2) \to L^2(\R,\C^2)$ be the linear operators defined by
	\begin{equation}\label{8.7.3}
		A_0 = \sqrt{\gamma+\beta_1\beta_2}\,\partial_\xi I_2,\;
		W_0 = \mathbf{a}(\cdot)\,\partial_\xi+ \mathbf{b}(\cdot).
	\end{equation}
	Taking Fourier transform we obtain 
	\begin{equation}\label{8.7.4}
		A_0 = \mathcal{F}^{-1} M_{\Gamma_0} \mathcal{F}\quad\text{and}\quad
		B_0 = \mathcal{F}^{-1} M_{\Lambda_0} \mathcal{F},
	\end{equation}
	where $\Gamma_0, \Lambda_0: \R \to \C^2$ are the matrix-valued functions defined by
	\begin{equation}\label{8.7.5}
		\Gamma_0(\eta) = \begin{bmatrix}i\sqrt{\gamma+\beta_1\beta_2}\,\eta & 0\\ 0 & i\sqrt{\gamma+\beta_1\beta_2}\,\eta\end{bmatrix},\qquad
		\Lambda_0(\xi) = \begin{bmatrix}0 & i\beta_1\eta\\ i\beta_2\eta & 0\end{bmatrix}.
	\end{equation}
	We note that $\Gamma_0(\eta)$ and $\Lambda_0(\eta)$ are skew-adjoint for any $\eta\in\R$,
	which implies that $M_{\Gamma_0}$, $M_{\Lambda_0}$ are skew-adjoint operators on
	$L^2(\R,\C^2)$. We infer that $A_0, B_0, A_0^\pm := \pm A_0 - B_0$ generate
	$C_0$-groups on $L^2(\R,\C^2)$. Moreover, since
	$\Gamma_0(\eta)\Lambda_0(\eta) = \Lambda_0(\eta)\Gamma_0(\eta)$ for any $\eta\in\R$,
	we have $A_0 B_0 = B_0 A_0$.
	In addition, one can readily check that $\dom(A_0) = \dom(B_0)$,
	$B_0(\dom(A_0^2)) = B_0(H^2(\R,\C^2)) \subseteq H^1(\R,\C^2) = \dom(A_0)$.
	We compute
	\begin{equation}\label{8.7.6}
		\bigl(\Gamma_0(\eta) + \Lambda_0(\eta)\bigr)\bigl(\Gamma_0(\eta) - \Lambda_0(\eta)\bigr)
		= -\gamma\,I_2\,\eta^2\;\mbox{for any}\;\eta\in\RR
	\end{equation}
	which implies that
	\begin{equation}\label{8.7.8}
		A_0^2-B_0^2 = -\gamma\,I_2\,\partial_\xi^2.
	\end{equation}
	From \eqref{8.7.8} we infer that $A = A_0^2-B_0^2+ W_0$, which shows that Hypothesis (HAB)
	is satisfied.
\end{example}

\begin{example}\label{e8.8}
	Another example of a linear model that satisfies Hypothesis (HAB) is given by the coupled
	system with strong damping
	\begin{equation}\label{8.8.0}
		\begin{cases}
			\begin{aligned}&\partial_t^2 u + 2\beta_1\,\Delta_\xi\,\partial_t v
				= -\gamma\,\Delta_\xi^2 u + \nabla_\xi\cdot(\mathbf{a}_{11}(\xi)\nabla_\xi u)
				+ \nabla_\xi\cdot(\mathbf{a}_{12}(\xi)\nabla_\xi v)
				+ \vec{b}_{11}(\xi)\cdot\nabla_\xi u + \vec{b}_{12}(\xi)\cdot\nabla_\xi v\\
				&\qquad\qquad\qquad\qquad\qquad+ c_{11}(\xi)\,u + c_{12}(\xi)\,v,\\
				&\partial_t^2 v + 2\beta_2\,\Delta_\xi\,\partial_t u
				= -\gamma\,\Delta_\xi^2 u + \nabla_\xi\cdot(\mathbf{a}_{21}(\xi)\nabla_\xi u)
				+ \nabla_\xi\cdot(\mathbf{a}_{22}(\xi)\nabla_\xi v)
				+ \vec{b}_{21}(\xi)\cdot\nabla_\xi u + \vec{b}_{22}(\xi)\cdot\nabla_\xi v\\
				&\qquad\qquad\qquad\qquad\qquad+ c_{21}(\xi)\,u + c_{22}(\xi)\,v,\\
				&u(\xi,0) = f_1(\xi),\; v(\xi,0) = f_2(\xi),\;
				u_t(\xi,0) = g_1(\xi),\; v_t(\xi,0) = g_1(\xi).
			\end{aligned}
		\end{cases}
	\end{equation}
	Here $\gamma > 0$, $\beta_1,\beta_2\in\R$, $\beta_1\beta_2 < 0$,
	$\mathbf{a}_{jk}(\cdot) \in \cC_\rmb(\R,\C^{k\times k})$,
	$\vec{b}_{jk}(\cdot) \in \cC_\rmb(\R,\C^k)$, $c_{jk}(\cdot) \in \cC_\rmb(\RR,\CC)$
	for $j,k = 1,2$. $f_1, f_2 \in H^2(\R^k)$, $g_1, g_2 \in L^2(\R^k)$.
	The system \eqref{8.8.0} can be written as
	\begin{equation}\label{8.8.1}
		\partial_t^2\begin{pmatrix}u\\v\end{pmatrix}
		+ 2\begin{bmatrix}0 & \beta_1\Delta_\xi\\\beta_2\Delta_\xi & 0\end{bmatrix}
		\partial_t\begin{pmatrix}u\\v\end{pmatrix}
		= -\gamma I_2\,\Delta_\xi^2\begin{pmatrix}u\\v\end{pmatrix}
		+ W_0\begin{pmatrix}u\\v\end{pmatrix},
	\end{equation}
	where
	\begin{equation}\label{8.8.2}
		W_0\begin{pmatrix}u\\v\end{pmatrix}
		= \begin{pmatrix}
			\nabla_\xi\cdot(\mathbf{a}_{11}(\cdot)\nabla_\xi u) + \nabla_\xi\cdot(\mathbf{a}_{12}(\cdot)\nabla_\xi v)
			+ \vec{b}_{11}(\cdot)\cdot\nabla_\xi u + \vec{b}_{12}(\cdot)\cdot\nabla_\xi v\\
			\nabla_\xi\cdot(\mathbf{a}_{21}(\cdot)\nabla_\xi u) + \nabla_\xi\cdot(\mathbf{a}_{22}(\cdot)\nabla_\xi v)
			+ \vec{b}_{21}(\cdot)\cdot\nabla_\xi u + \vec{b}_{22}(\cdot)\cdot\nabla_\xi v
		\end{pmatrix}
		+ \begin{bmatrix}c_{11}(\cdot) & c_{12}(\cdot)\\ c_{21}(\cdot) & c_{22}(\cdot)\end{bmatrix}
		\begin{pmatrix}u\\v\end{pmatrix}.
	\end{equation}
	Here $W_0$ can be considered as a bounded linear operator from $H^2(\R,\C^2)$ to $L^2(\R,\C^2)$. Let $A: H^4(\R^k,\C^2) \subseteq L^2(\R^k,\C^2) \to L^2(\R^k,\C^2)$, $B: H^2(\R^k,\C^2) \subseteq L^2(\R^k,\C^2) \to L^2(\R^k,\C^2)$,
	and
	$A_0: H^2(\R^k,\C^2) \to L^2(\R^k,\C^2)$ be the linear operators defined by 
	\begin{equation}\label{8.8.3}
		A = -\gamma\,I_2\,\Delta_\xi^2 + W_0,\;B= \begin{bmatrix}0 & \beta_1\Delta_\xi\\\beta_2\Delta_\xi & 0\end{bmatrix},\;A_0 = i\sqrt{\gamma - \beta_1\beta_2}\,I_2\,\Delta_\xi.
	\end{equation}
	Set $B_0=B$. Since $\Delta_\xi$ is self-adjoint, we infer that $A_0, B_0$ are skew-adjoint, thus
	$A_0, B_0, A_0^\pm = \pm A_0 - B_0$
	generate $C^0$-groups on $L^2(\R^k,\C^2)$. Moreover, one can readily check that $A_0$ and
	$B_0$ commute and $\dom(A_0) = \dom(B_0) = H^2(\R^k,\C^2)$ and
	$B_0(\dom(A_0^2)) = B_0(H^4(\R^k,\C^2)) \subseteq H^2(\R^k,\C^2) = \dom(A_0)$. Since $B_0^2=\beta_1\beta_2\,I_2\,\Delta_\xi^2$, from \eqref{8.8.3} we immediately infer that $A=A_0^2-B_0^2+W_0$
	Hence, hypothesis (HAB) is satisfied.
\end{example}
\begin{example}\label{e8.9} We consider the sixth-order linear model,
	\begin{equation}\label{8.9.0}
		\begin{cases}
			\partial_t^2 u + 2\alpha\,\partial_t\partial_\xi u + 2\beta\,\partial_t\partial_\xi^3 u
			= \gamma\,\partial_\xi^6 u + a_3(\xi)\,\partial_\xi^3 u+a_2(\xi)\,\partial_\xi^2 u+a_1(\xi)\,\partial_\xi u + a_0(\xi)\,u,\\
			u(\xi,0) = f(\xi),\; u_t(\xi,0) = g(\xi).
		\end{cases}
	\end{equation}	
	Here $\gamma>0$, $\alpha,\beta\in\RR$, $(\alpha,\beta)\ne(0,0)$, $a_3,a_2,a_1, a_0 \in \cC_\rmb(\RR,\CC)$, $\gamma > 0$,
	$f \in H^3(\R)$, $g \in L^2(\R)$. This linear model is a special case of \eqref{AB}. In this case $A:H^6(\RR)\subseteq L^2(\RR)\to L^2(\RR)$, $B:H^m(\RR)\subseteq L^2(\RR)\to L^2(\RR)$ are defined by 
	\begin{equation}\label{8.9.1}
		A=\gamma\,\partial_\xi^6 + a_3(\cdot)\,\partial_\xi^3 +a_2(\cdot)\,\partial_\xi^2 +a_1(\cdot)\,\partial_\xi  + a_0(\cdot),\;B=\alpha\,\partial_\xi u+\beta\,\partial_\xi^3,
	\end{equation}	
	where the Sobolev parameter is given by $m=3$ if $\beta\ne0$ and $m=1$ if $\beta=0$. We also introduce the operators $A_0, W_0:H^3(\RR)\subseteq L^2(\RR)\to L^2(\RR)$ defined by 
	
	\begin{equation}\label{8.9.2}
		A_0 = \mathcal{F}^{-1} M_{if_0} \mathcal{F},\;W_0=a_3(\cdot)\,\partial_\xi^3 +a_2(\cdot)\,\partial_\xi^2 +a_1(\cdot)\,\partial_\xi  + a_0(\cdot).
	\end{equation}
	Here $f_0:\R\to\R$ is the function defined by $f_0(\eta) =\sqrt{\gamma+\beta^2\eta^6-2\alpha\beta\eta^4+\alpha^2\eta^2}$. Since $\gamma>0$ we have that $f_0$ is real valued as the quantity under the square root is always non-negative. Also, we set $B_0=B$. Clearly, $A_0$ and $B_0$ commute. Taking Fourier Transform one can readily check that $A_0^\pm=\pm A_0-B_0$ generate $C_0$-groups on $L^2(\RR)$. Clearly, $\dom(A_0)\subseteq\dom(B_0)$ and $B_0(\dom(A_0^2)) \subseteq \dom(A_0)$. Moreover, the linear operator  $W_0$ is bounded. Since $A=A_0^2-B_0^2+W_0$ we conclude that Hypothesis (HAB) is satisfied. We note that \eqref{8.9.0} satisfies Hypothesis (H3), whenever $\beta=0$ and Hypothesis (H2) if $\alpha=0$, $a_1\equiv 0$ and $a_2\equiv0$.
\end{example}
\begin{remark}\label{r8.10}
	Many nonlinear second order PDEs with damping can be written in the form
	\begin{equation}\label{nonlinear-damping}	
		\partial_t^2h+2B_{\mathrm{unp}}\partial_th+F_1(h,\partial_th)\partial_th=A_{\mathrm{unp}} h+F_2(h).	
	\end{equation}	
	Here $F_1\in\cC^2(\RR^2,\CC)$ and $F_2\in\cC^2(\RR,\CC)$, $A_{\mathrm{unp}}$ and $B_{\mathrm{unp}}$ are (unperturbed) differential expressions on $\RR^k$ that can define differential operators on some Banach spaces $\bX_A\hookrightarrow\bX$ and $\bX_B\hookrightarrow\bX$. Here the orders of the operators are assumed to satisfy $\mathrm{ord}(A_{\mathrm{unp}})\geq 2\mathrm{ord}(B_{\mathrm{unp}})$. A \emph{standing wave} solution is a time independent solution $\bar h$ of \eqref{nonlinear-damping}. Hence, $\bar h$ solves the equation 
	\begin{equation}\label{standing-wave}	
		A_{\mathrm{unp}} \bar h+F_2(\bar h)=0.	
	\end{equation}	
	Here $\bar h\in\cC^m(\RR^k,\CC)$, where $m:=\mathrm{ord}(A_{\mathrm{unp}})\geq2$, whenever $B_{\mathrm{unp}}$ defines an unbounded linear operator on $\bX$ with domain $\bX_B$. Linearizing \eqref{nonlinear-damping} along $\bar h$ we obtain the linear model
	\begin{equation}\label{linearized-damping}	
		\partial_t^2u+2B_{\mathrm{unp}}\partial_tu+F_1(\bar h(x),0)\partial_tu=A_{\mathrm{unp}} u+F_2'(\bar h(x))u.	
	\end{equation}	
	This linear model is in the general form \eqref{AB}, where $A:\bX_A\subseteq\bX\to\bX$	and $B:\bX_B\subseteq\bX\to\bX$	are given by
	\begin{equation}\label{linearized-AB}
		A=A_{\mathrm{unp}}+F_2'(\bar h(\cdot)),\;B=B_{\mathrm{unp}}+\tfrac{1}{2}F_1(\bar h(\cdot),0).	
	\end{equation}	
	Since $\bar h\in\cC(\R^k,\CC)$ with $m\geq 2$, we infer that in the case second order equations with damping obtained by linearizing along a localized pattern the operators $A$ and $B$ have \emph{smooth coefficients}. Moreover, in the case of localized patterns, the function $\bar h$ is a bounded functions, hence the coefficients of the differential operators $A$ and $B$ are also bounded functions.
\end{remark}		
	
\appendix
\section{Auxiliary Lemmas}\label{A1}
In this appendix we gather several lemmas needed in this paper. We start with a lemma used in the proof of existence of classical solution of \eqref{AB}.
\begin{lemma}\label{App1}
	Assume $B : \dom(B) \subseteq\bX \to\bX$ is a closed linear operator, $\rho(B) \neq \emptyset$,
	$f \in \mathcal{C}^1(\mathbb{R}, \bX)$, $f(t), f'(t) \in \dom(B)$ for any $t \in \mathbb{R}$, $Bf(\cdot), Bf'(\cdot) \in \mathcal{C}(\mathbb{R},\bX)$.
	Then,
	\begin{equation}\label{App1.1}
		Bf(\cdot) \in \mathcal{C}^1(\mathbb{R},\bX) \quad \text{and} \quad (Bf(\cdot))' = Bf'(\cdot). 
	\end{equation}
\end{lemma}
\begin{proof}
	Fix $\mu \in \rho(B)$. Define $F : \mathbb{R} \to\bX$ by
	$F(t) = \int_0^t Bf'(s)\,\rmd s$.
	Since $Bf'(\cdot) \in \mathcal{C}(\mathbb{R},\bX)$ one can readily check that $F \in \mathcal{C}^1(\mathbb{R},\bX)$ and $F' = Bf'(\cdot)$. Since $R(\mu, B)\,Bf(t) = B\,R(\mu,B)f(t)$ for any $t\in\RR$, $BR(\mu,B) \in \mathcal{B}(\bX)$ and $f \in \mathcal{C}^1(\mathbb{R},\bX)$ we obtain $R(\mu,B)Bf(\cdot) \in \mathcal{C}^1(\mathbb{R},\bX)$
	and
	\begin{equation}\label{App1.2}
		\bigl(R(\mu,B)Bf(\cdot)\bigr)' = BR(\mu,B)f'(\cdot) = R(\mu,B)Bf'(\cdot). 
	\end{equation}
	Since $F, R(\mu,B)Bf(\cdot) \in \mathcal{C}^1(\mathbb{R},\bX)$ and $R(\mu,B) \in \mathcal{B}(\bX)$
	it follows that $R(\mu,B)Bf(\cdot) - R(\mu,B)F(\cdot) \in \mathcal{C}^1(\mathbb{R},\bX)$. Moreover, 
	\begin{equation}\label{App1.3}
		\bigl(R(\mu,B)Bf(\cdot) - R(\mu,B)F(\cdot)\bigr)' = R(\mu,B)Bf'(\cdot) - R(\mu,B)Bf'(\cdot) = 0. 
	\end{equation}
	Hence, there exists $x_0 \in\bX$ such that
	$R(\mu,B)Bf(t) - R(\mu,B)F(t) = x_0$ for any $t \in \mathbb{R}$. This identity proves that $x_0\in\dom(B)$. Therefore, there exists $x_1 \in\bX$ such that
	$x_0 = R(\mu, B)x_1$. We conclude that $Bf(\cdot) = F(\cdot) + x_1$. Since $F\in\mathcal{C}^1(\mathbb{R},\bX)$ we infer that $ Bf(\cdot) \in \mathcal{C}^1(\mathbb{R},\bX)$
	and $(Bf(\cdot))' = F'(\cdot) = Bf'(\cdot)$, proving the lemma.
\end{proof}
\begin{lemma}\label{App2}
	Let $\varphi \in \mathcal{C}^1(\mathbb{R}, \mathbb{C})$ and $\{H(t)\}_{t \in \mathbb{R}}$ be a $C^0$ group with generator $\mathcal{J}$ on a Banach space $\bX$. Then,
	\begin{enumerate}
		\item[(i)] $\displaystyle \int_0^t \varphi(t-s) H(s) x\, \rmd s \in \operatorname{dom}(\mathcal{J})$
		\item[(ii)] $\displaystyle \mathcal{J} \int_0^t \varphi(t-s) H(s) x\, \rmd s = \int_0^t \varphi'(t-s) H(s) x\, \rmd s + \varphi(0) H(t) x - \varphi(t) x$
	\end{enumerate}
	for any $t \in \mathbb{R}$, $x \in \bX$.
\end{lemma}
\begin{proof}
	First, assume that $x \in \operatorname{dom}(\mathcal{J})$. Then,
	$H(\cdot) x \in C^1(\mathbb{R}, \bX) \cap C^0(\mathbb{R}, \operatorname{dom}(\mathcal{J}))$. Since $\mathcal{J}$ is closed, we obtain $\displaystyle \int_0^t \varphi(t-s) H(s) x\, \rmd s \in \operatorname{dom}(\mathcal{J})$ and
	\begin{align}\label{App2.1}
		\mathcal{J} \int_0^t &\varphi(t-s) H(s) x\, \rmd s = \int_0^t \varphi(t-s) \mathcal{J} H(s) x\, \rmd s = \int_0^t \varphi(t-s) (H(s) x)' \, \rmd s \nonumber\\
		&= \varphi(t-s) H(s) x \Big|_{s=0}^{s=t} + \int_0^t \varphi'(t-s) H(s) x\, \rmd s= \varphi(0) H(t) x - \varphi(t) x + \int_0^t \varphi'(t-s) H(s) x\, \rmd s,
	\end{align}
	for any $t \in \mathbb{R}$. Thus, assertions (i) and (ii) hold In the case when $x \in \operatorname{dom}(\mathcal{J})$. Fix $x \in \bX$ and let $\{x_n\}_{n \geq 1}$ be a sequence of elements of $\operatorname{dom}(\mathcal{J})$ such that $x_n \rightarrow x$ as $n\to\infty$. Then,
	\begin{align}\label{App2.2}
		&\left\| \int_0^t \varphi^{(k)}(t-s) H(s) x_n\, \rmd s -\int_0^t \varphi^{(k)}(t-s) H(s) x\, \rmd s \right\| \leq  \int_0^t |\varphi^{(k)}(t-s)| \, \|H(s)(x_n - x)\| \, \rmd s \nonumber\\ &\qquad\qquad\qquad\leq \int_0^t |\varphi^{(k)}(t-s)| e^{\omega s} \, \rmd s \, \|x_n - x\|\;\mbox{for any}\; t\in\RR,\,n\in\NN,\,k=0,1.
	\end{align}
	It follows that
	\begin{equation}\label{App2.3}
		\int_0^t \varphi^{(k)}(t-s) H(s) x_n\, \rmd s \xrightarrow[n \to \infty]{\quad \bX \quad} \int_0^t \varphi^{(k)}(t-s) H(s) x\, \rmd s \;\mbox{for any}\; t\in\RR,\,k=0,1.
	\end{equation}
	From \eqref{App2.3} and since $H(t) \in \mathcal{B}(\bX)$ for any $t \in \mathbb{R}$, we infer that
	\begin{equation}\label{App2.4}
		\int_0^t \varphi(t-s) H(s) x_n\, \rmd s \xrightarrow[n \to \infty]{\quad \bX \quad} \int_0^t \varphi(t-s) H(s) x\, \rmd s \;\mbox{for any}\; t\in\RR;
	\end{equation}
	\begin{equation}\label{App2.5}
		\mathcal{J} \int_0^t \varphi(t-s) H(s) x_n\, \rmd s \xrightarrow[n \to \infty]{\quad \bX \quad} \int_0^t \varphi(t-s) H(s) x\, \rmd s + \varphi(0) H(t) x - \varphi(t) x\;\mbox{for any}\; t\in\RR.
	\end{equation}
	Using again that $\mathcal{J}$ is a closed linear operator, from \eqref{App2.4} and \eqref{App2.5} we infer that (i) and (ii) hold, proving the lemma.
\end{proof}
\begin{lemma}\label{App3}
	Assume $\bX_1, \bX_2$ are Banach spaces, $K: \mathbb{R} \to \mathcal{B}(\bX_1, \bX_2)$ is such that $K(\cdot)\,x_1 \in \mathcal{C}^1(\mathbb{R}, \bX_2)$ for any $x_1 \in \bX_1$. If $g \in \mathcal{C}(\mathbb{R}, \bX_1)$ then
	\begin{equation}\label{K-conv}
		K * g \in \mathcal{C}^1(\mathbb{R}, \bX_2) \;\text{and} \quad (K * g)' = K' * g + K(0)\,g,
	\end{equation}
	where  $K': \mathbb{R} \to \mathcal{B}(\bX_1, \bX_2)$ is the strongly continuous operator-valued function defined by $K'(t)\,x = (K(t)\,x)'$.	
\end{lemma}	
\begin{proof}
	Since $K'$ is strongly continuous and $g \in L^1_{\mathrm{loc}}(\mathbb{R}, \bX_1)$, \cite[Proposition 1.3.4]{ABHN} we infer $K' * g \in \mathcal{C}(\mathbb{R}, \bX_2)$. Moreover,
	\begin{align*}
		\int_0^t (K' * g)(s)\,\rmd s
		&= \int_0^t \int_0^s K'(s-\tau)\,g(\tau)\,\rmd\tau\,\rmd s = \int_0^t \int_\tau^t K'(s-\tau)\,g(\tau)\,\rmd s\,\rmd\tau\\
		&= \int_0^t K(s-\tau)\Big|_{s=\tau}^{s=t} \,g(\tau)\rmd\tau= \int_0^t K(t-\tau)\,g(\tau)\,\rmd\tau - K(0)\int_0^t g(\tau)\,\rmd\tau \\
		&= (K * g)(t) - K(0)\int_0^t g(\tau)\,\rmd\tau \; \text{for any}\; t \in \mathbb{R}.
	\end{align*}
	We immediately conclude that $K * g \in \mathcal{C}^1(\mathbb{R}, \bX_2)$ and $(K * g)' = K' * g + K(0)\,g$.
\end{proof}

\end{document}